\documentclass{article}
\usepackage[numbers,sort&compress]{natbib}

\usepackage{amscd}

\usepackage{mathtools}
\usepackage{authblk}
\usepackage[title]{appendix}
\usepackage{geometry}
\usepackage{subfiles}
\usepackage{caption}
\usepackage{subcaption}
\usepackage{setspace}
\usepackage{amsthm}
\usepackage{amsmath}
\usepackage{amssymb}
\usepackage{ushort}
\usepackage{float}
\usepackage{fancyhdr}
\usepackage{graphicx}
\usepackage[colorlinks,linkcolor=black,anchorcolor=black,citecolor=black]{hyperref}
\usepackage{mathrsfs}
\usepackage{listings}
\usepackage{esint}
\usepackage{tikz}
\usetikzlibrary{arrows,decorations.pathmorphing,backgrounds,positioning,fit,petri,arrows.meta,bending,positioning}

\mathtoolsset{showonlyrefs}

\graphicspath{ {./images/} }
\newtheorem{corollary}{Corollary}[section]
\newtheorem{conjecture}{Conjecture}
\newtheorem{con-proposition}{``Proposition''}
\newtheorem{proposition}{Proposition}[section]
\newtheorem{theorem}{Theorem}[section]
\newtheorem{lemma}{Lemma}[section]
\newtheorem{definition}{Definition}[section]

\newtheorem{remark}{Remark}[section]
\newtheorem{assumption}{Assumption}

\numberwithin{equation}{section}

\newcommand{\p}{\partial}
\newcommand{\beq}{\begin{equation}}
\newcommand{\eeq}{\end{equation}}

\newcommand{\bconj}{\begin{conjecture}}
	\newcommand{\econj}{\end{conjecture}}

\newcommand{\bueq}{\begin{equation*}}
\newcommand{\eueq}{\end{equation*}}

\newcommand{\bthm}{\begin{theorem}}
	\newcommand{\ethm}{\end{theorem}}

\newcommand{\eps}{\varepsilon}

\DeclareMathOperator\init{init}

\providecommand{\keywords}[1]
{
	\small	
	\textbf{{Keywords:}} #1
}
\providecommand{\msc}[1]
{
	\small	
	\textbf{{Mathematics Subject Classification:}} #1
}

\newcommand{\id}{\text{id}}

\newcommand{\sign}[1]{{\rm sign }\left({#1}\right)}

\newcounter{saveenumi}

\title{
Classical solutions of a mean field system for pulse-coupled oscillators: long time asymptotics versus blowup
}

\author{Jos\'e A. Carrillo\thanks{Mathematical Institute, University of Oxford, Oxford OX2 6GG, UK (carrillo@maths.ox.ac.uk)},\, \, Xu'an Dou\thanks{Beijing International Center for Mathematical Research, Peking University, Beijing, 100871, China (dxa@pku.edu.cn)},\, \, Pierre Roux\thanks{Institut Camille Jordan, \'Ecole Centrale de Lyon, 69134 Ecully, France (pierre.roux@ec-lyon.fr)} \, \, and\, \,  Zhennan Zhou\thanks{Institute for Theoretical Sciences, Westlake University, Hangzhou, Zhejiang Province, 310030, China (zhouzhennan@westlake.edu.cn).}}

\begin{document}
\maketitle
\begin{abstract}
We introduce a novel reformulation of the mean-field system for pulse-coupled oscillators. It is based on writing a closed equation for the inverse distribution function associated to the probability density of oscillators with a given phase in a suitable time scale. This new framework allows to show a hidden contraction/expansion of certain distances leading to a full clarification of the long-time behavior, existence of steady states, rates of convergence, and finite time blow-up of classical solutions for a large class of monotone phase response functions. In the process, we get insights about the origin of obstructions to global-in-time existence and uniform in time estimates on the firing rate of the oscillators.
\end{abstract}


\

\msc{35Q92, 35B40, 35B44, 34C15}


\section{Introduction}

Systems of pulse-coupled oscillators model synchronization by singular interactions occurring as phase jumps changing in a particular pulsatile manner. They have numerous applications in science and engineering (e.g. \cite{Ramirez-Avila2019,peleg2023new,olami1992self,herz1995earthquake,rhouma2001self,hong2005scalable,pagliari2010scalable,wang2020global}). Yet their unique singular coupling mechanism renders them challenging to analyze. A partial differential equation (PDE) for pulse-coupled oscillators, which characterizes the nonlinear interaction at a continuous level, has been derived in physics literature 30 years ago by a formal mean-field argument \cite{kuramoto1991collective,1993PhysRevE.48.1483}. However, it still lacks a systematic mathematical treatment. In this work, we study the mean-field equation from a novel reformulation, which reveals several hidden structures of the equation, and facilitates the study on its basic properties, including well-posedness, blow-up and long time convergence to the steady state. 

\subsection{The particle system}

Pulse-coupled oscillators, as a particle system, describe a finite number of oscillators coupled in a pulselike manner. Here we present a model following \cite{MR0414135,mirollo1990synchronization}. Consider $\mathcal{N}$ oscillators ($\mathcal{N}\in\mathbb{N}_+$), each described by a phase variable $\phi_i(t)$, for time $t\geq 0$ and $i=1,2,...,\mathcal{N}$. The phase variable $\phi_i$ takes value in $[0,\Phi_F]$, where $\Phi_F>0$ is the firing phase. When all $\phi_i\in[0,\Phi_F)$, i.e., strictly less than $\Phi_F$, each $\phi_i$ is subject to its own dynamics, which is here assumed to be the simple ordinary differential equation (ODE)
\begin{equation}\label{pco-eq-1:no-spike}
    \frac{d}{dt}\phi_i(t)=1,\qquad i=1,2,...,\mathcal{N}.
\end{equation} An oscillator ($\phi_i$) reaching $\Phi_F$ constitutes a firing or spiking event, which has two consequences. For $\phi_i$ itself, it is immediately reset to $0$, i.e.
\begin{equation}\label{pco-eq-2:reset}
    \phi_i(t^-)=\Phi_F,\quad \Rightarrow \quad \phi_i(t)=0.
\end{equation} Also, the firing of $\phi_i$ induces an immediate phase jump in other oscillators, and the phase jump magnitude depends on the recipient's phase. More specifically, we have
\begin{equation}\label{pco-eq-3:spike-coupling}
    \phi_i(t^-)=\Phi_F,\quad \Rightarrow \quad \phi_j(t)=\phi_j(t^-)+\frac{K(\phi_j(t^-))}{\mathcal{N}},\quad j\neq i.
\end{equation} Here the function $K(\phi)$ reflects how oscillators at different phases respond \textit{differently} to the same spike.  We call $K(\phi)$ the phase response function, which is related to the infinitesimal phase response curve (iPRC) in some literature (e.g. \cite{mauroy2011thesis}).  In the induced jump size in \eqref{pco-eq-3:spike-coupling}, the factor $\frac{1}{\mathcal{N}}$ is a scaling to ensure (at least formally) that there is a meaningful mean-field limit, i.e. the limit when the number of particles $\mathcal{N}$ goes to infinity. Throughout this work we assume $K(\phi)>0$, which implies that oscillators at lower phases become closer to $\Phi_F$ after receiving a spike. Note that $\phi_j(t)$ in \eqref{pco-eq-3:spike-coupling} may exceed $\Phi_F$, in which case $\phi_j$ will also spike at the same time $t$. This can cause more oscillators to spike at the same time in a chain-reaction manner \cite{mirollo1990synchronization,delarue2015particle}, which we will however not discuss in details here. 

In \eqref{pco-eq-1:no-spike}-\eqref{pco-eq-3:spike-coupling}, the oscillators interact with each other only through the sudden impulse \eqref{pco-eq-3:spike-coupling} when firing events occur. This singular interaction mechanism is in accordance with the name pulse-coupled oscillators. We also note that the phase response function $K(\phi)$ can already induce intricate dynamics, despite that here the inter-spike dynamics \eqref{pco-eq-1:no-spike} appears simple.

The particle system in the form of \eqref{pco-eq-1:no-spike}-\eqref{pco-eq-3:spike-coupling} was proposed by Peskin in 1975 \cite{MR0414135}, whose original motivation was to model pacemaker cells in the heart. Nowadays, pulse-coupled oscillators have found widespread applications, both in modeling various natural phenomena, including flashing fireflies \cite{Ramirez-Avila2019,peleg2023new}, spiking neurons \cite{ermentrout1991multiple,ermentrout2010mathematical,gerstner2014neuronal,coombes2023neurodynamics} and earthquakes \cite{olami1992self,herz1995earthquake}, as well as in engineering fields, including clustering algorithms \cite{rhouma2001self} and wireless sensor networks \cite{hong2005scalable,pagliari2010scalable,wang2020global}. These diverse applications share the common feature of the pulselike coupling \eqref{pco-eq-2:reset}-\eqref{pco-eq-3:spike-coupling}. For instance, in the context of neurons, \eqref{pco-eq-1:no-spike}-\eqref{pco-eq-3:spike-coupling} gives a network of the so-called integrate-and-fire neurons \cite{ermentrout1991multiple,ermentrout2010mathematical,gerstner2014neuronal,coombes2023neurodynamics}, where $\phi_i$ represents a rescaled version of a neuron's voltage (membrane potential). The interaction here is through the spike \eqref{pco-eq-3:spike-coupling}, triggered when a neuron's voltage reaches the threshold $\Phi_F$ \eqref{pco-eq-2:reset}, followed by a reset of the voltage to $0$ \eqref{pco-eq-2:reset}.

While the pulselike interaction \eqref{pco-eq-3:spike-coupling} is relevant for various applications, its discontinuous nature also contributes to the challenges for the mathematical analysis of \eqref{pco-eq-1:no-spike}-\eqref{pco-eq-3:spike-coupling}. In 1990, Mirollo and Strogatz \cite{mirollo1990synchronization} proved that when the phase response function $K(\phi)$ is monotonically increasing, the particle system will converge to a perfect synchronization state in finite time, for almost every initial data. A key idea was to introduce a tool called the \textit{firing map}, which became a popular framework in the analysis of such systems. The seminal work \cite{mirollo1990synchronization} initiated a wide range of studies on \eqref{pco-eq-1:no-spike}-\eqref{pco-eq-3:spike-coupling} in various settings and extensions, including  decreasing $K$ \cite{mauroy2008clustering}, delays \cite{gerstner1996rapid}, refractory periods \cite{kirk1997effect}, noise \cite{2010Newhall,guiraud2016stability}, non-identical oscillator \cite{urbanczik2001similar} and graph structures \cite{nunez2014synchronization,lyu2018global}. However, there are still many problems open even for the simple model \eqref{pco-eq-1:no-spike}-\eqref{pco-eq-3:spike-coupling} (e.g. \cite{mauroy2011thesis}), due to its distinctive interaction mechanism. 

\subsection{The mean-field continuity equation}

A powerful approach to study the particle system \eqref{pco-eq-1:no-spike}-\eqref{pco-eq-3:spike-coupling} is to take the mean-field limit, seeking statistical descriptions when the number of particles $\mathcal{N}$ goes to infinity. PDEs characterize the limit of the system, for which continuous tools can be applicable to gain insights and to make the problem tractable for analysis. Such an approach has been widely used in many fields, from statistical physics to biology and social science \cite{meanfieldreview,particlereview,pareschi2013interacting,jabinstochastic}.

For the pulse-coupled oscillator \eqref{pco-eq-1:no-spike}-\eqref{pco-eq-3:spike-coupling}, the state of the mean-field system is represented by a distribution function $\rho(t,\phi)$, where $\rho(t,\cdot)$ is the phase distribution for each time $t$, and the spike activity is characterized by a mean firing rate $N(t)$, the number of spikes per unit time. The dynamics of $(\rho,N)$ is governed by the following PDE system, which has been formally derived in physics literature \cite{kuramoto1991collective,1993PhysRevE.48.1483,mauroy2012global}, 
\begin{align}
    \label{eq:rho-t-1}
&\p_t\rho+\p_{\phi}\bigl([1+K(\phi)N(t)]\rho\bigr)=0, &&t>0,\ \phi\in(0,\Phi_F),\\
    &N(t)=[1+K(\Phi_F)N(t)]\rho(t,\Phi_F), &&t>0,\label{def-N-t}\\
    &[1+K(0)N(t)]\rho(t,0)=[1+K(\Phi_F)N(t)]\rho(t,\Phi_F), &&t>0,\label{bc-rho-t}\\
        &\rho(t=0,\phi)=\rho_{\init}(\phi), &&\phi\in[0,\Phi_F].\label{ic-rho-t}
\end{align} 
Equation \eqref{eq:rho-t-1} is a continuity equation describing the transport of density $\rho$ with the velocity field $\left(1+K(\phi)N(t)\right)$, for $\phi\in(0,\Phi_F)$, where the $K(\phi)N(t)$ term reflects the pulse-coupled interaction \eqref{pco-eq-3:spike-coupling}. The phase response function $K(\phi)$ describes how oscillators at different phases respond to the same stimulus differently. The basic assumptions on $K$ along this work can be summarized as 
\begin{assumption}\label{as:global-K}
The phase response function $K$ satisfies
\begin{align}
        &K\in C^2(\mathbb{R})\qquad \text{and $K',K''$ are bounded on $\mathbb{R}$,}\\
        &K(\phi)>0,\qquad \phi\in[0,\Phi_F].
        \end{align}
\end{assumption}

\

Note that while physically $K(\phi)$ is only defined on $[0,\Phi_F]$, it is sometimes mathematically convenient to extend it as a function on $\mathbb{R}$. A simple example is $K(\phi)=k\phi+b$ with some constants $k,b$. The phase response function $K$ has a crucial influence on the dynamics. In fact, we shall see increasing or decreasing (in $\phi$) phase response functions $K$ lead to drastically different qualitative behaviors.

The firing rate $N(t)$ is defined via the out-going boundary flux at $\Phi_F$ in \eqref{def-N-t}, since a particle spikes when it reaches $\Phi_F$ \eqref{pco-eq-2:reset}-\eqref{pco-eq-3:spike-coupling}. As the velocity field itself also depends on $N(t)$, \eqref{def-N-t} is a self-consistent equation, from which we can derive an expression of $N(t)$ in terms of the boundary value $\rho(t,\Phi_F)$
\begin{equation}\label{expression-N-t-classical}
    N(t)=\frac{\rho(t,\Phi_F)}{1-K(\Phi_F)\rho(t,\Phi_F)}.\quad
\end{equation} We note that $N(t)$ is vital to the dynamics of $(\rho,N)$. To ensure physical values for the firing rate $N(t)\in[0,+\infty)$, by \eqref{expression-N-t-classical} the following constraint on $\rho(t,\Phi_F)$ is needed 
\begin{equation}\label{constaint-classical-rho}
    0\leq \rho(t,\Phi_F)<\frac{1}{K(\Phi_F)}.
\end{equation} The case $\rho(t,\Phi_F)\geq \frac{1}{K(\Phi_F)}$ is interpreted as $N(t)=+\infty$, the blow-up of $N(t)$. 

Due to the boundary condition \eqref{bc-rho-t}, the flux at $\phi=0$ matches that at $\phi=\Phi_F$, since a particle is reset at $\phi=0$ after a spike at $\phi=\Phi_F$ \eqref{pco-eq-2:reset}. This reset process gives the conservation of mass for the system \eqref{eq:rho-t-1}
\begin{equation*}   
\int_0^{\Phi_F}\rho(t,\phi)d\phi=\int_0^{\Phi_F}\rho_{\init}(\phi)d\phi=1,
\end{equation*} 
where we assume the initial data $\rho_{\init}(\phi)$ is a probability density on $[0,\Phi_F]$.  Note that as $K(0)$ is not necessarily equal to $K(\Phi_F)$, and thus \eqref{bc-rho-t} does not imply a periodic boundary condition in $\rho$.

Despite being derived formally in the nineties \cite{kuramoto1991collective,1993PhysRevE.48.1483}, the mathematical literature on \eqref{eq:rho-t-1}-\eqref{ic-rho-t} is scarce. In physics and engineering literature, most studies focus on local linearized analysis around steady states for various extensions of \eqref{eq:rho-t-1}-\eqref{ic-rho-t} (e.g. \cite{kuramoto1991collective,1993PhysRevE.48.1483,van1996partial,bressloff1999mean,de2010coexistence}). A pioneering work is \cite{mauroy2012global} by Mauroy and Sepulchre in 2012. They constructed a Lyapunov function and showed some global qualitative behaviors for \eqref{eq:rho-t-1}-\eqref{ic-rho-t} in certain regimes. In contrast, there is a much richer literature on \textit{phase-coupled} oscillators (e.g. \cite{strogatz2000kuramoto,acebron2005kuramoto,dorfler2014synchronization,ha2016collective,ko2023emerging}), where the interaction is in a smooth and in some cases more symmetric way compared to the \textit{pulse-coupled} oscillators considered here.

More recently, some closely related models from neuroscience have attracted many mathematical studies, from both PDE \cite{caceres2011analysis,Antonio_Carrillo_2015,roux2021towards} and probability \cite{delarue2015particle,AIHP1164,cormier2020hopf} perspectives, which also connect to models in finance \cite{hambly2019mckean} and physics \cite{delarue2022global}. These models can be regarded as noisy variants of pulse-coupled oscillators. A comprehensive understanding of their long-term behavior, especially beyond linearized regimes, remains largely elusive (see e.g. \cite{cormier2020hopf,caceres2024asymptotic} for recent advances). One of our initial motivations towards \eqref{eq:rho-t-1}-\eqref{ic-rho-t} is an asymptotic limit for the voltage-conductance equation, another PDE from neuroscience \cite{perthame2019derivation,SIMA2024}. In our opinion, the system \eqref{eq:rho-t-1}-\eqref{ic-rho-t} is of fundamental importance, as a prototype model to understand the pulse-coupled interactions. 

\subsection{Our work: reformulation and main results}
In this work, we revisit the mean-field equation \eqref{eq:rho-t-1}-\eqref{ic-rho-t} from a PDE perspective, aiming to derive a suitable framework for a rigorous PDE analysis. We aim to do so with a \textit{reformulation}, which reveals hidden structures, leads to fruitful results, and opens new directions. 

The reformulation consists of two steps. First, instead of working with the probability density $\rho(t,\phi)$, we consider the quantile function $Q(t,\eta)$, also called the pseudo-inverse of the cumulative distribution function $F(t,\phi)$ associated with $\rho$. These relations can be informally illustrated as (see Section \ref{subsec:derive-pseudo-inverse} for full rigorous details)
\begin{equation}
F(t,\phi)=\int_0^{\phi}\rho(t,\tilde{\phi})d\tilde{\phi},\qquad \phi\in[0,\Phi_F],\qquad\qquad\qquad     F(t,Q(t,\eta))=\eta,\qquad \eta\in[0,1].
\end{equation} 
The pseudo-inverse reformulation has been applied to many other PDEs (e.g. \cite{Niet2000,CARRILLO2021271,carrillo2022optimal,CARRILLO2020106883}), in particular in the presence of the Dirac masses. In the context of pulse-coupled oscillators, the pseudo-inverse $Q$ has been used in the pioneering work \cite{mauroy2012global} to construct a Lyapunov function. Yet a self-contained equation for $Q$ itself was not explored before.

Second, instead of working in the original timescale $t$, we introduce a new timescale $\tau$, called the dilated timescale, which relates to the firing rate $N(t)$ as follows
\begin{equation}\label{def-tau-sing-1}
    d\tau=N(t)dt.
\end{equation} 
The classical framework to study the particle system \eqref{pco-eq-1:no-spike}-\eqref{pco-eq-3:spike-coupling} focuses on the firing map \cite{mirollo1990synchronization}, that is, by evolving the system from a spike time to the next spike time, regardless of the elapsed time $t$ in between. The use of the dilated timescale $\tau$ can be understood as a continuous abstract analogy of this firing map. Note that if $N(t)$ is a sum of Dirac deltas at the spiking times of a finite number of particles, then the dilated time scale $\tau$ is a piecewise constant non-decreasing function of $t$. The idea of using a firing-rate related timescale was introduced more recently in \cite{taillefumier2022characterization,sadun2022global} and \cite{dou2022dilating}, for related models which can be seen as noisy variants of \eqref{eq:rho-t-1}-\eqref{ic-rho-t}. Here we adopt the terminology ``dilated timescale'' from \cite{dou2022dilating}. The primary motivation to introduce the new timescale in \cite{taillefumier2022characterization,sadun2022global,dou2022dilating} was to define solutions beyond the blow-up of $N(t)$. Nevertheless, this work shows that the dilated timescale is also useful even in the classical regime with finite $N(t)$.

With these two ingredients, we reformulate \eqref{eq:rho-t-1}-\eqref{ic-rho-t}, the system for $\rho$ in $t$ timescale, as the following system for $Q$ in $\tau$ timescale (see Section \ref{sec:formulation-classical})
\begin{align}\label{eq-Q-tau-classical1}
     &\p_{\tau}Q+\p_{\eta}Q=\frac{1}{N(\tau)}+{K}(Q), &&\tau>0,\ \eta\in(0,1),\\
     &Q(\tau,0)=0, &&\tau>0,\label{bc-Q-tau-classical1}\\
     &\frac{1}{N(\tau)}=\p_{\eta}Q(\tau,1)-{K}(\Phi_F), &&\tau>0,\label{def-N-tau-classical1}\\
     &Q(\tau=0,\eta)=Q_{\init}(\eta), &&\eta\in[0,1].\label{ic-q-tau-classical1}
    \end{align} 
Here, a counterpart of \eqref{constaint-classical-rho} to ensure $N(\tau)<+\infty$  is
\begin{equation}\label{constaint-classical-Q-tau1}
    K(\Phi_F)<\p_{\eta}Q(\tau,1)<+\infty.
\end{equation}   

Many advantages emerge with the reformulation \eqref{eq-Q-tau-classical1}-\eqref{bc-Q-tau-classical1}-\eqref{def-N-tau-classical1}-\eqref{ic-q-tau-classical1}.  
In the reformulation, the nonlinear coupling through the firing rate is isolated, as $1/N(\tau)$ appears in the right-hand side of \eqref{eq-Q-tau-classical1}. Indeed, we can further obtain an alternative characterization of $1/N(\tau)$, as a Lagrangian multiplier to ensure a hidden constraint $Q(\tau,1)=\Phi_F$ (Proposition \ref{prop:char-N-classical}). Moreover, at a technical level, the characteristics of \eqref{eq-Q-tau-classical1} $\p_{\tau}+\p_{\eta}$ are moving with a constant speed compared to \eqref{eq:rho-t-1}, which facilitates analysis.

We now state the main results concerning the PDE analysis for the system \eqref{eq-Q-tau-classical1}-\eqref{ic-q-tau-classical1} under Assumption \ref{as:global-K} avoiding additional technical assumptions and precise statements which are postponed to the relevant sections.

\begin{theorem}[Well-posedness, Section \ref{sec:wps-classical}]\label{thm:wps-formal}
    Under suitable compatibility conditions on the initial data, there exists a unique classical solution to \eqref{eq-Q-tau-classical1}-\eqref{ic-q-tau-classical1} on the maximal existence interval $[0,\tau^*)$, with $0<\tau^*\leq+\infty$. If $\tau^*<+\infty$, then
    \begin{equation}
            \lim_{\tau\rightarrow (\tau^*)^-}N(\tau)=+\infty.
    \end{equation}
    Furthermore, the solution depends continuously on the initial data on every time interval $[0,T]$ with $T<\tau^*$. Moreover, the solution has better regularity under additional suitable assumptions for the initial data.
\end{theorem}

The main strategy to show Theorem \ref{thm:wps-formal} is to introduce an auxiliary problem, where the constraint \eqref{constaint-classical-Q-tau1} is relaxed and the well-posedness is easier to handle. This is realized by allowing the inverse of the firing rate $1/N(\tau)$ to take unphysical negative values. We then identify the blow-up of classical solutions as the first time when $1/N(\tau)$ touches zero.

Next, we focus on understanding the global dynamics of the system \eqref{eq-Q-tau-classical1}-\eqref{ic-q-tau-classical1} depending on the nonlinear interaction through the phase response function $K$. 

\begin{theorem}[Key Stability, Section \ref{sec:dynamics-classical}]\label{thm:formal-4}
    Under suitable compatibility conditions on the initial data, for any given two classical solutions $(Q_1,N_1)$ and $(Q_2,N_2)$ to \eqref{eq-Q-tau-classical1}-\eqref{ic-q-tau-classical1}, we have
    \begin{equation}\label{estimate-BV-first-time}
            e^{k_{\min}\tau}\|\p_{\eta}Q_1(0)-\p_{\eta}Q_2(0)\|_{L^1(0,1)}\leq \|\p_{\eta}Q_1(\tau)-\p_{\eta}Q_2(\tau)\|_{L^1(0,1)}\leq e^{k_{\max}\tau}\|\p_{\eta}Q_1(0)-\p_{\eta}Q_2(0)\|_{L^1(0,1)}.
    \end{equation} 
    The estimates hold whenever the two classical solutions exist up to time $\tau$, and $k_{\min}$ and $k_{\max}$ are given by
    \begin{equation}
{k}_{\min}:=K'(0)+\int_0^{\Phi_F}(K''(\phi))_-d\phi \, \leq \, {k}_{\max}:=K'(0)+\int_0^{\Phi_F}(K''(\phi))_+d\phi.
\end{equation}
\end{theorem}
Here $(x)_-:=\min(x,0)$ and $(x)_+:=\max(x,0)$. Note that in general
\begin{equation*}
    k_{\min}\leq \min_{\phi\in[0,\Phi_F]}K'(\phi)\leq \max_{\phi\in[0,\Phi_F]}K'(\phi)\leq k_{\max}.
\end{equation*}
If $K$ is either convex or concave, then
\begin{equation*}
k_{\min}=\min_{\phi_\in[0,\Phi_F]}K'(\phi),\qquad k_{\max}=\max_{\phi_\in[0,\Phi_F]}K'(\phi)\,.
\end{equation*} 
While Theorem \ref{thm:formal-4} holds for general $K$, it has particularly interesting consequences on the long time behavior if $k_{\min}$ and $k_{\max}$ are of the same sign.
Notice that a particularly important case of $k_{\max}\geq k_{\min}>0$ is $K' > 0$ and $K$ is either concave or convex, and analogously, a particular case of $k_{\min}\leq k_{\max}<0$ is $K' < 0$ and $K$ is either concave or convex.

Such results can be generalized to other distances, including a modified $L^2$ distance (see Theorem \ref{thm:L2}-\ref{thm:P_first}). Both the statements and the proofs of these results crucially rely on the reformulation in terms of $Q$ and $\tau$. For the statements, the definitions of the distances used involve $Q$, and moreover we need to compare two solutions at the same time value in timescale $\tau$, \textit{not} in the original timescale $t$. For the proofs, a key step is to ``filter out'' the $1/N(\tau)$ term, which is possible as in the reformulated equations it is singled out as a constant in $\eta$ in~\eqref{eq-Q-tau-classical1}.

\begin{table}[H]
\centering
\renewcommand{\arraystretch}{1.3}{
\begin{tabular}{|c|ccc|}
\hline
 &
  \multicolumn{1}{c|}{$k_{\min}>0$} &
  \multicolumn{1}{c|}{$k_{\max}<0$} &
  General $K$ \\ \hline
$K$ large &
  \multicolumn{3}{c|}{Blow-up} \\ \hline
$K$ not large &
  \multicolumn{1}{c|}{\begin{tabular}[c]{@{}c@{}} \quad Blow-up $^{(a)}$ \quad  \\ \end{tabular}} &
  \multicolumn{1}{c|}{\begin{tabular}[c]{@{}c@{}} \quad Convergence $^{(b)}$ \quad \end{tabular}} &
 \quad Unique steady state \quad \\ \hline
\end{tabular}
}
\caption{{\small Qualitative behaviors of the solution in different regimes of $K$. Here ``$K$ large'' means that \eqref{ss-condition-first-time} is violated and ``$K$ not large'' means that \eqref{ss-condition-first-time} holds. We recall that throughout this paper Assumption \ref{as:global-K} is taken, which in particular imposes that $K>0$.   Blow-up $^{(a)}$: every solution blows up in finite time except the unique steady state. Convergence $^{(b)}$: for good initial data, the solution converges to the unique steady state exponentially in time. 
}
}
\label{tb:1}
\end{table}

Now we present results on the qualitative behaviors of the solution in different regimes of $K$, which are concisely summarized in Table \ref{tb:1} (with expanded statements to be given in Theorem \ref{thm:formal-5}-\ref{thm:formal-6} below). The regimes are identified by two factors: the signs of $k_{\min}$ and $k_{\max}$, and the size of $K$. The signs of $k_{\min}$ and $k_{\max}$, which relate to the monotonicity of $K$, already play an important role in Theorem \ref{thm:formal-4}. The size of $K$, which reflects the strength of the pulse-coupled interaction, is also crucial to the dynamics. For instance, it is directly linked to the existence of the steady state.

We first state the long time convergence result when $k_{\max}<0$ and $K$ is not large.

\begin{theorem}[Long time convergence, Section \ref{sec:global}]\label{thm:formal-5} 
\

(i) (Dichotomy on the steady state, Proposition \ref{prop:ss}) There exists a unique steady state to \eqref{eq-Q-tau-classical1}-\eqref{ic-q-tau-classical1} if
\begin{equation}\label{ss-condition-first-time}
        \int_0^{\Phi_F}\frac{1}{K(\phi)}d\phi> 1,
\end{equation} otherwise there is no steady state.

(ii) (Global existence and long time convergence to the steady state, Theorem \ref{thm:global-solu}-\ref{thm:convergence}) Suppose that $K'<0$ and \eqref{ss-condition-first-time}. Under suitable compatibility conditions on the initial data, and another assumption on the initial data relating to the size of $K$, the solution to \eqref{eq-Q-tau-classical1}-\eqref{ic-q-tau-classical1} is global with a uniform-in-time bound on the firing rate $N(\tau)$
\begin{equation}\label{bd-firing-rate-first-time}
            0<C_{\min}\leq N(\tau)\leq C_{\max}<+\infty,\qquad \tau\geq0,
\end{equation} where $C_{\min}$ and $C_{\max}$ are two explicit constants depending on the initial data.

Further suppose $k_{\max}<0$. Then the solution converges to the unique steady state exponentially in the long time.
\end{theorem} 

For the global existence, the proof relies on several auxiliary structures, including an integral equation for $1/N(\tau)$ (Proposition \ref{prop:eq-H}-\ref{prop:eq-tildeN}), which may be of their own interest. When $k_{\max}<0$, Theorem \ref{thm:formal-4} implies a contraction of the distance between arbitrary two solutions, provided that they exist. Therefore, the long time convergence to the steady state can be readily deduced once the global existence is available.
On the contrary, finite time blow-ups can be proved in two regimes -- when $k_{\min}>0$ or when \eqref{ss-condition-first-time} is violated. The latter means that the size of $K$ is large.  

\begin{theorem}[Blow-up, Section \ref{sec:blowup}]\label{thm:formal-6}  

\

        (i)  Suppose $k_{\min}>0$. Then every solution to \eqref{eq-Q-tau-classical1}-\eqref{ic-q-tau-classical1} blows up in finite time, except if the solution were the unique steady state, when it exists. 
         
        (ii) Suppose \eqref{ss-condition-first-time} is violated. Then  every solution to \eqref{eq-Q-tau-classical1}-\eqref{ic-q-tau-classical1} blows up in finite time. 
\end{theorem}

In Section \ref{sec:blowup}, we also obtain some explicit estimates on the blow-up time $\tau^*$. For the first scenario when $k_{\min}>0$, the finite time blow-up is a consequence of Theorem \ref{thm:formal-4}. Indeed, when $k_{\min}>0$, \eqref{estimate-BV-first-time} implies that the distance between two solutions expands exponentially, if they globally exist. However, we can directly show that such a distance shall be uniformly bounded $\|\p_{\eta}Q_1(\tau)-\p_{\eta}Q_2(\tau)\|_{L^1(0,1)}\leq 2\Phi_F$ (see later \eqref{bdd-BV}), which leads to a contradiction. For the second scenario, we note that it gives a sharp criteria for the blow-up of \textit{every} solution, since otherwise \eqref{ss-condition-first-time} holds and by Theorem \ref{thm:formal-5}-(i) there exists a steady state, which is naturally a global solution.

The convergence when $k_{\max}<0$ and the blow-up when $k_{\min}>0$ provide two cases to understand how the phase response function $K$ shapes the dynamics of pulse-coupled oscillators. Indeed, for the mean-field equation \eqref{eq:rho-t-1}-\eqref{ic-rho-t} some partial results have been shown in \cite{mauroy2012global}, 
see later Corollary \ref{cor:BV}, Remark \ref{rmk:51} and \ref{rmk:62} for more detailed discussions on the comparison between our results and theirs. The conditions $k_{\max}<0$/$k_{\min}>0$ avoid the convexity or concavity assumption on $K$ needed in \cite{mauroy2012global}. At the particle system level, dynamics with monotone phase response function $K$ are studied in  \cite{mauroy2008clustering,mauroy2009erratum} for $K'<0$ together with convexity/concavity assumptions on $K$, and in \cite{mirollo1990synchronization} for $K'>0$. 
Our results supersedes these previous works for classical solutions of the mean-field equation \eqref{eq:rho-t-1}-\eqref{ic-rho-t}. In particular, the blow-up of every solution, when \eqref{ss-condition-first-time} is violated, is new at the mean-field level. 

Summarizing, the main contribution of this work, is to build a rigorous framework to analyze the mean-field equation for pulse-coupled oscillators, based on the new \textit{reformulation}: the pseudo-inverse $Q$ and the dilated timescale $\tau$. Such a framework is \textit{natural} in our opinion as new structures are revealed and fruitful results can be formulated in a convenient way and improved from existing literature. It is worth emphasizing that the key stability Theorem \ref{thm:formal-4}, which gives much information on the qualitative behavior, relies deeply on the reformulation both in its statement and proof. There is much more unknown concerning the mean-field equation, including but not restricted to the dynamics of general non-monotone $K$ and beyond the blow-up, let alone its generalizations motivated from various fields of science and engineering.  

The rest of this paper is arranged as follows: In Section \ref{sec:formulation-classical}, we derive the reformulation and prove its basic properties. Section \ref{sec:wps-classical} is devoted to studying the well-posedness, where we prove Theorem \ref{thm:wps-formal}. Section \ref{sec:dynamics-classical} focuses on proving the key stability estimate, establishing Theorem \ref{thm:formal-4}. In Section \ref{sec:global}, we investigate global existence and long-time convergence to the steady state, providing the proof for Theorem \ref{thm:global-solu}. Section \ref{sec:blowup} is dedicated to the analysis of blow-up and the proof of Theorem \ref{thm:formal-6}. 


\section{Pseudo-inverse formulation in dilated timescale}\label{sec:formulation-classical}

In this section, we introduce a new formulation for the mean-field system of pulse-coupled oscillators, whose analysis is the main goal of this work.

We start with the mean-field system \eqref{eq:rho-t-1}-\eqref{ic-rho-t} for the distribution in the phase variable $\rho(t,\phi)$. Recall that Assumption \ref{as:global-K} is supposed throughout this paper. Moreover, to ensure physical values for the firing rate, the constraint \eqref{constaint-classical-rho} is imposed. We define classical solutions to \eqref{eq:rho-t-1}-\eqref{ic-rho-t} as follows.
\begin{definition}[Classical solution to \eqref{eq:rho-t-1}-\eqref{ic-rho-t}]\label{def:classical-rho} Given $0<T\leq +\infty$, initial data $\rho_{\init}(\phi)$ which is a $C^1$ probability density with $\rho_{\init}(\Phi_F)<1/K(\Phi_F)$, we say $(\rho,N)$ is a classical solution to \eqref{eq:rho-t-1}-\eqref{ic-rho-t} on $[0,T)$ if 
\begin{enumerate}
    \item $\rho(t,\phi)\in C^{1}([0,T)\times[0,\Phi_F])$ and $N(t)\in C[0,T)$ are non-negative.
    \item Equations \eqref{eq:rho-t-1}-\eqref{ic-rho-t} are satisfied in the classical sense for $t\in[0,T)$.
    \item The constraint \eqref{constaint-classical-rho} $0\leq \rho(t,\Phi_F)<{1}/{K(\Phi_F)}$ holds for all $t\in[0,T)$.
\end{enumerate}
\end{definition}

The new formulation for \eqref{eq:rho-t-1}-\eqref{ic-rho-t} consists of two key ingredients: the pseudo-inverse function and the time-dilation, which are introduced next. 


\subsection{Reformulation in terms of pseudo-inverses}\label{subsec:derive-pseudo-inverse}

Given a probability distribution $\mu$ on $[0,\Phi_F]$, we define its (right-continuous) cumulative function by
\begin{equation}    
F(\phi)=\mu\bigl([0,\phi]\bigr),\quad \phi\in[0,\Phi_F].
\end{equation}
In general $F$ may not be either continuous or strictly increasing, but we can always define its pseudo-inverse via
\begin{equation}
    Q(\eta)=\inf \{\phi\geq0: F(\phi)\geq \eta\},\qquad \eta\in[0,1].
\end{equation}
Here we focus on the classical case when $\mu$ has a positive density function $\rho$, and then $Q$ is indeed the regular inverse of $F$. See Figure \ref{fig:enter-label} for an illustration of the general case. 

\begin{figure}[htbp]
    \centering
    \includegraphics[width=0.8\linewidth]{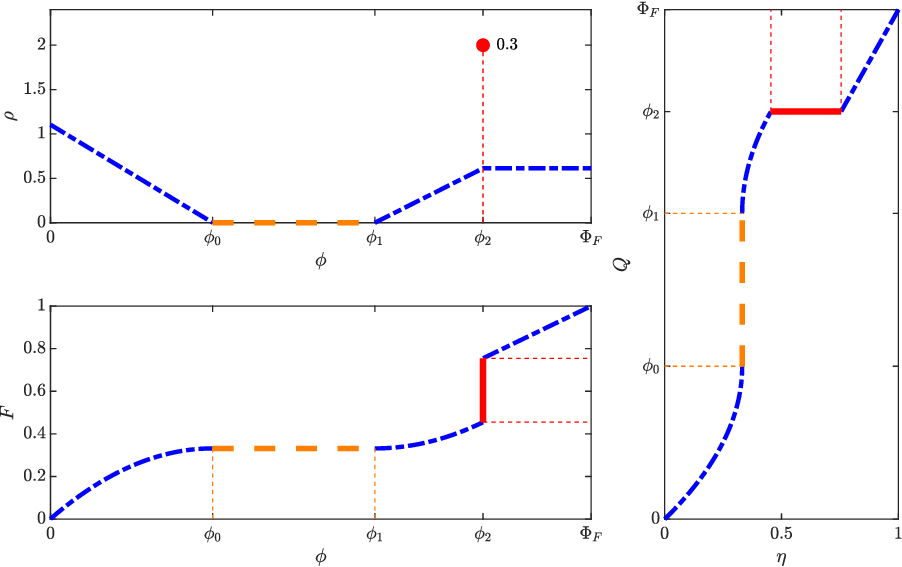}
    \caption{Illustration of the relations between the probability distribution $\mu$ (or its density $\rho$), the cumulative function $F$ and the pseudo-inverse of the cumulative function $Q$. Upper left: Plot of the density $\rho$ with a red circle indicating a Dirac mass part of $\mu$ (here the size of the Dirac mass is $0.3$). Lower left: Plot of the cumulative function $F$. Right: Plot of the pseudo-inverse $Q$. In all three figures, the red solid line indicates the Dirac mass part, which manifests as a jump in $F$ and a constant part in $Q$. The orange dashed line represents the vacuum part, where the density is zero, resulting in a flat segment in $F$ that induces a jump in $Q$. And the blue dash-dotted line indicates the regular part where there is no Dirac mass, and the density is positive.}
    \label{fig:enter-label}
\end{figure}

We carry on such reformulations for \eqref{eq:rho-t-1}-\eqref{ic-rho-t}, assuming that $\rho(t,\phi)$ is positive everywhere. We define for each $t$ the cumulative function via
\begin{equation}\label{def-F-t}  
F(t,\phi)=\int_0^{\phi}\rho(t,\tilde{\phi})d\tilde{\phi},\qquad \phi\in[0,\Phi_F].
\end{equation} 
Then for each $t$ we can define the pseudo-inverse for the cumulative function by
\begin{equation}\label{def-Q-t}
    Q(t,\eta)=\inf\{\phi\geq 0: F(t,\phi)\geq \eta\},\qquad \eta\in[0,1].
\end{equation} 
Due to the positivity of  $\rho$, $F(t,\cdot)$ is indeed strictly increasing and $C^1$. Therefore $Q(t,\cdot)$ is actually its $C^1$ inverse, satisfying
\begin{equation}\label{relation-FQ}
    F(t,Q(t,\eta))=\eta,\qquad \eta\in[0,1],
\end{equation}
\begin{equation}\label{pFpphi}
    \p_{\phi}F(t,\phi)=\rho(t,\phi),\quad \phi\in[0,\Phi_F],
\end{equation}
and
\begin{equation}\label{tmp-pqpe}
    \p_{\eta}Q(t,\eta)=\frac{1}{\p_{\phi}F(t,Q)}=\frac{1}{\rho(t,Q(t,\eta))},\quad \eta\in[0,1].
\end{equation}
Using the chain rule and the boundary conditions, it is not difficult to obtain
\begin{equation}\label{eq:q-t-1}
    \p_tQ=(1+{K}(Q)N(t))-N(t)\p_{\eta}Q,\quad t>0,\eta\in(0,1).
\end{equation}
The definition of $Q$ \eqref{def-Q-t} implies a boundary condition at $\eta=0$
\begin{equation}\label{bc:q-t-1}
    Q(t,0)=0,\quad t>0.
\end{equation}
To close the system, we need to represent $N(t)$ in terms of $Q$. Thanks to the assumption that $\rho$ is positive everywhere, we have
\begin{equation}\label{Q1-PhiF}
    Q(t,\eta=1)=\Phi_F,\quad t\geq 0.
\end{equation} Therefore using \eqref{tmp-pqpe} we get
\begin{equation*}\rho(t,\Phi_F)=\rho(t,Q(t,\eta=1))=\frac{1}{\p_{\eta}Q(t,1)}.
\end{equation*} Hence by \eqref{expression-N-t-classical} we obtain
\begin{equation}\label{def-N-Q-t}
    N(t)=\frac{\rho(t,\Phi_F)}{1-{K}(\Phi_F)\rho(t,\Phi_F)}=\frac{1}{\p_{\eta}Q(t,1)-{K}(\Phi_F)}.
\end{equation}
To summarize we have derived from \eqref{eq:rho-t-1}-\eqref{ic-rho-t} a system for the pseudo-inverse function $Q(t,\eta)$
\begin{equation}\label{system-Q-classical}
    \begin{aligned}
        &\p_tQ+N(t)\p_{\eta}Q=1+{K}(Q)N(t),\quad &&t>0,\ \eta\in(0,1),\\[5pt]
    &Q(t,0)=0,\quad &&t>0,\\
     &N(t)=\frac{1}{\p_{\eta}Q(t,1)-{K}(\Phi_F)},\quad &&t>0,
    \end{aligned}
\end{equation}
which is complemented by the initial data
\begin{equation}\label{ic-q-classical}
    Q(t=0,\eta)=Q_{\init}(\eta),\quad \eta\in[0,1].
\end{equation}
We need an additional constraint on $\p_{\eta}Q(t,1)$ to ensure that $0<N(t)<+\infty$ in \eqref{def-N-Q-t}
\begin{equation}\label{constaint-classical-Q-t}
    K(\Phi_F)<\p_{\eta}Q(t,1)<+\infty.
\end{equation} Here $\p_{\eta}Q(t,1)>K(\Phi_F)$ is the same as \eqref{constaint-classical-rho}, which ensures $N(t)<+\infty$. We also impose $\p_{\eta}Q(t,1)<+\infty$, which gives $N(t)>0$ and $\rho(t,\Phi_F)>0$, aligning with our assumption that $\rho$ is positive everywhere.

\subsection{Time-dilation and some basic properties}

Next, we rewrite the system in the dilated timescale $\tau$ with
\begin{equation}\label{def-tau-sing}
    d\tau=N(t)dt.
\end{equation} 
This change of time is invertible when $0<N<+\infty$. By dividing the equation \eqref{eq:q-t-1} by $N(t)$, we can derive the system in $\tau$ \eqref{eq-Q-tau-classical1}-\eqref{constaint-classical-Q-tau1}, which we recall here for convenience.
\begin{align}\label{eq-Q-tau-classical}
        &\p_{\tau}Q+\p_{\eta}Q=\frac{1}{N(\tau)}+{K}(Q), &&\tau>0,\ \eta\in(0,1),\\
    &Q(\tau,0)=0, &&\tau>0,\label{bc-Q-tau-classical}\\
     &\frac{1}{N(\tau)}=\p_{\eta}Q(\tau,1)-{K}(\Phi_F), &&\tau>0,\label{def-N-tau-classical}\\
     &Q(\tau=0,\eta)=Q_{\init}(\eta), &&\eta\in[0,1].\label{ic-q-tau-classical}
    \end{align} with the constraint to ensure that $0<N(\tau)<+\infty$ given by
\begin{equation}\label{constaint-classical-Q-tau}
    K(\Phi_F)<\p_{\eta}Q(\tau,1)<+\infty.
\end{equation} 
We shall show that the new timescale $\tau$ defined via \eqref{def-N-tau-classical} facilitates the estimates in later sections as the quantities of interest, in particular $Q$, are drifted with unit speed. 

Equation \eqref{eq-Q-tau-classical} can be interpreted as a nonlinear transport equation with an inflow boundary condition \eqref{bc-Q-tau-classical} at $\eta=0$. Its nonlinearity arises from two factors. Firstly the phase response function $K(Q)$ can be nonlinear. The second source of nonlinearity is the term $1/N(\tau)$. As in \eqref{def-N-tau-classical}, $1/N(\tau)$ depends on the boundary derivative, and the constraint \eqref{constaint-classical-Q-tau} is imposed to ensure its positivity. This term, $1/N(\tau)$, reflects the pulse-coupled interaction between oscillators, and contributes to a major challenge of this system.

\begin{remark}\label{rmk:dilated}
    The idea of introducing a firing-rate-dependent timescale has been used by \cite{taillefumier2022characterization,sadun2022global} and \cite{dou2022dilating} for different but related models. Their main motivation was to define generalized solutions allowing the blow-up of $N$. Formally, when $N$ is large, a small time interval in $t$ will be mapped into a long one by \eqref{def-tau-sing}. That is why $\tau$ is called the dilated timescale in \cite{dou2022dilating}, actually there they define $d\tau=(N(t)+1)dt$ instead of $N(t)dt$. We follow the same terminology here even if we focus on  classical solutions when $N\in(0,+\infty)$. We will show that under certain conditions the firing rate of the system \eqref{eq-Q-tau-classical}-\eqref{ic-q-tau-classical} for classical solutions can blow-up in finite time. 
\end{remark}

Similar to Definition \ref{def:classical-rho}, we define the classical solution to \eqref{eq-Q-tau-classical}-\eqref{ic-q-tau-classical}. 

\begin{definition}[Classical solution to \eqref{eq-Q-tau-classical}-\eqref{ic-q-tau-classical}]\label{def:classical-q-tau} Given $0<T\leq +\infty$, initial data $Q_{\init}(\eta)$ which is an increasing $C^1$ function with $Q_{\init}(0)=0,Q_{\init}(1)=\Phi_F$ and $\frac{d}{d\eta}Q_{\init}(\eta=1)>K(\Phi_F)$. We say $(Q,N)$ is a classical solution to \eqref{eq-Q-tau-classical}-\eqref{ic-q-tau-classical} on $[0,T)$ if 
\begin{enumerate}
    \item $Q(\tau,\eta)\in C^{1}([0,T)\times[0,1])$ is increasing in $\eta$ and $N(\tau)\in C[0,T)$ is positive.
    \item Equations  \eqref{eq-Q-tau-classical}-\eqref{ic-q-tau-classical} are satisfied in the classical sense for $\tau\in[0,T)$.
    \item The constraint \eqref{constaint-classical-Q-tau} holds for all $\tau\in[0,T)$.
\end{enumerate}
\end{definition}

We shall check that the feature of a pseudo-inverse can be preserved by system \eqref{eq-Q-tau-classical}-\eqref{constaint-classical-Q-tau}. Precisely we shall show that starting with a pseudo-inverse $Q_{\init}(\eta)$, at each time $\tau$ the solution $Q(\tau,\eta)$ is a pseudo-inverse to some probability distribution on $[0,\Phi_F]$. As a necessary condition, we first check that the monotonicity in $\eta$ can be preserved.

\begin{proposition}\label{prop:keep-mono}
        Suppose $Q(\tau,\eta)\in C^{1}([0,T)\times[0,1])$ and $N(\tau)>0$ satisfy \eqref{eq-Q-tau-classical}-\eqref{bc-Q-tau-classical} on $[0,T)$, with an initial data \eqref{ic-q-tau-classical} $Q_{\init}(\eta)$ increasing in $\eta$. Then $Q(\tau,\eta)$ is increasing in $\eta$, for any $\tau\in[0,T)$. Furthermore, for $\tau\in[0,T)$ we have
    \begin{equation}\label{ptau-bc-0}
    \p_{\eta}Q(\tau,0)=\frac{1}{N(\tau)}+K(0).
    \end{equation}
    If additionally \eqref{def-N-tau-classical} holds, we have for $\tau\in[0,T)$
    \begin{equation}\label{ptau-diff-0-1}
        \p_{\eta}Q(\tau,1)-\p_{\eta}Q(\tau,0)=K(\Phi_F)-K(0).
    \end{equation}
\end{proposition}
\begin{proof}
    The regularity of $Q$ ensures that the equation holds at $\eta=0$, resulting in
    \begin{align}\notag
        \p_{\tau}Q(\tau,0)&=\frac{1}{N(\tau)}+K(Q(\tau,0))-\p_{\eta}Q(\tau,0)\\&=\frac{1}{N(\tau)}+K(0)-\p_{\eta}Q(\tau,0),\label{tmp-1}
    \end{align} where the boundary condition \eqref{bc-Q-tau-classical} is used. Note that \eqref{bc-Q-tau-classical} also implies $\p_{\tau}Q(\tau,0)=0$, that combined with \eqref{tmp-1} leads to \eqref{ptau-bc-0}.
    Now suppose additionally $Q\in C^{2}([0,T)\times[0,1])$. Taking the derivative w.r.t $\eta$ in equation \eqref{eq-Q-tau-classical}, we have
    \begin{equation*}
        (\p_{\tau}+\p_{\eta})(\p_{\eta}Q)=K'(Q)\p_{\eta}{Q}.
    \end{equation*} Therefore we conclude that $Z(\tau,\eta):=\p_{\eta}Q(\tau,\eta)$ is a classical solution to
    \begin{equation}\label{eq-petaQ-Z}
        \begin{aligned}
            &(\p_{\tau}+\p_{\eta})Z=K'(Q)Z,\quad &&\tau\in(0,T),\eta\in(0,1),\\
            &Z(\tau,0)=\frac{1}{N(\tau)}+K(0)>0,\quad &&\tau\in(0,T),\\
            &Z(0,\eta)=\frac{d}{d\eta}Q_{\init}(\eta),\quad &&\eta\in[0,1],
        \end{aligned}
    \end{equation} which preserves the non-negativity of $\frac{d}{d\eta}Q_{\init}(\eta)$.

   For general case when $Q$ is not necessarily $C^2$, we can work with characteristics to show that $Z$ is a mild solution (in the sense of characteristics, c.f. Definition \ref{def:mild-1-external-regular} and \eqref{formula-peta-flow-recall} in Section \ref{sec:wps-classical}) of \eqref{eq-petaQ-Z}. And we can conclude the same result by working with the characteristics ODEs. For simplicity we omit the details here.
    
    Finally, \eqref{ptau-diff-0-1} follows from combining \eqref{ptau-bc-0} with \eqref{def-N-tau-classical}.
\end{proof}

\begin{remark}\label{rmk:initial-boundary-compatible}
 The $C^1$ regularity for classical solutions needs 
 that \eqref{ptau-diff-0-1} is satisfied at $\tau=0$, thereby imposing an implicit constraint on the initial data. Indeed, this means compatibility with the boundary condition for the first order derivative, which is in accordance with the $C^1$ requirement for a classical solution.
\end{remark}

We also need $Q(\tau,1)\leq \Phi_F$ for $Q(\tau,\cdot)$ to be a pseudo-inverse corresponding to a measure supported in $[0,\Phi_F]$. As \eqref{eq-Q-tau-classical} is a first order equation, one (inflow) boundary condition \eqref{bc-Q-tau-classical} at $\eta=0$ is enough. Yet in the previous derivation \eqref{Q1-PhiF} we have also used 
\begin{equation}\label{eq:q1-phiF}
    Q(\tau,1)=\Phi_F,\quad \tau>0,
\end{equation} which follows from the definition of the pseudo-inverse of the probability density function $\rho$ with $\rho(\tau,\Phi_F)>0$. Although for the self-contained system of $Q$, \eqref{eq:q1-phiF} is not explicitly imposed, we check next that it is ensured by the definition of $N(\tau)$ in \eqref{def-N-tau-classical}.

\begin{proposition}\label{prop:char-N-classical}
    Suppose $Q(\tau,\eta)\in C^{1}([0,T)\times[0,1])$ and $N(\tau)\in C[0,T)$ satisfies \eqref{eq-Q-tau-classical} and constraint \eqref{constaint-classical-Q-tau} on $[0,T)$, with an initial condition \eqref{ic-q-tau-classical} satisfying $Q_{\init}(1)=\Phi_F$. Then \eqref{def-N-tau-classical} and \eqref{eq:q1-phiF} are equivalent. 
\end{proposition}
\begin{proof}
     The regularity of $Q$ ensures that the equation is satisfied at $\eta=1$, which gives
     \begin{equation}\label{eq-at-eta-1-classical}
         \p_{\tau}Q(\tau,1)=\frac{1}{N(\tau)}+K(Q(\tau,1))-\p_{\eta}Q(\tau,1).
     \end{equation}
     If \eqref{eq:q1-phiF} holds, as initially $Q_{\init}(1)=\Phi_F$, we have
     \begin{align*}
         0\equiv \p_{\tau}Q(\tau,1)&=\frac{1}{N(\tau)}+K(Q(\tau,1))-\p_{\eta}Q(\tau,1)\\&=\frac{1}{N(\tau)}+K(\Phi_F)-\p_{\eta}Q(\tau,1),
     \end{align*} where in the last step we use again \eqref{eq:q1-phiF}. This yields \eqref{def-N-tau-classical}.

     On the other hand, if \eqref{def-N-tau-classical} holds, substituting it into \eqref{eq-at-eta-1-classical} results in
     \begin{equation}\label{tmp-ODE}
         \frac{d}{d\tau} Q(\tau,1)=K(Q(\tau,1))-K(\Phi_F),
     \end{equation}which is an ODE for $Q(\tau,1)$, together with an initial condition $Q(\tau=0,1)=Q_{\init}(1)=\Phi_F$. As $K$ is $C^1$, the unique solution to this ODE is the constant solution $\Phi_F$, which gives \eqref{eq:q1-phiF}.
\end{proof}

\begin{remark}
    In view of Proposition \ref{prop:char-N-classical}, we may interpret $1/N(\tau)$ as a Lagrangian multiplier to ensure the additional boundary condition \eqref{eq:q1-phiF}. 
\end{remark}

We can further extend Proposition \ref{prop:char-N-classical}  to the following result.

\begin{corollary}\label{cor:5-equivalent} 
        Suppose $Q(\tau,\eta)\in C^{1}([0,T)\times[0,1])$ and $N(\tau)\in C[0,T)$ satisfy \eqref{eq-Q-tau-classical} on $[0,T)$ with initially $Q(\tau=0,0)=0,Q(\tau=0,1)=\Phi_F$. 
        For the following five statements,
        \begin{enumerate}
            \item \eqref{bc-Q-tau-classical}, i.e. $Q(\tau,\eta=0)\equiv0,$ for $\tau\in[0,T)$,
            \item \eqref{eq:q1-phiF}, i.e. $Q(\tau,\eta=1)\equiv\Phi_F,$ for $\tau\in[0,T)$,
            \item \eqref{ptau-bc-0}, i.e. $ \p_{\eta}Q(\tau,\eta=0)=\frac{1}{N(\tau)}+K(0),$ for $\tau\in[0,T)$,
            \item \eqref{def-N-tau-classical}, i.e. $ \p_{\eta}Q(\tau,\eta=1)=\frac{1}{N(\tau)}+K(\Phi_F),$ for $\tau\in[0,T)$,
            \item \eqref{ptau-diff-0-1}, i.e. $\p_{\eta}Q(\tau,1)-\p_{\eta}Q(\tau,0)=K(\Phi_F)-K(0)$ for $\tau\in[0,T)$,
        \end{enumerate}
        we have 1. $\Leftrightarrow$ 3. and 2. $\Leftrightarrow$ 4., and every two of 3. 4. 5. can deduce the other. 
        
        In particular, if $(Q,N)$ is a classical solution to \eqref{eq-Q-tau-classical}-\eqref{ic-q-tau-classical}, then all of 1.-5. hold.
\end{corollary}
\begin{proof}
    Indeed, 2. $\Leftrightarrow$ 4. is Proposition \ref{prop:char-N-classical}, and 1. $\Leftrightarrow$ 3. can be proved similarly. And each one of 3.-5. is a linear combination of the other two.
\end{proof}
\begin{remark}
Corollary \ref{cor:5-equivalent} suggests some possibility of an equivalent reformulation of \eqref{eq-Q-tau-classical}-\eqref{ic-q-tau-classical}. Indeed, any two of 1.-5. can deduce the rest except the pair 1. 3. or 2. 4. . Totally we have 8 choices. Here we choose 1. and 4., i.e. \eqref{bc-Q-tau-classical} and \eqref{def-N-tau-classical}, which might be natural in view of the characteristic of equation \eqref{eq-Q-tau-classical}: i) \eqref{bc-Q-tau-classical} is an inflow boundary condition is imposed at $\eta=0$, and \eqref{def-N-tau-classical} says that the nonlinearity $1/N$ depends on the out-going flux at $\eta=1$.
\end{remark}

\subsection{Equivalent formulations}
Now we return to how the $Q$ formulation in $\tau$ \eqref{eq-Q-tau-classical}-\eqref{ic-q-tau-classical} corresponds to the $\rho$ formulation in $t$ \eqref{eq:rho-t-1}-\eqref{ic-rho-t}. We have focused on the classical case when $0<N(t)<+\infty$, which implies that the change of timescale \eqref{def-tau-sing} and its inverse $dt=\frac{1}{N(\tau)}d\tau$ are well-defined. We have also assumed that $0<\rho<\infty$, which in terms of pseudo-inverse \eqref{tmp-pqpe} means $0<\p_{\eta}Q<\infty$. Under these assumptions we have derived the equivalent system for $Q$ in $\tau$ timescale from the system of $\rho$ in $t$ timescale. We give a formal summary of this equivalence as follows.

\begin{proposition}\label{prop:equil-rho-Q}
Suppose $(\rho,N)$ is a classical solution to \eqref{eq:rho-t-1}-\eqref{ic-rho-t} on $[0,T)$. Additionally suppose $\rho$ is positive everywhere which implies $N$ is also positive. For each time $t$, denote the pseudo-inverse associated with $\rho(t,\cdot)$ as $Q(t,\cdot)$. Then $Q$ satisfies \eqref{system-Q-classical}. And through an invertible change of timescale \eqref{def-tau-sing}, we obtain a classical solution to \eqref{eq-Q-tau-classical}-\eqref{ic-q-tau-classical}, on time interval $[0,\tau_T)$ with $\tau_T=\int_0^{T}N(t)dt$.

On the other hand, suppose $(Q,N)$ is a classical solution to \eqref{eq-Q-tau-classical}-\eqref{ic-q-tau-classical}. Additionally assume that the spatial derivative $\p_{\eta}Q$ is $C^{1}$ and is positive everywhere. Then we can invert the above process to obtain a positive classical solution to \eqref{eq:rho-t-1}-\eqref{ic-rho-t}. 
\end{proposition}

\begin{remark}
    Note that $\rho$ connects to the spatial derivative $\p_{\eta}Q$ via \eqref{tmp-pqpe}. Hence in the second part of the statement we need $\p_{\eta}Q$ to be $C^{1}$, in order to ensure the corresponding $\rho$ is $C^1$ as required for a classical solution.
    
     Later in Theorem \ref{thm:regularity-classical} we shall see that the $C^2$ regularity of $Q$ can be ensured, as long as the initial data is in $C^2$ and is compatible with the boundary conditions in a certain sense.
\end{remark}

\begin{remark}\label{rmk:initial-boundary-rho}
 Let us reinterpret \eqref{ptau-diff-0-1} in the $\rho$ formulation \eqref{eq:rho-t-1}-\eqref{ic-rho-t}. By \eqref{def-N-t}-\eqref{bc-rho-t} we have
    \begin{align*}
        (1+K(0)N(t))\rho(t,0)=N(t),\qquad \mbox{and} \qquad
        (1+K(\Phi_F)N(t))\rho(t,\Phi_F)=N(t).
    \end{align*} Hence dividing the above equations by $N(t)$ we derive
    \begin{align}
    \frac{1}{N(t)}+K(0)=\frac{1}{\rho(t,0)},\qquad \mbox{and} \qquad
    \frac{1}{N(t)}+K(\Phi_F)=\frac{1}{\rho(t,\Phi_F)},
    \end{align}
a combination of which gives
    \begin{equation}
        \frac{1}{\rho(t,0)}-\frac{1}{\rho(t,\Phi_F)}=K(0)-K(\Phi_F).
    \end{equation} Using \eqref{tmp-pqpe}, we see at least formally the above three equations correspond to \eqref{ptau-bc-0}, \eqref{def-N-tau-classical} and \eqref{ptau-diff-0-1}. Proposition \ref{prop:keep-mono} tells us that \eqref{ptau-bc-0}-\eqref{ptau-diff-0-1} are implicitly preserved although not explicitly stated by dynamics of $Q$.
\end{remark}


\section{Well-posedness via a relaxed problem}\label{sec:wps-classical}

In this section, we study the well-posedness and regularity for classical solutions to the reformulation \eqref{eq-Q-tau-classical}-\eqref{ic-q-tau-classical}. The key ingredient is to introduce a relaxed problem, which will also be useful in later analysis in Section \ref{sec:global}.

To state the well-posedness result for classical solutions, we need the following assumption on initial data.
\begin{assumption}\label{as:wps}
    We assume that the initial data $Q_{\init}\in C^1[0,1]$ is non-decreasing and satisfies $Q_{\init}(0)=0,Q_{\init}(1)=\Phi_F$ as well as  
    \begin{equation}\label{first-order-in-assumption}
                    \frac{d}{d\eta} Q_{\init}(1)>K(\Phi_F),\qquad
            \frac{d}{d\eta} Q_{\init}(1)-\frac{d}{d\eta} Q_{\init}(0)=K(\Phi_F)-K(0).
    \end{equation}
\end{assumption}
\begin{remark}\label{rmk:As2-Ninit}
    Assumption \ref{as:wps} requires that the initial data is compatible with the boundary conditions up to the first order derivative, c.f. Corollary \ref{cor:5-equivalent}. We note that \eqref{first-order-in-assumption} is equivalent to that there exists some $N_{\init}>0$ such that
        \begin{align*}
            \frac{d}{d\eta} Q_{\init}(1)=\frac{1}{N_{\init}}+K(\Phi_F),\qquad
                        \frac{d}{d\eta} Q_{\init}(0)=\frac{1}{N_{\init}}+K(0).
        \end{align*}
\end{remark}
For the phase response function $K$, we recall Assumption \ref{as:global-K} is supposed throughout this paper.

Now we state the main results of this section. First we have the existence and uniqueness of a classical solution to \eqref{eq-Q-tau-classical}-\eqref{ic-q-tau-classical}, up to a maximal existence time $\tau^*$, which can be viewed as the first blow-up time of $N(\tau)$.

\begin{theorem}\label{thm:wps-blow-up-classical}
Suppose the initial data satisfies Assumption \ref{as:wps}. Then there exists a unique classical solution $(Q,N)$ to \eqref{eq-Q-tau-classical}-\eqref{ic-q-tau-classical} as in Definition \ref{def:classical-q-tau} on time interval $[0,\tau^*)$, where $0<\tau_*\leq+\infty$ is the maximal existence time.

In addition, $\tau_*$ can be viewed as the first blow-up time of $N(\tau)$. More precisely, if $\tau^*<+\infty$ then
        \begin{equation}\label{limit-N-blowup}
            \lim_{\tau\rightarrow (\tau^*)^-}N(\tau)=+\infty.
        \end{equation}
\end{theorem}

Next, we show that the solution depends continuously on the initial data in $C^1$ norm, up to the first blow-up time.
\begin{theorem}\label{thm:stability-classical}
    
    Let $Q_{\init,\eps}$ be a family of initial data indexed by $\eps\geq0$, satisfying Assumption \ref{as:wps}, and such that
    \begin{equation}
        Q_{\init,\eps}(\cdot)\rightarrow Q_{\init,0}(\cdot),\qquad \text{in $C^1[0,1]$,}\qquad \text{as $\eps\rightarrow0^+.$}
    \end{equation}
    Denote the corresponding classical solution to \eqref{eq-Q-tau-classical}-\eqref{ic-q-tau-classical} as $(Q_{\eps},N_{\eps})$ with maximal existence times $\tau^*_{\eps}$, as in Theorem \ref{thm:wps-blow-up-classical}. In particular, $Q_0,N_0$ and $\tau^*_0$ correspond to the initial data $Q_{\init,0}$. Then we have
    \begin{equation}\label{suplimit-blow-up-time}
        \liminf_{\eps\rightarrow0^+}\tau^*_{\eps}\geq \tau^*_0.
    \end{equation} Moreover, for every $0<T<\tau^*_0$, we have
    \begin{equation}\label{convergence-N}
        N_{\eps}(\tau)\rightarrow N_0(\tau)\qquad \text{in $C[0,T]$,}\qquad \text{as $\eps\rightarrow0^+,$}
    \end{equation} and
    \begin{equation}\label{convergence-Q}
         Q_{\eps}(T,\cdot)\rightarrow Q_0(T,\cdot)\qquad \text{in $C^1[0,1]$,}\qquad \text{as $\eps\rightarrow0^+.$}
    \end{equation}
\end{theorem}

Finally, we can show $C^2$ regularity of $Q$, provided that the initial data satisfies the following additional assumption.
\begin{assumption}\label{as:second-order}
    We assume $Q_{\init}(\eta)\in C^2[0,1]$ with
    \begin{equation}\label{second-compatible}
        \frac{d^2}{d\eta^2}Q_{\init}(1)-\frac{d^2}{d\eta^2}Q_{\init}(0)=K'(\Phi_F)\frac{d}{d\eta}Q_{\init}(1)-K'(0)\frac{d}{d\eta}Q_{\init}(0).
    \end{equation}
\end{assumption}
\begin{remark}
    Equation \eqref{second-compatible} is a second order condition for the initial data to be compatible with the boundary conditions (c.f. the first order condition \eqref{first-order-in-assumption} in Assumption \ref{as:wps}). It will be explained later in Section \ref{subsubsec:C2}.
\end{remark}
\begin{theorem}\label{thm:regularity-classical}
    In the same setting as Theorem \ref{thm:wps-blow-up-classical}, suppose the initial data additionally satisfies Assumption \ref{as:second-order}. Then we have
    $
        N(\tau)\in C^1[0,\tau^*)
    $
    and $Q\in C^2([0,\tau^*)\times[0,1])$.

\end{theorem}

\subsection{The relaxed problem}\label{subsec:relax}

\subsubsection{Definition and basic propetries}
A main challenge of \eqref{eq-Q-tau-classical} is the term $1/N(\tau)$, which is required to be positive in \eqref{constaint-classical-Q-tau}. To study the well-posedness, we first consider an auxiliary problem where this positive constraint is removed. More precisely, we introduce $\tilde{N}(\tau)\in\mathbb{R}$ in place of $1/N(\tau)>0$. The relaxed system is given as follows
    \begin{align}\label{tilde-eq-classical}
        &\p_{\tau}Q+\p_{\eta}Q=\tilde{N}(\tau)+{K}(Q), &&\tau>0,\ \eta\in(0,1),\\
    &Q(\tau,0)=0, &&\tau>0,\label{tilde-bc-classical}\\
     &{\tilde{N}(\tau)}=\p_{\eta}Q(\tau,1)-{K}(\Phi_F), &&\tau>0,\label{def-tildeN-tau-classical}\\
     &Q(\tau=0,\eta)=Q_{\init}(\eta), &&\eta\in[0,1].\label{tilde-ic-classical}
    \end{align} 
Compared to the original system \eqref{eq-Q-tau-classical}, the only difference is that we use $\tilde{N}(\tau)$, which is allowed to take any value in $\mathbb{R}$, to replace $1/N(\tau)>0$. Hence, under such a relaxation, we no longer need the constraint \eqref{constaint-classical-Q-tau} on $\p_{\eta}Q$. While a negative $\tilde{N}$ is not physical, it is convenient to study the well-posedness for the relaxed problem \eqref{tilde-eq-classical}-\eqref{tilde-bc-classical}-\eqref{def-tildeN-tau-classical}-\eqref{tilde-ic-classical} as a first step.

The definition of a classical solution to  \eqref{tilde-eq-classical}-\eqref{tilde-ic-classical} is given as follows.

\begin{definition}[Classical solution to the relaxed problem \eqref{tilde-eq-classical}-\eqref{tilde-ic-classical}]\label{def:tilde-classical-q-tau} Given $0<T\leq +\infty$, and $C^1$ initial data $Q_{\init}(\eta)$ with $Q_{\init}(0)=0,Q_{\init}(1)=\Phi_F$. We say $(Q,\tilde{N})$ is a classical solution to \eqref{tilde-eq-classical}-\eqref{tilde-ic-classical} on $[0,T)$ if 
\begin{enumerate}
    \item $Q(\tau,\eta)\in C^{1}([0,T)\times[0,1])$ and ${\tilde{N}}(\tau)\in C[0,T)$.
    \item Equations  \eqref{tilde-eq-classical}-\eqref{tilde-ic-classical} are satisfied in the classical sense for $\tau\in[0,T)$. 
\end{enumerate}
\end{definition}

We have an equivalent characterization of $\tilde{N}$, whose proof is the same as Proposition \ref{prop:char-N-classical} for $1/N(\tau)$. 

\begin{proposition}\label{prop:char-tildeN-classical}
    Suppose $Q(\tau,\eta)\in C^{1}([0,T)\times[0,1])$ and $\tilde{N}(\tau)\in C[0,T)$ satisfies \eqref{tilde-eq-classical}, with an initial condition \eqref{tilde-ic-classical} satisfying $Q_{\init}(1)=\Phi_F$. Then the expression \eqref{def-tildeN-tau-classical} for $\tilde{N}$  on $[0,T)$ is equivalent to 
    \begin{equation}\label{tilde-bc-classical-eta1}
        Q(\tau,1)\equiv \Phi_F,\qquad \tau\in[0,T).
    \end{equation}
\end{proposition}

For later use, we also list the following properties of the relaxed problem \eqref{tilde-eq-classical}-\eqref{tilde-ic-classical}, whose proof is similar to Corollary \ref{cor:5-equivalent} for the original problem \eqref{eq-Q-tau-classical}-\eqref{ic-q-tau-classical}.
\begin{proposition}\label{prop:tildeN-imbc}
Let $(Q,\tilde{N})$ be a classical solution \eqref{tilde-eq-classical}-\eqref{tilde-ic-classical}. Then it satisfies \eqref{tilde-bc-classical-eta1} and  
\begin{equation}
    \p_{\eta}Q(\tau,0)=\tilde{N}(\tau)+K(0),\qquad         \p_{\eta}Q(\tau,1)-\p_{\eta}Q(\tau,0)=K(\Phi_F)-K(0).
\end{equation} 
\end{proposition}

\subsubsection{Well-posedness and regularity: statements}
For the relaxed problem \eqref{tilde-eq-classical}-\eqref{tilde-ic-classical}, we can prove its \textit{global} well-posedness and regularity.

\begin{theorem}\label{thm:wps-classical-relax}
Suppose the initial data satisfies Assumption \ref{as:wps}. Then there exists a unique classical solution $(Q,\tilde{N})$ to the relaxed system \eqref{tilde-eq-classical}-\eqref{tilde-ic-classical} as in Definition \ref{def:tilde-classical-q-tau}. The solution is global, i.e. it exists on the time interval $[0,+\infty)$.
\end{theorem}
\begin{theorem}\label{thm:stability-classical-relax}    
    Let $Q_{\init,\eps}$ be a family of initial data indexed by $\eps\geq0$, satisfying Assumption \ref{as:wps}, and such that
    \begin{equation}
        Q_{\init,\eps}(\cdot)\rightarrow Q_{\init,0}(\cdot),\qquad \text{in $C^1[0,1]$,}\qquad \text{as $\eps\rightarrow0^+.$}
    \end{equation}
    Denote the corresponding classical solution to the relaxed problem \eqref{tilde-eq-classical} as $(Q_{\eps},\tilde{N}_{\eps})$, as in Theorem \ref{thm:wps-classical-relax}. In particular, $Q_0$ and $\tilde{N}_0$ correspond to the initial data $Q_{\init,0}$. Then we have for every $0<T<+\infty$
    \begin{equation}\label{convergence-tilde}
        \tilde{N}_{\eps}(\tau)\rightarrow \tilde{N}_0(\tau)\qquad \text{in $C[0,T]$,}\qquad \text{as $\eps\rightarrow0^+,$}
    \end{equation} and
    \begin{equation}\label{convergence-Q-relax}
         Q_{\eps}(T,\cdot)\rightarrow Q_0(T,\cdot)\qquad \text{in $C^1[0,1]$,}\qquad \text{as $\eps\rightarrow0^+.$}
    \end{equation}
\end{theorem}

\begin{theorem}\label{thm:regularity-classical-relax}
    In the same setting as Theorem \ref{thm:wps-classical-relax}, suppose the initial data additionally satisfies Assumption \ref{as:second-order}. Then we have
    $
        \tilde{N}(\tau)\in C^1[0,+\infty)
$ and 
   $
        Q\in C^2([0,+\infty)\times[0,1])$.
\end{theorem}

We postpone the proofs of Theorem \ref{thm:wps-classical-relax}-\ref{thm:regularity-classical-relax} to Section \ref{subsec:mild-relax} and \ref{subsec:regular-relax}.  In the following, we explore the connections between the relaxed problem and the original problem.
\subsubsection{Connections to the original problem}\label{subsubsec:relation-relax-original}

The classical solution to the relaxed problem \eqref{tilde-eq-classical}-\eqref{tilde-ic-classical} (Definition \ref{def:tilde-classical-q-tau}) connects to the original one \eqref{eq-Q-tau-classical}-\eqref{ic-q-tau-classical} (Definition \ref{def:classical-q-tau}) as in the following lemma.

\begin{lemma}\label{lem:relation-relaxed-original}
    (i) Suppose $(Q,N)$ is a classical solution to the original problem \eqref{eq-Q-tau-classical}-\eqref{ic-q-tau-classical} on the time interval $[0,\tau)$. Then with $\tilde{N}:=1/N$ the pair $(Q,\tilde{N})$ is a classical solution to the relaxed problem \eqref{tilde-eq-classical}-\eqref{tilde-ic-classical} on the same time interval $[0,\tau)$.

    (ii) Suppose $(Q,\tilde{N})$ is a classical solution to the relaxed problem \eqref{tilde-eq-classical}-\eqref{tilde-ic-classical} on $[0,\tau)$. Moreover assume that $Q_{\init}(\eta)$ is increasing in $\eta$, and  that there exists $0<\tau^*\leq \tau$ such that $\tilde{N}>0$ on $[0,\tau^*)$. Then with $N:=1/\tilde{N}$ the pair $(Q,N)$ gives a classical solution to the original problem \eqref{eq-Q-tau-classical}-\eqref{ic-q-tau-classical}, on the time interval $[0,\tau^*)$.
\end{lemma}
\begin{proof}
    For (i), it directly follows from Definition \ref{def:classical-q-tau} and \ref{def:tilde-classical-q-tau}. For (ii), it is similar. We just note that the positivity of $\tilde{N}$ on $[0,\tau^*)$ allows $N:=1/\tilde{N}$ to be well-defined, and that the two assumptions: $Q_{\init}(\eta)$ is increasing in $\eta$, and  that there exists $0<\tau^*\leq \tau$ such that $\tilde{N}>0$ on $[0,\tau^*)$, are enough to ensure that $Q(s,\eta)$ is increasing in $\eta$ at each time $s\in[0,\tau^*)$, thanks to Proposition \ref{prop:keep-mono}.
\end{proof}

Note that in the second part of Lemma \ref{lem:relation-relaxed-original}, $\tau^*$ can be smaller than $\tau$. We can recover the original problem from the relaxed one only when $\tilde{N}$ stays positive.

\begin{remark}
    Both $\tilde{N}$ and $1/N$ can be understood as some Lagrange multipliers to ensure the additional boundary condition $Q(\tau,1)=\Phi_F$, as shown in Proposition \ref{prop:char-N-classical} and \ref{prop:char-tildeN-classical}. However, in the original problem we require $1/N$ to be positive \eqref{constaint-classical-Q-tau}, while here $\tilde{N}$ is allowed to take any values in $\mathbb{R}$, including negative ones.

    By allowing negative values for $\tilde{N}$, the relaxed problem can be continued after the blow-up of $N$. However, such a continuation might no longer represent the dynamics of pulse-coupled oscillators. Indeed, in the relaxed problem $Q$ can be non-monotone after $\tilde{N}$ being negative. This implies that $Q$ no longer correspond to a pseudo-inverse for a probability distribution. See Figure \ref{fig:sec32} for an illustration.  
\end{remark}
\begin{figure}[htbp]
    \centering    
    \includegraphics[width=1.0\linewidth]{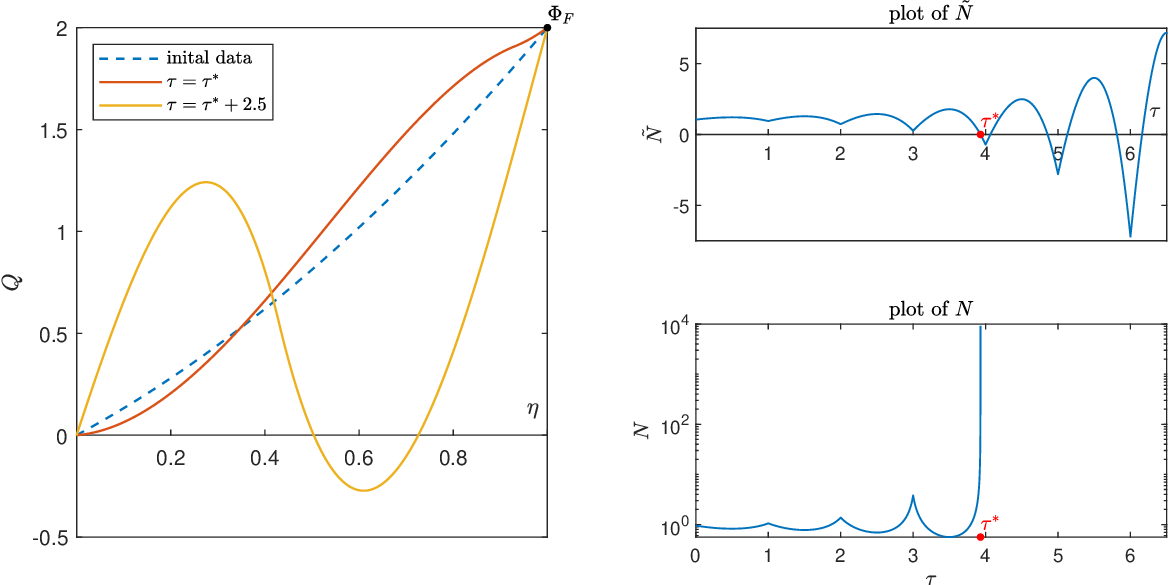}
    \caption{Illustration of the relaxed problem \eqref{tilde-eq-classical}-\eqref{tilde-ic-classical} and its relation to the original problem \eqref{eq-Q-tau-classical}-\eqref{ic-q-tau-classical}. Right: Plots of $\tilde{N}$ and $N$ in log scale. In the original problem, the firing rate $N$ blows up at $\tau^*$. In the relaxed problem, allowed to take negative values, $\tilde{N}$ is continued after $\tau^*$. Left: Profiles of $Q$ at different times.  Note that at $\tau=\tau^*+2.5$, $Q$ is non-monotone, which implies that it no longer correspond to a pseudo-inverse. Here $K(\phi)=0.75\phi+0.2$.}
    \label{fig:sec32}
\end{figure}

\subsection{Well-posedness of the original problem}
Theorem \ref{thm:wps-classical-relax} shows that the solution to the relaxed problem \eqref{tilde-eq-classical}-\eqref{tilde-ic-classical} is global, but it may not correspond to the physical dynamics of the pulse-coupled oscillator whenever $\tilde{N}$ touches zero or becomes negative. Nevertheless, we can use the well-posedness results for the relaxed problem (Theorem \ref{thm:wps-classical-relax}-\ref{thm:regularity-classical-relax}) to prove
 the well-posedness of the original problem (Theorem \ref{thm:wps-blow-up-classical}-\ref{thm:regularity-classical}).

First, we use Theorem \ref{thm:wps-classical-relax} to derive the maximal existence and blow-up criteria of the original problem \eqref{eq-Q-tau-classical}-\eqref{ic-q-tau-classical}, and therefore prove Theorem \ref{thm:wps-blow-up-classical}.

\begin{proof}[Proof of Theorem \ref{thm:wps-blow-up-classical}]
\ \newline
\indent
    \textit{1. Uniqueness.} Suppose there are two solutions to the original problem. Then by Lemma \ref{lem:relation-relaxed-original}-(i) they are also solutions to the relaxed problem. Therefore, they must be the same, as the uniqueness of the relaxed problem is given in Theorem \ref{thm:wps-classical-relax}.

    \textit{2. Construct the maximal classical solution.}  Let $(Q,\tilde{N})$ be the solution to the relaxed problem with the same initial data, whose global existence is ensured by Theorem \ref{thm:wps-classical-relax}. Now we aim to find the first time $\tilde{N}(\tau)$ touching zero,  defined via
    \begin{equation}\label{def-tau*-aszero}
        \tau^*:=\inf\{\tau\geq 0:\, \tilde{N}(\tau)=0\}.
    \end{equation} 
    
    We have $\tau^*>0$ and $\tilde{N}(\tau)>0$ for all $\tau$ in $[0,\tau^*)$, which follows from the continuity and that initially $\tilde{N}(0)>0$ by Assumption \ref{as:wps} (c.f. Remark \ref{rmk:As2-Ninit}). Hence by Lemma \ref{lem:relation-relaxed-original}-(ii) we can construct a classical solution to the original problem on $[0,\tau^*)$, with the same $Q$ and $N=1/\tilde{N}$. 

    When $\tau^*=+\infty$, we have constructed the global classical solution to the original problem.
    
    When otherwise $0<\tau^*<\infty$, we have $\tilde{N}(\tau^*)=0$, hence we derive that the firing rate blows up at $\tau^*$
    \begin{equation}
        \lim_{\tau\rightarrow(\tau^*)^{-}}N(\tau)= \lim_{\tau\rightarrow(\tau^*)^{-}}\frac{1}{\tilde{N}(\tau)}=+\infty.
    \end{equation} In this case, we have constructed a classical solution on $[0,\tau^*)$ with $\tau^*$ being the first blow-up time of the firing rate $N(\tau)$.
\end{proof}We summarize the construction in the proof of Theorem \ref{thm:wps-blow-up-classical} as the following corollary.

\begin{corollary}\label{cor:constructed-original}
Suppose the initial data satisfies Assumption \ref{as:wps}. Then there exists a unique global classical solution $(Q,\tilde{N})$ to the relaxed problem \eqref{tilde-eq-classical}-\eqref{tilde-ic-classical}. Define $\tau^*>0$ as in \eqref{def-tau*-aszero}. Then, restricting to $[0,\tau^*)$, with $N:=1/\tilde{N}$, $(Q,N)$ gives a classical solution to the original problem \eqref{eq-Q-tau-classical}-\eqref{ic-q-tau-classical} with $\tau^*$ as its maximal existence time.
\end{corollary}
Together with Theorem \ref{thm:stability-classical-relax}-\ref{thm:regularity-classical-relax}, we can obtain the continuous dependence on initial data and regularity before the blow-up, and thus prove Theorem \ref{thm:stability-classical}-\ref{thm:regularity-classical}.
\begin{proof}[Proof of Theorem \ref{thm:stability-classical}]
    We follow the construction in Corollary \ref{cor:constructed-original}.

    Let $(Q_{\eps},\tilde{N}_{\eps})$ be the corresponding solution to the relaxed problem \eqref{tilde-eq-classical}-\eqref{tilde-ic-classical}. In particular, $\tau_0^*$ is the first time that $\tilde{N}_0$ touches zero and $\tilde{N}_0(\tau)>0$ for all $0\leq\tau<\tau_0^*$.  Then for every $0<T<\tau^*_0$, $\tilde{N}_0$ has a uniform positive lower bound on $[0,T]$, which allows us to use \eqref{convergence-tilde} in Theorem \ref{thm:stability-classical-relax} to deduce that $\tilde{N}_{\eps}(\tau)>0$ on $[0,T]$ for $\eps$ small enough. Hence  we derive
    \begin{equation}
                \liminf_{\eps\rightarrow0^+}\tau^*_{\eps}\geq T,
    \end{equation} for every $0<T<\tau^*_0$,  which implies \eqref{suplimit-blow-up-time}.

    Then by the construction $N_{\eps}=1/{\tilde{N}}_{\eps}$, \eqref{convergence-N} follows from \eqref{convergence-tilde} and the uniform positive lower bound on $[0,T]$. And \eqref{convergence-Q} follows from \eqref{convergence-Q-relax}.

\end{proof}
\begin{proof}[Proof of Theorem \ref{thm:regularity-classical}]
The proof is direct, as the solution to the original problem is constructed via a restriction in time by Corollary \ref{cor:constructed-original}, which will inherit the regularity for the relaxed problem in Theorem \ref{thm:regularity-classical-relax}.
\end{proof}

Finally, we note that when a classical solution to the original problem blows up, the profile $Q$ still has a well-defined limit towards the blow-up time $\tau^*$. This is  a direct consequence of the construction in Corollary \ref{cor:constructed-original}.

\begin{corollary}\label{cor:pre-blow-up-Q}
    Let $(Q,N)$ be a classical solution to \eqref{eq-Q-tau-classical}-\eqref{ic-q-tau-classical} on the maximal existence interval $[0,\tau^*)$ as in Theorem \ref{thm:wps-blow-up-classical}. If $\tau^*<\infty$, then the pre-blow-up profile
    \begin{equation}
Q((\tau^*)^-,\cdot):=\lim_{\tau\rightarrow(\tau^*)^{-}}Q(\tau,\cdot),\qquad \text{where the limit is in $C^1[0,1]$,}
    \end{equation} is a well-defined  function. Moreover, it satisfies $
\p_{\eta}Q((\tau^*)^-,1)=K(\Phi_F)$.
\end{corollary}

\subsection{Relaxed problem: mild solution}\label{subsec:mild-relax}

In what follows, we study the well-posedness and regularity for the relaxed problem and provide proofs for Theorem \ref{thm:wps-classical-relax}-\ref{thm:regularity-classical-relax}. We take a ``bottom-up'' approach, starting with simpler problems for which the existence is more directly available, then step by step move to the full relaxed problem \eqref{tilde-eq-classical}-\eqref{tilde-ic-classical} and recover its regularity.

We start with a further simplified problem, replacing $\tilde{N}(\tau)$ in \eqref{tilde-eq-classical} by an external source $f(\tau)$. 
\begin{equation}\label{simplified-1-external-regular}
    \begin{aligned}
        &\p_{\tau}Q+\p_{\eta}Q=f(\tau)+{K}(Q),\qquad\qquad &&\tau>0,\ \eta\in(0,1),\\
    &Q(\tau,0)=0,\qquad\qquad &&\tau>0,\\
     &Q(\tau=0,\eta)=Q_{\init}(\eta),\qquad\qquad &&\eta\in[0,1].
    \end{aligned}
\end{equation} Equation \eqref{simplified-1-external-regular} is still nonlinear due to a possibly nonlinear $K(Q)$. But we can already solve it via the characteristics when $f$ is given. This motivates us to define the flow map and the mild solution as follows.

For a given (locally bounded) function $f$, we introduce the flow map $\Psi_{\tau_s\rightarrow\tau}^f(x)$, starting at time $\tau_s$ with initial position $x$, defined via the solution to
\begin{equation}\label{flow-map-extra-f-K}
    \begin{aligned}
        \frac{d}{d\tau}\Psi_{\tau_s\rightarrow\tau}^f(x)&={K}(\Psi_{\tau_s\rightarrow\tau}^f(x))+f(\tau),\quad \tau>\tau_s,\\
        \Psi_{\tau_s\rightarrow\tau_s}^f(x)&=x,\quad x\geq0.
    \end{aligned}
\end{equation}
The flow map is well-defined thanks to Assumption \ref{as:global-K} on $K$. Then we can define the mild solution to \eqref{simplified-1-external-regular}.
\begin{definition}[Mild solution to external source problem \eqref{simplified-1-external-regular}]\label{def:mild-1-external-regular}
Given a pointwise-defined initial data $Q_{\init}(\eta)$ with $Q_{\init}(0)=0$ and $f(\tau)$ a locally bounded function in time, a mild solution to \eqref{simplified-1-external-regular} is a pointwisely defined function $Q(\tau,\eta)$ given by
\begin{equation}\label{expression-mild}
    Q(\tau,\eta)=\begin{dcases}\Psi_{\tau-\eta\rightarrow\tau}^f(0),\qquad  \tau\geq \eta\geq 0,\\
    \Psi_{0\rightarrow\tau}^f(Q_{\init}(\eta-\tau)), \qquad 0\leq \tau\leq \eta\leq 1,
    \end{dcases}
\end{equation} where $\Psi^f_{\tau_s\rightarrow\tau}$ is the flow map associated with $f$ as defined in \eqref{flow-map-extra-f-K}.
\end{definition}

Now we define the mild solution to the full relaxed problem \eqref{tilde-eq-classical}-\eqref{tilde-ic-classical}. Compared to Definition \ref{def:mild-1-external-regular} we need to determine the correct external source $\tilde{N}$. 

\begin{definition}[Mild solution to the relaxed problem \eqref{tilde-eq-classical}]\label{def:mild-2-full-regular}
Given initial data $Q_{\init}(\eta)$ satisfying $Q_{\init}(0)=0$ and $Q_{\init}(1)=\Phi_F$, a mild solution to \eqref{tilde-eq-classical}-\eqref{tilde-ic-classical} is a pair $(Q,\tilde{N})$ where $\tilde{N}$ is a locally bounded function and $Q$ is the corresponding mild solution to \eqref{simplified-1-external-regular} with $f(\tau) = {\tilde N(\tau)}$ as the external source, such that
\begin{equation}\label{sec32-condition-equiv-tau>0-}
    Q(\tau,1)\equiv \Phi_F,\qquad \forall \tau>0.
\end{equation}
\end{definition} 

Compared to Definition \ref{def:tilde-classical-q-tau} of the classical solution, we require less regularity. Here we use \eqref{sec32-condition-equiv-tau>0-} to determine $\tilde{N}$ instead of \eqref{def-tildeN-tau-classical}. Note that the condition \eqref{sec32-condition-equiv-tau>0-} does not involve $\p_{\eta}Q$ in contrast to \eqref{def-tildeN-tau-classical}. We know these two conditions are equivalent for the classical solution by Proposition \ref{prop:char-tildeN-classical}. In particular a classical solution in Definition \ref{def:tilde-classical-q-tau} is a mild solution here.

As a first step towards Theorem \ref{thm:wps-classical-relax}, we prove the well-posedness to \eqref{tilde-eq-classical}-\eqref{tilde-ic-classical} in terms of the mild solution.

\begin{proposition}[Existence and Uniqueness for the mild solution to the relaxed problem]\label{Prop:wps-mild-solution}
    Given initial data with the regularity $Q_{\init}(\eta)\in W^{1,\infty}(0,1)$ satisfying $Q_{\init}(0)=0$ and $Q_{\init}(1)=\Phi_F$, there exists a unique mild solution $(Q,\tilde{N})$ to the relaxed problem \eqref{tilde-eq-classical} in the sense of Definition \ref{def:mild-2-full-regular}, which is global in time. 
    Moreover, for each $\tau>0$, $\p_{\eta}Q(\tau,\cdot)$ is in $ W^{1,\infty}(0,1)$ with growth estimates
    \begin{align}\label{growth-Qw1infty}
                \|\p_{\eta}Q(\tau,\cdot)\|_{L^{\infty}(0,1)}&\leq e^{K_1\tau}\left(\|\frac{d}{d\eta}Q_{\init}(\eta)\|_{L^{\infty}(0,1)}+C\tau+C\right),\qquad \tau>0,\\ \label{growth-tilde-N}
                \|\tilde{N}\|_{L^{\infty}(0,\tau)}&\leq e^{K_1\tau}\left(\|\frac{d}{d\eta}Q_{\init}(\eta)\|_{L^{\infty}(0,1)}+C\tau+C\right),\qquad\tau> 0,
    \end{align} where $K_1:=\|K'\|_{L^{\infty}(\mathbb{R})}>0$ and $C>0$ are constants independent of initial data.
    \end{proposition}

The following classical properties of the flow map $\Psi_{\tau_s\rightarrow\tau}^f(x)$ will be useful for the proof of Proposition \ref{Prop:wps-mild-solution}. Recall we take Assumption \ref{as:global-K} for $K$ throughout this paper. 
\begin{lemma}[Classical properties of the flow map]\label{lem:flow-map-property}
    Let $f$ be a locally bounded function. Then the flow map $\Psi_{\tau_s\rightarrow\tau}^f(x)$ defined in \eqref{flow-map-extra-f-K} has the following properties.
    \begin{enumerate}
        \item (Dependence on $x$) It is $C^1$ w.r.t. the initial position $x$ and the derivative is given by
        \begin{equation}
            \frac{\p}{\p x}(\Psi^f_{\tau_s\rightarrow\tau}(x))=\exp\left(\int_{\tau_s}^{\tau}K'(\Psi^f_{\tau_s\rightarrow t}(x))dt\right).
        \end{equation}
        \item (Dependence on $\tau_s$) It is absolutely continuous w.r.t. the starting time $\tau_s$, and the almost everywhere derivative is given by
        \begin{equation}
            \frac{\p}{\p \tau_s}(\Psi^f_{\tau_s\rightarrow\tau}(x))=-\Bigl(K(x)+f(\tau_s)\Bigr)\exp\left(\int_{\tau_s}^{\tau}K'(\Psi^f_{\tau_s\rightarrow t}(x))dt\right).
        \end{equation}
        \item (Stability in $f$) For two different external sources $f_1,f_2$, we have the following stability estimate
        \begin{equation}
            |\Psi^{f_1}_{\tau_s\rightarrow\tau}(x)-\Psi^{f_2}_{\tau_s\rightarrow\tau}(x)|\leq e^{C\Delta\tau}\Delta \tau\|f_1-f_2\|_{L^{\infty}(\tau_s,\tau)},\qquad \Delta \tau:=\tau-\tau_s\geq0,
        \end{equation} with a constant $C>0$ depending only on $K$.
    \end{enumerate}
\end{lemma}
The proof of Lemma \ref{lem:flow-map-property} is classical. 

\begin{proof}[Proof of Proposition \ref{Prop:wps-mild-solution}]
\ \newline \indent
\textit{Step 1. $\tau\in[0,1)$, reduce to a fixed point problem.} Denote the solution to the simplified problem \eqref{simplified-1-external-regular} with external source $f$ as $Q^f$. Then by Definition \ref{def:mild-1-external-regular} for $\tau\in[0,1)$, $Q^f(\tau,1)= \Psi_{0\rightarrow \tau}^f(Q_{\rm init}(1-\tau))$. Therefore on this time interval \eqref{sec32-condition-equiv-tau>0-} is equivalent to
\begin{equation}\label{tmp-mildwps1}
        \Psi_{0\rightarrow \tau}^f(Q_{\init}(1-\tau))\equiv \Phi_F,\quad \tau\in[0,1).
\end{equation}Due to the regularity assumed for $Q_{\init}$, $K$ and $f$, we note the map $\tau\rightarrow\Psi_{0\rightarrow \tau}^f(Q_{\rm init}(1-\tau))$ is absolutely continuous in $\tau$. Hence as $Q_{\init}(1)=\Phi_F$, taking the derivative in $\tau$, we see \eqref{tmp-mildwps1} is equivalent to

\begin{equation}\label{tmp-cond-1}
    \frac{d}{d\tau}\left(\Psi_{0\rightarrow \tau}^f(Q_{\init}(1-\tau))\right)\equiv 0,\quad \tau\in[0,1),
\end{equation} where the derivative is understood in the almost everywhere sense. 

 We calculate
\begin{align}\label{tmp-cal-1}
    \frac{d}{d\tau}\left(\Psi_{0\rightarrow \tau}^f(Q_{\init}(1-\tau))\right)&=\frac{\partial}{\p\tau}\Psi_{0\rightarrow \tau}^f(x)|_{x=Q_{\init}(1-\tau)}-\frac{\p}{\p x}\Psi_{0\rightarrow \tau}^f(Q_{\init}(1-\tau))\frac{d}{d\eta}Q_{\init}(1-\tau),
\end{align} where  we use the chain rule for the second term, relying on the regularity of $Q_{\init}$.

To treat the first term on the right hand side of \eqref{tmp-cal-1}, by definition of the flow map we have
\begin{align}
    \frac{\partial}{\p\tau}\Psi_{0\rightarrow \tau}^f(x)|_{x=Q_{\init}(1-\tau)}&=f(\tau)+K(\Psi_{0\rightarrow \tau}^f(Q_{\init}(1-\tau)))\\&=f(\tau)+K(Q^f(\tau,1)).\label{tmp-sec32-1}
\end{align} For the second term we use Lemma \ref{lem:flow-map-property}-1 to derive
\begin{align}
        \frac{\p}{\p x}(\Psi^f_{0\rightarrow\tau}(Q_{\init}(1-\tau)))&=\exp\left(\int_0^{\tau}K'(\Psi^f_{0\rightarrow s}(Q_{\init}(1-\tau)))ds\right)\\&=\exp\left(\int_0^{\tau}K'(Q^f(s,1-
        \tau+s))ds\right).\label{tmp-sec32-2}
\end{align}

Plugging the above expressions \eqref{tmp-sec32-1}-\eqref{tmp-sec32-2} into \eqref{tmp-cal-1}, we deduce that  \eqref{tmp-cond-1} is equivalent to the following \textit{fixed point problem}, for $\tau\in[0,1)$
\begin{equation}\label{fixed-point-1} 
    \begin{aligned}f(\tau)&=\exp\left(\int_0^{\tau}K'(Q^f(s,1-
        \tau+s))ds\right)\frac{d}{d\eta}Q_{\init}(1-\tau)-K(Q^f(\tau,1))\\&=:\mathcal{F}(f)(\tau).
\end{aligned}
\end{equation}
Recall here that $Q^f$ is the solution to the simplified problem \eqref{simplified-1-external-regular} with external source $f$, as defined in \eqref{def:mild-1-external-regular}. More explicitly we have for $0\leq s\leq \tau\leq 1$
\begin{align}\label{Qf-1}
    Q^f(s,1-\tau+s)=\Psi_{0\rightarrow s}^f(Q_{\init}(1-\tau)),\\
    Q^{f}(\tau,1)=\Psi_{0\rightarrow\tau}^f(Q_{\init}(1-\tau)).\label{Qf-2}
\end{align}

\textit{Step 2. $\tau\in[0,\Delta\tau)$, locally solve the fixed point problem.} Next, we show that the fixed point problem \eqref{fixed-point-1} can be locally solved via the Banach fixed point theorem. More precisely, we shall prove the following estimate, for all $0<\Delta\tau\leq 1$,
\begin{equation}\label{mild-wps-key}
    \|\mathcal{F}(f_1)-\mathcal{F}(f_2)\|_{L^{\infty}(0,\Delta \tau)}\leq C\Delta\tau\left(1+\left\|\frac{d}{d\eta}Q_{\init}(\eta)\right\|_{L^{\infty}(0,1)}\right)\|f_1-f_2\|_{L^{\infty}(0,\Delta \tau)},
\end{equation} where $C>0$ is a constant independent of the initial data. The estimate \eqref{mild-wps-key} implies that $\mathcal{F}$ is a contraction on $L^{\infty}(0,\Delta \tau)$ with $\Delta\tau$ given by, e.g.
\begin{equation}\label{cond-delta-tau}
    \Delta \tau=\min\left(1,\frac{1}{2C\left(1+\|\frac{d}{d\eta}Q_{\init}(\eta)\|_{L^{\infty}(0,1)}\right)}\right).
\end{equation} Hence, by the Banach fixed point theorem, on $(0,\Delta \tau)$ we can obtain a unique $f=:\tilde{N}$ as a fixed point to \eqref{fixed-point-1}, therefore a unique mild solution on that time interval.

Now we prove \eqref{mild-wps-key}. First we apply Lemma \ref{lem:flow-map-property}-3 to \eqref{Qf-1}-\eqref{Qf-2} to obtain
\begin{align}
     |Q^{f_1}(s,1-\tau+s)-Q^{f_2}(s,1-\tau+s)|&\leq Cs\|f_1-f_2\|_{L^{\infty}(0,s)}\\&\leq C\Delta\tau\|f_1-f_2\|_{L^{\infty}(0,\Delta\tau)},\qquad 0\leq s\leq \tau\leq \Delta\tau \leq 1,\label{tmp-Qf-1}\intertext{and}
    |Q^{f_1}(\tau,1)-Q^{f_2}(\tau,1)|&\leq C\tau\|f_1-f_2\|_{L^{\infty}(0,\tau)}\\&\leq C \Delta\tau\|f_1-f_2\|_{L^{\infty}(0,\Delta\tau)},\qquad 0\leq \tau\leq \Delta\tau \leq 1.\label{tmp-Qf-2}
\end{align}
Next we estimate the difference $\mathcal{F}(f_1)-\mathcal{F}(f_2)$. For the first term in \eqref{fixed-point-1}, we calculate
\begin{multline}
       \left|\exp\left(\int_0^{\tau}K'(Q^{f_1}(s,1-
        \tau+s))ds\right)\frac{d}{d\eta}Q_{\init}(1-\tau)-\exp\left(\int_0^{\tau}K'(Q^{f_2}(s,1-
        \tau+s))ds\right)\frac{d}{d\eta}Q_{\init}(1-\tau)\right|\leq \\ \|\frac{d}{d\eta}Q_{\init}(\eta)\|_{L^{\infty}} \left|\exp\left(\int_0^{\tau}K'(Q^{f_1}(s,1-
        \tau+s))ds\right)-\exp\left(\int_0^{\tau}K'(Q^{f_2}(s,1-
        \tau+s))ds\right)\right|\\ \leq C\|\frac{d}{d\eta}Q_{\init}(\eta)\|_{L^{\infty}}\left|\int_0^{\tau}K'(Q^{f_1}(s,1-
        \tau+s))ds-\int_0^{\tau}K'(Q^{f_2}(s,1-
        \tau+s))ds\right|.\label{tmp-Qf-1024}
\end{multline}
And using \eqref{tmp-Qf-1}, we have
\begin{align}
        \left|\int_0^{\tau}K'(Q^{f_1}(s,1-
        \tau+s))ds-\int_0^{\tau} \right.&\left. K'(Q^{f_2}(s,1-
        \tau+s))ds\right|\\&\leq C\tau \sup_{s\in[0,\tau]}| K'(Q^{f_1}(s,1-
        \tau+s))ds-K'(Q^{f_2}(s,1-
        \tau+s))|\\&\leq C\tau^2\|f_1-f_2\|_{L^{\infty}(0,\tau)}\\&\leq C\Delta\tau \|f_1-f_2\|_{L^{\infty}(0,\Delta\tau)},\qquad 0<\tau\leq \Delta\tau\leq 1.
\end{align} Hence together with \eqref{tmp-Qf-1024}, we derive for $\tau\leq \Delta\tau$
\begin{align}
     \left|\exp\left(\int_0^{\tau}K'(Q^{f_1}(s,1-
        \tau+s))ds\right)\frac{d}{d\eta}Q_{\init}(1-\tau)-\right.&\left.\exp\left(\int_0^{\tau}K'(Q^{f_2}(s,1-
        \tau+s))ds\right)\frac{d}{d\eta}Q_{\init}(1-\tau)\right|\\
        &\leq  C\Delta\tau  \|\frac{d}{d\eta}Q_{\init}(\eta)\|_{L^{\infty}}\|f_1-f_2\|_{L^{\infty}(0,\Delta\tau)}.\label{mildwps-key-1}
\end{align}

For the second term in \eqref{fixed-point-1}, we use \eqref{tmp-Qf-2} to derive for $\tau\leq \Delta\tau$
\begin{equation}\label{mildwps-key-2}
    |K(Q^{f_1}(\tau,1))-K(Q^{f_2}(\tau,1))|\leq  C \|{K'}\|_{\infty} \Delta\tau\|f_1-f_2\|_{L^{\infty}(0,\Delta\tau)} \leq  C\Delta\tau \|f_1-f_2\|_{L^{\infty}(0,\Delta\tau)}.
\end{equation}

Finally, as both terms in \eqref{fixed-point-1} are estimated, we combine \eqref{mildwps-key-1} and \eqref{mildwps-key-2} to obtain \eqref{mild-wps-key}. 

For later uses, we note that since \eqref{tmp-mildwps1} holds when $f$ solves the fixed point problem \eqref{fixed-point-1}, we can replace $Q^f(\tau,1)$ by $\Phi_F$ to derive
\begin{equation}\label{fixed-point-1-tmp} 
    f(\tau)=\exp\left(\int_0^{\tau}K'(Q^f(s,1-
        \tau+s))ds\right)\frac{d}{d\eta}Q_{\init}(1-\tau)-K(\Phi_F).
\end{equation}

\textit{Step 3. A Priori Growth Estimates.} In Step 2 we have shown that we can locally obtain a unique mild solution, on a small time interval depending on $\|\frac{d}{d\eta}Q_{\init}(\eta)\|_{L^{\infty}(0,1)}$, as in \eqref{cond-delta-tau}. To obtain a global solution, now we control the growth of the $L^{\infty}$ norm of the spatial derivative.

We first work for $\tau\in[0,1]$. For the desired solution, $f=\tilde{N}$ solves the fixed point problem \eqref{fixed-point-1} and therefore \eqref{fixed-point-1-tmp}, which allows us to estimate as follows
\begin{align}
        |\tilde{N}(\tau)|=|f(\tau)|&\leq\left|\exp\left(\int_0^{\tau}K'(Q^f(s,1-
        \tau+s))ds\right)\frac{d}{d\eta}Q_{\init}(1-\tau)\right|+\left|K(\Phi_F)\right|\\&\leq  e^{K_1\tau}\|\frac{d}{d\eta}Q_{\init}(\eta)\|_{L^{\infty}(0,1)}+ C,\qquad \tau\in[0,1],\label{S3-1}
\end{align} where we have used the boundedness of $K'$ with $K_1=\|K'\|_{L^{\infty}(\mathbb{R})}$. Then we calculate $\p_{\eta}Q$, using Definition \ref{def:mild-1-external-regular} and Lemma \ref{lem:flow-map-property}

\begin{equation}\label{formula-peta-flow}
    \p_{\eta}Q(\tau,\eta)=\begin{dcases}
        (K(0)+f(\tau-\eta))\exp\left(\int_{\tau-\eta}^{\tau}K'(Q^f(s,s-\tau+\eta))ds\right),\quad 0\leq\eta< \tau,\\
        \frac{d}{d\eta}Q_{\init}(\eta-\tau) \exp\left(\int_{0}^{\tau}K'(Q^f(s,s-\tau+\eta))ds\right),\quad\tau< \eta\leq 1.
    \end{dcases}
\end{equation} Note that at this stage whether $\p_{\eta}Q(\tau,\eta)$ is continuous at $\eta=\tau$ does not come into play as we look for the $L^{\infty}$ weak derivative.  For $0\leq \tau<\eta$, we estimate using \eqref{S3-1}
\begin{align}
    |\p_{\eta}Q(\tau,\eta)|&\leq (K(0)+|f(\tau-\eta)|)e^{K_1\eta}
\leq \left(C+e^{K_1(\tau-\eta)}\|\frac{d}{d\eta}Q_{\init}(\eta)\|_{L^{\infty}}\right)e^{K_1\eta}\leq e^{K_1\tau}\|\frac{d}{d\eta}Q_{\init}(\eta)\|_{L^{\infty}}+C.
\end{align}
For $\tau<\eta<1$, it is direct to obtain
\begin{align}
    |\p_{\eta}Q(\tau,\eta)|\leq\left|\frac{d}{d\eta}Q_{\init}(\eta-\tau) \right|e^{K_1\tau}\leq e^{K_1\tau}\|\frac{d}{d\eta}Q_{\init}(\eta)\|_{L^{\infty}}.
\end{align} All together we have
\begin{equation}\label{MS-1}
     \|\p_{\eta}Q(\tau,\cdot)\|_{L^{\infty}(0,1)}\leq e^{K_1\tau}\|\frac{d}{d\eta}Q_{\init}(\eta)\|_{L^{\infty}(0,1)}+C,\quad \tau\in[0,1].
\end{equation}
For $\tau>1$ we similarly have
\begin{equation}
    \|\p_{\eta}Q(\tau,\cdot)\|_{L^{\infty}(0,1)}\leq  e^{K_1}\|\p_{\eta}Q(\tau-1,\cdot)\|_{L^{\infty}(0,1)}+C,
\end{equation} 
which iteratively yields
\begin{equation}\label{MS-2}
        \|\p_{\eta}Q(\tau,\cdot)\|_{L^{\infty}(0,1)}\leq  e^{K_1[\tau]}\left(\|\p_{\eta}Q(\tau-[\tau],\cdot)\|_{L^{\infty}(0,1)}+C[\tau]\right),
\end{equation} where we use $[\tau]$ to denote the integer part of $\tau$. Combining \eqref{MS-2} and \eqref{MS-1} we obtain the desired estimate \eqref{growth-Qw1infty}.

For $\tilde{N}$, similar to \eqref{S3-1} we have for $\tau>1$
\begin{equation}
    |\tilde{N}(\tau)|\leq e^{K_1}\|\p_{\eta}Q(\tau-1,\cdot)\|_{L^{\infty}(0,1)}+C,
\end{equation} combining which with \eqref{growth-Qw1infty} gives \eqref{growth-tilde-N}.

\textit{Step 4. Towards a global solution.} In Step 2 we have shown the local well-posedness on $[0,\Delta\tau]$, with a timestep related to the $L^{\infty}$ norm of the spatial derivative \eqref{cond-delta-tau}. The bound \eqref{growth-Qw1infty} ensures that this $L^{\infty}$ norm, despite may grow exponentially as time goes to infinity, is uniformly bounded on every finite time interval $[0,T]$, which allows us to extend the solution towards $\tau=T$ for every $T<+\infty$. Therefore, we can obtain a global solution.

\end{proof}
As a corollary of the construction in Proposition \ref{Prop:wps-mild-solution}, we deduce the continuous dependence on initial data for the mild solution.
\begin{corollary}\label{cor:stability-mild}
    Let $Q_{\init,\eps}$ be a family of initial data in $W^{1,\infty}(0,1)$ indexed by $\eps\geq0$ satisfying $Q_{\init,\eps}(0)=0,Q_{\init,\eps}(1)=\Phi_F$ such that
    \begin{equation}\label{stability-mild-cond}
        Q_{\init,\eps}(\cdot)\rightarrow Q_{\init,0}(\cdot),\qquad \text{in $W^{1,\infty}(0,1)$,}\qquad \text{as $\eps\rightarrow0^+.$}
    \end{equation}
    Denote the corresponding mild solution to the relaxed problem \eqref{tilde-eq-classical} as $(Q_{\eps},\tilde{N}_{\eps})$, as constructed in Proposition \ref{Prop:wps-mild-solution}. In particular, $Q_0$ and $\tilde{N}_0$ correspond to the initial data $Q_{\init,0}$. Then we have for every $0<T<+\infty$
    \begin{equation}\label{convergence-tilde-mild}
        \tilde{N}_{\eps}(\tau)\rightarrow \tilde{N}_0(\tau)\qquad \text{in $L^{\infty}(0,T)$,}\qquad \text{as $\eps\rightarrow0^+,$}
    \end{equation} and
    \begin{equation}\label{convergence-Q-relax-mild}
         Q_{\eps}(T,\cdot)\rightarrow Q_0(T,\cdot)\qquad \text{in $W^{1,\infty}(0,1)$,}\qquad \text{as $\eps\rightarrow0^+.$}
    \end{equation}
\end{corollary}
\begin{proof}
\ \newline \indent    
    \textit{Step 1. Local convergence of $\tilde{N}$.}
    We shall use that $f=\tilde{N}$ satisfies the fixed point problem \eqref{fixed-point-1}
    \begin{equation}\label{fixed-point-initial} 
    \begin{aligned}f(\tau)&=\exp\left(\int_0^{\tau}K'(Q^f(s,1-
        \tau+s))ds\right)\frac{d}{d\eta}Q_{\init}(1-\tau)-K(Q^f(\tau,1))\\&=:\mathcal{F}(f,Q_{\init}),%
\end{aligned}
\end{equation} where we extend the notation of $\mathcal{F}$ to $\mathcal{F}(f,Q_{\init})$ for the dependence on $Q_{\init}$. Note that in \eqref{fixed-point-initial} $Q^f$ depends on both the external source $f$ and the initial data $Q_{\init}$.

We shall estimate how the fixed point $f$ of \eqref{fixed-point-initial} depend on $Q_{\init}$. Similar to the calculations behind \eqref{mild-wps-key}, it is tedious but straightforward to show for $0<\Delta\tau<1$
\begin{equation}\label{tmp-stability}
    \|\mathcal{F}(f,Q_{\init}^1)-\mathcal{F}(f,Q_{\init}^2)\|_{L^{\infty}(0,\Delta\tau)}\leq C\|Q_{\init}^1-Q_{\init}^2\|_{W^{1,\infty}(0,1)},
\end{equation} where the constant $C>0$ depends on $\|f\|_{L^\infty(0,1)}$ and $\max(\|Q_{\init}^1\|_{W^{1,\infty}(0,1)},\|Q_{\init}^2\|_{W^{1,\infty}(0,1)})$.

For $\Delta\tau$ small enough (depending on $\|Q_{\init}^i\|_{W^{1,\infty}(0,1)}$) we know from \eqref{mild-wps-key} that $\mathcal{F}(\cdot,Q_{\init}^i)$ are contractions on $L^{\infty}(0,\Delta\tau)$ with a rate, says, $0<\alpha<1$, for both $i=1,2$.

Now we estimate the distance between two fixed points given by
\begin{equation}
    f_i=\mathcal{F}(f_i,Q_{\init}^i),\qquad i=1,2.
\end{equation} We write
\begin{align*}
       f_1-f_2&= \mathcal{F}(f_1,Q_{\init}^1)-\mathcal{F}(f_2,Q_{\init}^2)\\&=\mathcal{F}(f_1,Q_{\init}^1)-\mathcal{F}(f_2,Q_{\init}^1)+\mathcal{F}(f_2,Q_{\init}^1)-\mathcal{F}(f_2,Q_{\init}^2),
\end{align*}
and derive using the contraction property of $f$ and \eqref{tmp-stability}
\begin{align}
    \|f_1-f_2\|_{L^{\infty}(0,\Delta\tau)}&\leq\|\mathcal{F}(f_1,Q_{\init}^1)-\mathcal{F}(f_2,Q_{\init}^1)\|_{L^{\infty}(0,\Delta\tau)}+\|\mathcal{F}(f_2,Q_{\init}^1)-\mathcal{F}(f_2,Q_{\init}^2)\|_{L^{\infty}(0,\Delta\tau)}\\&\leq \alpha \|f_1-f_2\|_{L^{\infty}(0,\Delta\tau)}+C\|Q_{\init}^1-Q_{\init}^2\|_{W^{1,\infty}(0,1)},
\end{align} which gives
\begin{equation}\label{stability-key}
    \|f_1-f_2\|_{L^{\infty}(0,\Delta\tau)}\leq\frac{C}{1-\alpha}\|Q_{\init}^1-Q_{\init}^2\|_{W^{1,\infty}(0,1)}.
\end{equation}

We apply \eqref{stability-key} to $Q_{\init,\eps}$ to derive that
\begin{equation}\label{stability-first-local}
     \tilde{N}_{\eps}(\tau)\rightarrow \tilde{N}_0(\tau)\qquad \text{in $L^{\infty}(0,\Delta\tau)$,}\qquad \text{as $\eps\rightarrow0^+.$}
\end{equation}

Note that here constants in \eqref{stability-key} $C>0,\alpha\in(0,1),\Delta\tau>0$ can be uniformly chosen, due to that the $W^{1,\infty}$ norm of $Q_{\init,\eps}$ is uniformly bounded thanks to \eqref{stability-mild-cond}, and that the $L^{\infty}$ norm of $\tilde{N}$ is controlled by \eqref{growth-tilde-N} in Proposition \ref{Prop:wps-mild-solution}. This justifies the application of \eqref{stability-key} to obtain \eqref{stability-first-local}.

    \textit{Step 2. Local convergence for $Q$.} With \eqref{stability-first-local}, using the expressions \eqref{expression-mild} and \eqref{formula-peta-flow} for $Q$ and $\p_{\eta}Q$, we see for every $0<\tau<\Delta\tau$
    \begin{equation}\label{stability-second-local}
                 Q_{\eps}(\tau,\cdot)\rightarrow Q_0(\tau,\cdot)\qquad \text{in $W^{1,\infty}(0,1)$,}\qquad \text{as $\eps\rightarrow0^+.$}
    \end{equation}

    \textit{Step 3. Towards infinity.} With the result in the second step, we can start at a new time, choose $Q(\tau,\cdot)$ as a new initial data and repeat the previous arguments.

    Similar to the proof of Proposition \ref{Prop:wps-mild-solution}, the only constraint on the choice of the time step is the growth of the $W^{1,\infty}$ norm of $Q(\tau,\cdot)$, which is under control thanks to \eqref{growth-Qw1infty}. Therefore we can repeat the above procedure towards infinity (that is, \eqref{convergence-tilde-mild} and \eqref{convergence-Q-relax-mild} hold for all $T<+\infty$). 
    
\end{proof}

We state the following characterization of $\tilde{N}$, which will be useful for the later proof of Proposition \ref{prop:regularity-classical}.

\begin{corollary}\label{cor:fp-mild}
    Let $(Q,\tilde{N})$ be a mild solution to \eqref{tilde-eq-classical}-\eqref{tilde-ic-classical}. Then $f(\tau)=\tilde{N}(\tau)$ satisfies the following fixed point equation for $\tau\in[0,1)$
    \begin{align}
        \label{tcor-fp1}
    f(\tau)&=\exp\left(\int_0^{\tau}K'(Q^f(s,1-
        \tau+s))ds\right)\frac{d}{d\eta}Q_{\init}(1-\tau)-K(\Phi_F),\\&=\exp\left(\int_0^{\tau}K'(\Psi_{0\rightarrow s}^f(Q_{\init}(1-\tau)))ds\right)\frac{d}{d\eta}Q_{\init}(1-\tau)-K(\Phi_F).\label{tcor-fp2}
    \end{align}
    Here $Q^f$ denotes the solution to the simplified problem \eqref{simplified-1-external-regular} with external source $f$. 
\end{corollary}
\begin{proof}    Here \eqref{tcor-fp1} is just \eqref{fixed-point-1-tmp}. And then \eqref{tcor-fp2} is obtained by plugging in \eqref{Qf-1}.
\end{proof}

\subsection{Relaxed problem: regularity}\label{subsec:regular-relax}

We proceed to improve the regularity of the mild solution obtained in the previous section, under more requirements on initial data.

\subsubsection{From mild to classical: Proof of Theorem \ref{thm:wps-classical-relax}-\ref{thm:stability-classical-relax}}
\begin{proposition}\label{prop:regularity-classical}
Suppose $Q_{\init}\in C^1[0,1]$ with $Q_{\init}(0)=0,Q_{\init}(1)=\Phi_F$ and suppose additionally
\begin{equation}\label{compatible-S33}
                \frac{d}{d\eta} Q_{\init}(1)-\frac{d}{d\eta} Q_{\init}(0)=K(\Phi_F)-K(0).
\end{equation}
Then the mild solution obtained in Proposition \ref{Prop:wps-mild-solution} is classical, in the sense of Definition \ref{def:tilde-classical-q-tau}, and thus it satisfies the properties in Proposition \ref{prop:tildeN-imbc}.
\end{proposition}
Compared to Proposition \ref{Prop:wps-mild-solution}, two more conditions on initial data are imposed: $C^1$ regularity instead of $W^{1,\infty}$, and the compatibility condition \eqref{compatible-S33}, which is \eqref{first-order-in-assumption} in Assumption \ref{as:wps}. Proposition \ref{prop:regularity-classical} claims that with these two additional conditions, the mild solution (Definition \ref{def:mild-2-full-regular}) as constructed in Proposition \ref{Prop:wps-mild-solution} is indeed a classical solution (Definition \ref{def:tilde-classical-q-tau}).

\begin{proof}[Proof of Proposition \ref{prop:regularity-classical}]

We first show the desired regularity for $\tau\in[0,1)$.

\textit{Step 1. Continuity of $\tilde{N}$.} We already know $\tilde{N}(\tau)$ is locally bounded in Proposition \ref{Prop:wps-mild-solution}. To show the continuity, we shall use the fixed point equation \eqref{tcor-fp2} in Corollary \ref{cor:fp-mild}. For convenience, we still use the notation of the external source $f$ with $f=\tilde{N}$. Note that as long as $f$ is locally bounded and $Q_{\init}$ is continuous, then $\Psi_{0\rightarrow s}^f(Q_{\init}(1-\tau))$ is also continuous with respect to $\tau$ by Lemma \ref{lem:flow-map-property}. Therefore we can readily check that the right-hand side in \eqref{tcor-fp2} is continuous in $\tau$, when we further know $\frac{d}{d\eta}Q_{\init}$ is continuous.  This gives the continuity of $f=\tilde{N}(\tau)$ on $[0,1)$.

\textit{Step 2. Regularity of $Q$.} Now we proceed to show that $Q(\tau,\eta)$ is $C^1([0,1)\times[0,1])$. By Definition \ref{def:mild-1-external-regular} for the mild solution, we see $Q^f$ is continuous when the initial data and the external source $f=\tilde{N}$ is continuous. 

Moreover, we can compute the spatial derivative as \eqref{formula-peta-flow}
\begin{equation}\label{formula-peta-flow-recall}
    \p_{\eta}Q(\tau,\eta)=\begin{dcases}
        (K(0)+f(\tau-\eta))\exp\left(\int_{\tau-\eta}^{\tau}K'(Q^f(s,s-\tau+\eta))ds\right),\quad 0\leq\eta< \tau,\\
        \frac{d}{d\eta}Q_{\init}(\eta-\tau) \exp\left(\int_{0}^{\tau}K'(Q^f(s,s-\tau+\eta))ds\right),\quad\tau< \eta\leq 1.
    \end{dcases}
\end{equation} For this piecewise expression, we need to check the continuity at $\tau=\eta$. Here the compatible condition \eqref{compatible-S33} comes into play. We derive, using \eqref{tcor-fp2} at $\tau=0$ and \eqref{compatible-S33}
\begin{align}
    K(0)+f(0)&=K(0)+\frac{d}{d\eta}Q_{\init}(1)-K(\Phi_F)=\frac{d}{d\eta}Q_{\init}(0),\label{compatible-tmp-pf}
\end{align} 
which implies that the expression \eqref{formula-peta-flow-recall} is continuous at $\eta=\tau$. Furthermore, together with the continuity of $Q^f$ we can check that \eqref{formula-peta-flow-recall} gives a continuous function on $[0,1)\times[0,1]$.
For the temporal derivative, by definition \eqref{expression-mild} we compute
$
    \p_{\tau}Q=K(Q)+\tilde{N}(\tau)-\p_{\eta}Q,
$ 
which gives both the regularity of $\p_{\tau}Q$ as the right hand side is a sum of three $C^1$ functions, and that the equation \eqref{tilde-eq-classical} is satisfied in the classical sense. From \eqref{tcor-fp1} and \eqref{formula-peta-flow-recall}, we also recover the relation \eqref{def-tildeN-tau-classical}.

\textit{Step 3. Beyond $[0,1)$.} Above, we have worked for $\tau\in[0,1)$. To extend the results for all $\tau\geq0$, it suffices to show that the conditions on initial data holds for $Q(\tau,\cdot)$ when $\tau\in[0,1)$. 

Indeed, we have already shown that $\p_{\eta}Q$ is $C^1$ and that the boundary values of $Q$ are prescribed in the construction of the mild solution. It only remains to check the compatibility condition \eqref{compatible-S33}, for which we compute as follows, using \eqref{formula-peta-flow-recall} and \eqref{tcor-fp1} 
\begin{align}
    \p_{\eta}Q(\tau,1)=\frac{d}{d\eta}Q_{\init}(1-\tau) \exp\left(\int_{0}^{\tau}K'(Q^f(s,s+\tau-1))ds\right)=f(\tau)+K(\Phi_F) =\p_{\eta}Q(\tau,0)-K(0)+K(\Phi_F).
\end{align}
Hence, we can start from e.g. $\tau=1/2$, and extend the arguments in Step 1-2 to obtain the regularity for  $\tau\in[1/2,3/2)$.  We can conclude the result for all $\tau\geq 0$ by iteration.

\end{proof}
\begin{remark}\label{rmk:classical-subtle}
Note that if the compatibility condition \eqref{compatible-S33} is initially violated, then there will be a discontinuity in $\p_{\eta}Q$, which propagates along the characteristic $\tau=\eta$ till $\tau=1$ since the compatibility condition in Proposition \ref{prop:tildeN-imbc} is not met. And then this discontinuity will appear again at $\eta=0$.
\end{remark}

Now we have enough preparations to prove Theorem \ref{thm:wps-classical-relax} and Theorem \ref{thm:stability-classical-relax}.
\begin{proof}[Proof of Theorem \ref{thm:wps-classical-relax}]
\ \newline \indent    
    \textit{1. Global existence.} The mild solution constructed in Proposition \ref{Prop:wps-mild-solution} is classical by Proposition \ref{prop:regularity-classical}, which gives the global existence.
    
    \textit{2. Uniqueness.} For two classical solutions with a same initial data, we can directly check thanks to Proposition \ref{prop:char-tildeN-classical} that they are also two mild solutions. Hence by Proposition \ref{Prop:wps-mild-solution} they must be the same.
\end{proof}

\begin{proof}[Proof of Theorem \ref{thm:stability-classical-relax}]
    It directly follows from Corollary \ref{cor:stability-mild}. Note that the convergence in $L^{\infty}(0,T)$ becomes in $C[0,T]$ etc., thanks to the regularity by Proposition \ref{prop:regularity-classical}.
\end{proof}

\subsubsection{From $C^1$ to $C^2$: Proof of Theorem \ref{thm:regularity-classical-relax}} \label{subsubsec:C2}

\begin{proposition}\label{Prop:C2}
In the same setting of Proposition \ref{prop:regularity-classical}, suppose additionally  $Q_{\init}\in C^2[0,1]$ with
 \begin{equation}\label{second-compatible-recall}
        \frac{d^2}{d\eta^2}Q_{\init}(1)-\frac{d^2}{d\eta^2}Q_{\init}(0)=K'(\Phi_F)\frac{d}{d\eta}Q_{\init}(1)-K'(0)\frac{d}{d\eta}Q_{\init}(0).
    \end{equation}
    Then we have better regularity for the solution: $ \tilde{N}(\tau)\in C^1[0,+\infty)$ and     $Q\in C^2([0,+\infty)\times[0,1])$.
\end{proposition}

Let us explain how a second order condition \eqref{second-compatible-recall} arises. It might be better to view $\p_{\eta}Q(\tau,\eta)=:Z(\tau,\eta)$ as the solution to the following first order system, as derived at least formally in \eqref{eq-petaQ-Z} for the proof of Proposition \ref{prop:keep-mono},
\begin{equation}\label{eq-petaQ-Z-recall}
        \begin{aligned}
            &(\p_{\tau}+\p_{\eta})Z=K'(Q)Z,\quad &&\tau>0,\eta\in(0,1),\\
            &Z(\tau,0)=\tilde{N}(\tau)+K(0)>0,\quad &&\tau>0,\\
            &Z(0,\eta)=\frac{d}{d\eta}Q_{\init}(\eta),\quad &&\eta\in[0,1].
        \end{aligned}
    \end{equation}
In terms of $Z$, we rewrite the previous first order compatibility condition in Proposition \ref{prop:tildeN-imbc} as
\begin{equation}
    Z(\tau,1)-Z(\tau,0)=K(\Phi_F)-K(0).
\end{equation} Taking derivatives w.r.t $\tau$, we obtain
\begin{equation}\label{rw-tmp-z}
    \p_{\tau}Z(\tau,1)-\p_{\tau}Z(\tau,0)=0.
\end{equation} Now we use the equation \eqref{eq-petaQ-Z-recall} to replace the temporal derivatives in \eqref{rw-tmp-z} by the spatial ones 
\begin{align}
        \p_{\eta}Z(\tau,1)-\p_{\eta}Z(\tau,0)&=K'(Q(\tau,1))Z(\tau,1)-K'(Q(\tau,0))Z(\tau,0)\\&=K'(\Phi_F)Z(\tau,1)-K'(0)Z(\tau,0).\label{dynamical-second-order}
\end{align} Set $\tau=0$ in \eqref{dynamical-second-order} and we obtain \eqref{second-compatible-recall}.

The proofs of Proposition \ref{Prop:C2} and Theorem \ref{thm:regularity-classical-relax} follow the exact blueprint of Proposition \ref{prop:regularity-classical}, so we skip them for the sake of brevity.


\section{Contraction/Expansion estimates}\label{sec:dynamics-classical}

In this section, we prove stability estimates in different norms for classical solutions of the $Q$-formulation \eqref{eq-Q-tau-classical}-\eqref{ic-q-tau-classical}. Later on, in Sections \ref{sec:global} and \ref{sec:blowup}, it will help us to distinguish different qualitative behaviors: asymptotic equilibration and blow-up in finite time. Using the $Q$-formulation, we shall derive key estimates for the dynamics of classical solutions, which are in the following form
\begin{equation}\label{estimate-formal}
    e^{k_{\min}\tau}\|Q_1(0)-Q_2(0)\|\leq \|Q_1(\tau)-Q_2(\tau)\|\leq e^{k_{\max}\tau}\|Q_1(0)-Q_2(0)\|.
\end{equation} 
Here $Q_1,Q_2$ are two classical solutions to \eqref{eq-Q-tau-classical}-\eqref{ic-q-tau-classical}, and $k_{\min},k_{\max}$ relate to $K(\phi)$ via
\begin{equation}\label{def-kmin_kmax}   
{k}_{\min}:=K'(0)+\int_0^{\Phi_F}(K''(\phi))_-d\phi,\qquad {k}_{\max}:=K'(0)+\int_0^{\Phi_F}(K''(\phi))_+d\phi.
\end{equation}
The estimates hold as long as the classical solutions exist, and $\|\cdot\|$ is a suitable norm to be specified. A particular class of phase response functions $K$ will simplify some of our main results:
\begin{assumption}\label{as:K-cc}
    $K\in C^2[0,\Phi_F]$ is either convex on $[0,\Phi_F]$ or concave on $[0,\Phi_F]$.
\end{assumption}
\noindent In other words, Assumption \ref{as:K-cc} is equivalent to $K''(\phi)$ does not change sign on $[0,\Phi_F]$.

\begin{remark}\label{rmk:Kcc}
If Assumption \ref{as:K-cc} holds, then
\begin{equation*}
k_{\min}=\min_{\phi_\in[0,\Phi_F]}K'(\phi),\qquad k_{\max}=\max_{\phi_\in[0,\Phi_F]}K'(\phi)\,.
\end{equation*} Note that in general
\begin{equation}\label{ineq-kminkmax}
    k_{\min}\leq \min_{\phi\in[0,\Phi_F]}K'(\phi)\leq \max_{\phi\in[0,\Phi_F]}K'(\phi)\leq k_{\max}.
\end{equation}
\end{remark} 

In Section \ref{subsec:BV} we shall present and explain the estimate \eqref{estimate-formal} when $\|\cdot\|$ is a BV norm. Section \ref{subsec:L2} is devoted to the case when $\|\cdot\|$ is a variant of the $L^2$ norm. These two cases can be unified into a framework in Section \ref{subsec:unifying}, utilizing the perspective that $N(\tau)$ is a ``Lagrangian multiplier'' for the constraint $Q(\tau,\eta=1)=\Phi_F$ (see Proposition \ref{prop:char-N-classical}).

\subsection{A BV estimate}\label{subsec:BV}

Let $(Q_1,N_1)$, $(Q_2,N_2)$ be two classical solutions to \eqref{eq-Q-tau-classical}-\eqref{ic-q-tau-classical} with different initial data. Then we have \eqref{eq-Q-tau-classical} holds in the existence time interval of $Q_1,Q_2$, i.e. for $i=1,2$,
\begin{equation*}
      \begin{aligned}
        &\p_{\tau}Q_i+\p_{\eta}Q_i=\frac{1}{N_i(\tau)}+{K}(Q_i).
    \end{aligned}
\end{equation*} If we directly compare $Q_1$ and $Q_2$ by subtracting their respective equations, the term $1/N_1 - 1/N_2$ arises. This term is difficult to control due to the nonlocal dependence of $1/N$ on the boundary derivative of $Q$, as in \eqref{def-N-tau-classical}, but it is also simple as it is a constant in $\eta$. Taking the derivative w.r.t $\eta$ in the equation \eqref{eq-Q-tau-classical}, then $1/N$ disappears and we obtain an equation for $\p_{\eta}Q$
\begin{equation}\label{eq-peta}
    (\p_\tau+\p_{\eta})(\p_{\eta}Q)=K'(Q)\p_{\eta}Q,\quad \tau>0,\eta\in(0,1).
\end{equation}
Note that \eqref{eq-peta}, as an equation for $\p_{\eta}Q$, is still nonlocal due to $Q=\int_0^{\eta}\p_{\eta}Qd\tilde{\eta}$ in $K'(Q)$. 

\begin{theorem}\label{thm:BV}
    Let $(Q_1,N_1)$ and $(Q_2,N_2)$ be two classical solutions to \eqref{eq-Q-tau-classical}-\eqref{ic-q-tau-classical}. Then we have
    \begin{equation}\label{estimate-BV}
            e^{k_{\min}\tau}\|\p_{\eta}Q_1(0)-\p_{\eta}Q_2(0)\|_{L^1(0,1)}\leq \|\p_{\eta}Q_1(\tau)-\p_{\eta}Q_2(\tau)\|_{L^1(0,1)}\leq e^{k_{\max}\tau}\|\p_{\eta}Q_1(0)-\p_{\eta}Q_2(0)\|_{L^1(0,1)}.
    \end{equation}
    The estimates hold whenever the two classical solutions exist up to time $\tau$. Here $k_{\min}$ and $k_{\max}$ are defined in \eqref{def-kmin_kmax}.
\end{theorem}

Before proving Theorem \ref{thm:BV}, we will discuss its consequences on the asymptotic behavior of classical solutions and implications in particular cases.

For instance, when $k_{\min}>0$, \eqref{estimate-BV} implies that $\|\p_{\eta}Q_1(\tau)-\p_{\eta}Q_2(\tau)\|_{L^1(0,1)}$ grows exponentially. However, it shall be uniformly bounded since
\begin{equation}\label{bdd-BV}
\begin{aligned}
        \|\p_{\eta}Q_1(\tau)-\p_{\eta}Q_2(\tau)\|_{L^1(0,1)}&\leq  \|\p_{\eta}Q_1(\tau)\|_{L^1(0,1)}+\|\p_{\eta}Q_2(\tau)\|_{L^1(0,1)}\\&=\int_0^{1}\p_{\eta}Q_1(\tau,\eta)d\eta+\int_0^{1}\p_{\eta}Q_2(\tau,\eta)d\eta\\&=Q_1(\tau,1)-Q_1(\tau,0)+Q_2(\tau,1)-Q_2(\tau,0)=2\Phi_F,
\end{aligned}
\end{equation} where we use $Q_i$ is increasing in space in the second line, and use boundary condition \eqref{bc-Q-tau-classical} and Proposition \ref{prop:char-N-classical} in the last line. This leads to a contradiction with the lower bound in \eqref{estimate-BV} when $k_{\min}>0$ and $Q_1(0)\neq Q_2(0)$. As a consequence, at least one of $Q_i$ blows up in finite time (for more discussion on this see later Theorem \ref{thm:blow-up-increasingK}).

On the other hand if $k_{\max}<0$, \eqref{estimate-formal} implies exponential convergence to the steady state, provided that both the global solution and the steady state exist.

In view of Remark \ref{rmk:Kcc}, $k_{\min}>0$ implies that $K'>0$ on $[0,\Phi_F]$. They are equivalent if Assumption \ref{as:K-cc} holds. The same can be said between $k_{\max}<0$ and $K'<0$ on $[0,\Phi_F]$.

\begin{remark}
    By \eqref{tmp-pqpe}, we can rewrite the distance in \eqref{estimate-BV} using the density function $\rho$ and denoting $\phi_1 = Q_1(\eta) $, $\phi_2 = Q_2(\eta)$
    \begin{align}\label{tmp-bv-rho-1}
        \|\p_{\eta}Q_1-\p_{\eta}Q_2\|_{L^1(0,1)}&=\int_0^{1}\left|\frac{1}{\rho_1(Q_1(\eta))}-\frac{1}{\rho_2(Q_2(\eta))}\right|d\eta =\int_0^{1}\left|\frac{\rho_1(\phi_1)-\rho_2(\phi_2)}{\rho_1(\phi_1)\rho_2(\phi_2)}\right|d\eta.
    \end{align}
    This can be understood as a modified weighted $L^1$ distance (between probability density $\rho$). 
\end{remark}
\begin{remark}
Thanks to the boundary condition \eqref{bc-Q-tau-classical} $Q_1(\tau,0)=Q_2(\tau,0)=0$, the norm in \eqref{estimate-BV} controls the Wasserstein distance between the measures $\rho_1$ and $\rho_2$ of any exponent. Actually, we have
\begin{equation}\label{L-infty-BV}
    |Q_1(\tau,\eta)-Q_2(\tau,\eta)|\leq \int_0^{\eta}|\p_{\eta}Q_1(\tau,\tilde{\eta})-\p_{\eta}Q_2(\tau,\tilde{\eta})|d\tilde{\eta}\leq \|\p_{\eta}Q_1(\tau)-\p_{\eta}Q_2(\tau)\|_{L^1(0,1)}.
\end{equation}
 The link to optimal transport is built on the following classical isometry between probability measures and their pseudo-inverses (see e.g. \cite[Theorem 2.18, Chapter 2.2]{MR1964483}) 
    \begin{equation}\label{isometry}
        \|Q_1-Q_2\|_{L^p(0,1)}=W_p(\mu_1,\mu_2),
    \end{equation}
where $W_p$ is the $p$-Wasserstein distance between measures, for all $1\leq p\leq \infty$, see e.g. \cite{MR1964483,MR3409718}.
Taking the supremum in $\eta$ on the left-hand side of \eqref{L-infty-BV}, we conclude that 
$$
W_1(\rho_1,\rho_2)\leq W_p(\rho_1,\rho_2)\leq W_\infty(\rho_1,\rho_2)\leq \|\p_{\eta}Q_1(\tau)-\p_{\eta}Q_2(\tau)\|_{L^1(0,1)}
$$
for all $1\leq p\leq \infty$.
\end{remark}

Let us point out that Theorem \ref{thm:BV} relies crucially on the dilated timescale $\tau$. Indeed, to obtain contraction/expansion estimates, we compare two solutions at the same $\tau$ time value. For two different firing rate $N_1$,$N_2$, by definition \eqref{def-tau-sing}, the same $\tau$ time value in general corresponds to different times in timescale $t$. More precisely, $Q_1(\tau)-Q_2(\tau)$ corresponds to  $Q_1(t_1)-Q_2(t_2)$ with
\begin{equation}
    t_1=\int_0^{\tau}\frac{1}{N_1(\tilde{\tau})}d\tilde{\tau},\qquad  t_2=\int_0^{\tau}\frac{1}{N_2(\tilde{\tau})}d\tilde{\tau},
\end{equation} which are not the same in general. Nevertheless, we can still derive a simple result in $t$ from Theorem \ref{thm:BV} when $Q_2$ is a steady state, i.e. does not depend on time.

\begin{corollary}[See also {\cite[Theorem 1]{mauroy2012global}}]\label{cor:BV}
    Suppose there exists a steady state $Q^*(\eta)$ to \eqref{eq-Q-tau-classical}-\eqref{ic-q-tau-classical}. 
    Let $(Q,N)$ be a classical solution to \eqref{eq-Q-tau-classical}-\eqref{ic-q-tau-classical}. With abuse of notation, we also use $Q(t,\eta)$ and $N(t)$ to denote their counterparts in $t$ timescale through the change of variable \eqref{def-tau-sing}, which give a solution to the system in $t$ \eqref{eq:q-t-1}-\eqref{ic-q-classical}. Then we have
    \begin{equation}\label{estimate-BV-t}
            e^{k_{\min}\tau(t)}\|\p_{\eta}Q(0)-\p_{\eta}Q^*\|_{L^1(0,1)}\leq \|\p_{\eta}Q(t)-\p_{\eta}Q^*\|_{L^1(0,1)}\leq e^{k_{\max}\tau(t)}\|\p_{\eta}Q(0)-\p_{\eta}Q^*\|_{L^1(0,1)},
    \end{equation} where
    \begin{equation}
        \tau(t)=\int_0^{t}N(s)ds,
    \end{equation} and $k_{\min}$ and $k_{\max}$ are defined in \eqref{def-kmin_kmax}.
\end{corollary}
\begin{proof}
    Take $Q_1(\tau,\eta)=Q(\tau,\eta)$ and $Q_2(\tau,\eta)=Q^*(\eta)$ in Theorem \ref{thm:BV}, and then we have
    \begin{equation}
            e^{k_{\min}\tau}\|\p_{\eta}Q(0)-\p_{\eta}Q^*\|_{L^1(0,1)}\leq \|\p_{\eta}Q(\tau)-\p_{\eta}Q^*\|_{L^1(0,1)}\leq e^{k_{\max}\tau}\|\p_{\eta}Q(0)-\p_{\eta}Q^*\|_{L^1(0,1)},
    \end{equation} 
    Changing back to timescale $t$ via \eqref{def-tau-sing} we obtain \eqref{estimate-BV-t}. Note that $d\tau=N(t)dt$ gives the formula of $\tau(t)$.
\end{proof}

Indeed, Corollary \ref{cor:BV} recovers in particular a main result obtained in the pioneering work \cite{mauroy2012global} under Assumption \ref{as:K-cc}. They directly calculated in $t$ timescale without working with an equation for $Q$ explicitly. The proof of Theorem \ref{thm:BV} below provides the general argument for estimating the distance between two general solutions by working with the equation for $Q$ in the dilated timescale $\tau$. This generic property for two solutions unveils a more generic structure of the equation with respect this distance. The original use of the BV distance $\|\p_{\eta}Q(t)-\p_{\eta}Q^*\|_{L^1(0,1)}$ between pseudo-inverses in \cite{mauroy2012global} was inspired by their earlier work on the particle system \cite{mauroy2008clustering,mauroy2009erratum}. We refer to \cite[Section IV.B]{mauroy2012global} for this discussion.\\

We now turn to the proof of Theorem \ref{thm:BV}. Its proof is divided into two steps. In Lemma \ref{lem:dIdtau} we derive an identity for the time derivative of $\|\p_{\eta}Q_1(\tau)-\p_{\eta}Q_2(\tau)\|_{L^1(0,1)}$ and in Lemma \ref{lem:integral-ineq-bv} we estimate it. Then we can conclude by the Gronwall inequality. We recall the definition of sign function as it will appear in the following lemmas.
    \begin{equation}\label{def-sign}
            \sign{x}:=\begin{cases}1,\,\,\,\,\quad x>0,\\0,\,\,\,\,\quad x=0,\\-1,\quad x<0.
    \end{cases}
    \end{equation}

\begin{lemma}\label{lem:dIdtau}
    Let $(Q_1,N_1)$ and $(Q_2,N_2)$ be two classical solutions to \eqref{eq-Q-tau-classical}-\eqref{ic-q-tau-classical}. Suppose two classical solutions exist up to time $\tau$. Given
    \begin{equation}\label{def-Itau}
    \begin{aligned}
        I(\tau)&:=\int_0^{1}\Bigl|\p_{\eta}Q_1(\tau,\eta)-\p_{\eta}Q_2(\tau,\eta)\Bigr|d\eta=\|\p_{\eta}Q_1(\tau)-\p_{\eta}Q_2(\tau)\|_{L^1(0,1)},
    \end{aligned}
    \end{equation}
we have
    \begin{equation}\label{dI-dtau}
         I(\tau)-I(0)=\int_0^{\tau}I_K(s)ds,
    \end{equation} 
where $I_K(\tau)$ is defined as
    \begin{equation}\label{def-IK}
        \begin{aligned}            I_K(\tau):=\int_{0}^1\p_{\eta}\Bigl(K(Q_1(\tau,\eta))-K(Q_2(\tau,\eta))\Bigr){\rm sign }\left({\p_\eta}Q_1(\tau,\eta)-{\p_\eta}Q_2(\tau,\eta)\right)d\eta.
        \end{aligned}
    \end{equation}
\end{lemma}

\begin{proof}
We will primarily work in the case when the initial data of $Q_1$ and $Q_2$ satisfy Assumption \ref{as:second-order}. This ensures the solution $Q_1,Q_2$ are $C^2$ by Theorem \ref{thm:regularity-classical}, which is convenient for some intermediate calculations involving second derivatives of $Q_i$. In general when the initial data may not satisfy Assumption \ref{as:second-order}, we can use a density argument. More precisely, we can approximate the initial data in the $C^1[0,1]$ norm, by a sequence of initial data that satisfy Assumption \ref{as:second-order}. With \eqref{dI-dtau} proved for those initial data, we can pass to the limit and obtain the result for the original initial data, using Theorem \ref{thm:stability-classical}.

First we justify the differentiation in time for the $L^1$ norm, treating the issue that the absolute value function $|x|$ is not differentiable in the classical sense at $x=0$.

Note that the following holds for a $C^1$ function $f(\tau)$
\begin{equation}\label{tmp-lemma1-0}
        |f(\tau)|-|f(0)|=\int_0^{\tau}f'(s)\sign{f(s)}ds,
\end{equation}
with the sign function defined in \eqref{def-sign}. This can be verified by, e.g., using a family of regularized functions to approximate the absolute value $|x|$.

We extend this argument to include spatial dependence. Consider a two-dimensional $C^1$ function $f(\tau,\eta)$. We derive
\begin{align*}
    \int_0^{1}|f(\tau,\eta)|d\eta-\int_0^{1}|f(0,\eta)|d\eta&=\int_0^{1}\left(\int_0^{\tau}\frac{\p }{\p \tau}f(s,\eta)\sign{f(t,\eta)}ds\right)d\eta\\&=\int_0^{\tau}\left(\int_0^{1}\frac{\p }{\p \tau}f(s,\eta)\sign{f(s,\eta)}d\eta\right)ds,
\end{align*} 
by Fubini's theorem. Taking $f(\tau,\eta)=\p_{\eta}Q_1(\tau,\eta)-\p_{\eta}Q_2(\tau,\eta)$, we obtain
\begin{equation}\label{tmp-lemma1-0.5}
    I(\tau)-I(0)=\int_0^{\tau}J(s)ds,
\end{equation} where $I(\tau)$ is defined in \eqref{def-Itau} and $J(\tau)$ is given by
\begin{equation}\label{def-Jtau}       J(\tau)=\int_{0}^1\p_{\tau}\Bigl(\p_{\eta}Q_1(\tau,\eta)-\p_{\eta}Q_2(\tau,\eta)\Bigr){\rm sign }\left({\p_\eta}Q_1(\tau,\eta)-{\p_\eta}Q_2(\tau,\eta)\right)d\eta.
\end{equation}
Then it remains to show that $J(\tau)$ coincides with $I_K(\tau)$ defined in \eqref{def-IK}. 
Taking the derivative w.r.t $\eta$ in the equation \eqref{eq-Q-tau-classical}, we have
$
    (\p_{\tau}+\p_{\eta})(\p_{\eta}Q_i)=\p_{\eta}(K(Q_i)),
$ for $i=1,2$ which gives
$
    \p_{\tau}(\p_{\eta}Q_i)=\p_{\eta}(K(Q_i))-\p_{\eta}(\p_{\eta}Q_i).
$ Hence, we compute
\begin{equation}\label{tmp-lemma1}
    \p_{\tau}(\p_{\eta}Q_1-\p_{\eta}Q_2)=\p_{\eta}(K(Q_1)-K(Q_2))-\p_{\eta}(\p_{\eta}Q_1-\p_{\eta}Q_2)
\end{equation}
Substitute \eqref{tmp-lemma1} into the definition of $J(\tau)$ in \eqref{def-Jtau}, and we derive
\begin{align}\notag
J(\tau)=&\,\int_{0}^1\p_{\eta}\Bigl(K(Q_1)-K(Q_2)\Bigr){\rm sign }\left({\p_\eta}Q_1(\tau,\eta)-{\p_\eta}Q_2(\tau,\eta)\right)d\eta\\ \notag&-\int_{0}^1\p_{\eta}\Bigl(\p_{\eta}Q_1-\p_{\eta}Q_2\Bigr){\rm sign }\left({\p_\eta}Q_1(\tau,\eta)-{\p_\eta}Q_2(\tau,\eta)\right)d\eta\\[5pt]= &\, I_K(\tau)- \int_{0}^1\p_{\eta}\Bigl(\p_{\eta}Q_1-\p_{\eta}Q_2\Bigr){\rm sign }\left({\p_\eta}Q_1(\tau,\eta)-{\p_\eta}Q_2(\tau,\eta)\right)d\eta,\label{tmp-lemma1-2}
\end{align} where in the last line we use the definition of $I_K$ \eqref{def-IK}. 

Our conclusion is equivalent to show that the last term in \eqref{tmp-lemma1-2} is zero. We calculate (or using \eqref{tmp-lemma1-0})
\begin{align}\notag
    \int_0^{1}\p_\eta\left(\p_{\eta}Q_1-\p_{\eta}Q_2\right)\sign{{\p_\eta}Q_1-{\p_\eta}Q_2}d\eta&=\int_0^1\p_{\eta}|\p_{\eta}Q_1-\p_{\eta}Q_2|d\eta\\&=|\p_{\eta}Q_1-\p_{\eta}Q_2|(\tau,1)-|\p_{\eta}Q_1-\p_{\eta}Q_2|(\tau,0).\label{tmp-lemma1-3}
\end{align} 
To show the right hand side of \eqref{tmp-lemma1-3} is zero, it suffices to prove (a stronger result)
\begin{equation}\label{claim-eq-bd}
    \p_{\eta}Q_1(\tau,1)-\p_{\eta}Q_2(\tau,1)=\p_{\eta}Q_1(\tau,0)-\p_{\eta}Q_2(\tau,0).
\end{equation}Actually, we have the following formulas, for $i=1,2$,
\begin{align}\label{tmp-lem1-a}
    \frac{1}{N_i}+K(0)-\p_{\eta}Q_i(\tau,0)&=0,\\\frac{1}{N_i}+K(\Phi_F)-\p_{\eta}Q_i(\tau,1)&=0.\label{tmp-lem1-b}
\end{align} 
Here \eqref{tmp-lem1-a} is just \eqref{ptau-bc-0} in Proposition \ref{prop:keep-mono} and \eqref{tmp-lem1-b} is just the definition of $N$ in \eqref{def-N-tau-classical}. 
Using \eqref{tmp-lem1-a} and \eqref{tmp-lem1-b} we derive
\begin{equation*}
    \p_{\eta}Q_1(\tau,0)-\p_{\eta}Q_2(\tau,0)=\frac{1}{N_1}-\frac{1}{N_2}=\p_{\eta}Q_1(\tau,1)-\p_{\eta}Q_2(\tau,1),
\end{equation*} which gives \eqref{claim-eq-bd}. Hence by \eqref{tmp-lemma1-3} the last term in \eqref{tmp-lemma1-2} is zero and we have $J(\tau)=I_K(\tau)$. Then by \eqref{tmp-lemma1-0.5} we conclude the proof.
\end{proof}

We have derived \eqref{dI-dtau}. The next step is to control $I_K(\tau)$ in terms of $I(\tau)$. This is trivial when $K(Q)=kQ+b$ is an affine function, in which case
\begin{align*}
    I_K(\tau)&=k\int_{0}^1\p_{\eta}\Bigl(Q_1(\tau,\eta)-Q_2(\tau,\eta)\Bigr){\rm sign }\left({\p_\eta}Q_1(\tau,\eta)-{\p_\eta}Q_2(\tau,\eta)\right)d\eta\\&=k\int_{0}^1|{\p_\eta}Q_1(\tau,\eta)-{\p_\eta}Q_2(\tau,\eta)|d\eta=kI(\tau).
\end{align*}

For general $K$, we have the following Lemma \ref{lem:integral-ineq-bv}. It does not rely on any temporal information for the solution of the PDE \eqref{eq-Q-tau-classical}. Hence we state it as an inequality for spatial functions $Q(\eta)$ such that $Q$ is increasing with $Q(0)=0$ and $Q(1)=\Phi_F$. We recall that those properties are features of a pseudo-inverse and hold for a solution to \eqref{eq-Q-tau-classical}-\eqref{ic-q-tau-classical}, as in Proposition \ref{prop:keep-mono}, the boundary condition \eqref{bc-Q-tau-classical} and Proposition \ref{prop:char-N-classical}.

\begin{lemma}
    \label{lem:integral-ineq-bv} Let $Q_1,Q_2\in C^1[0,1]$ two be non-decreasing functions with 
\begin{equation}\label{lem-IK-BD-Q}
    Q_1(0)=Q_2(0)=0,\qquad  Q_1(1)=Q_2(1)=\Phi_F.
\end{equation}
Given 
\begin{equation}
\begin{aligned}
                    I:=\int_0^{1}|{\p_\eta}Q_1(\eta)-{\p_\eta}Q_2(\eta)|d\eta=\int_0^{1}({\p_\eta}Q_1-{\p_\eta}Q_2)\sign{{\p_\eta} Q_1-{\p_\eta} Q_2}d\eta
\end{aligned}
\end{equation} 
and 
\begin{equation}\label{def-IK}
    I_K:=\int_0^{1}\p_{\eta}(K(Q_1)-K(Q_2))\sign{{\p_\eta} Q_1-{\p_\eta} Q_2}d\eta.
\end{equation} 
Then we have
\begin{equation}\label{est_Ik}
    k_{\min}I\leq I_K\leq k_{\max} I,
\end{equation} where
 \begin{align*}
{k}_{\min}:=K'(0)+\int_0^{\Phi_F}(K''(\phi))_-d\phi,\qquad {k}_{\max}:=K'(0)+\int_0^{\Phi_F}(K''(\phi))_+d\phi.
    \end{align*}
as defined in \eqref{def-kmin_kmax}
\end{lemma}

\begin{proof}
Let us use the shorthand notation $f_1(\eta):=\p_{\eta}Q_1(\eta)$, $f_2(\eta):=\p_{\eta}Q_2(\eta)$.
Then $I$ and $I_K$ can be written as
\begin{align}\label{tmp-Ik}
    I=\int_0^{1}|f_1(\eta)-f_2(\eta)|d\eta,
\qquad
\mbox{and}
\qquad
        I_K=\int_0^{1}\bigl(K'(Q_1)f_1-K'(Q_2)f_2\bigr)\sign{f_1-f_2}d\eta.
\end{align}
Now we rewrite $I_K$. Using
\begin{equation}\label{tmp-NL-K}
    K'(Q_i(\eta))=K'(0)+\int_0^{Q_i(\eta)}K''(\phi)d\phi,\quad i=1,2,
\end{equation}
we deduce from \eqref{tmp-Ik}
\begin{align}
        I_K&=\int_0^{1}K'(0)(f_1-f_2)\sign{f_1-f_2}d\eta+R=K'(0)\int_0^{1}|f_1-f_2|d\eta+R=K'(0)I+R, \label{tmp-IK-2}
\end{align}
where $R$ is given by
\begin{equation}\label{def-R}
    R:=\int_0^{1}\left[\left(\int_0^{Q_1}K''(\phi)d\phi\right)f_1-\left(\int_0^{Q_2}K''(\phi)d\phi\right)f_2\right]\sign{f_1-f_2}d\eta.
\end{equation}
We use Fubini theorem to rewrite $R$. For the first term we have
\begin{align*}
    \int_0^{1}\left(\int_0^{Q_1}K''(\phi)d\phi\right)f_1\sign{f_1-f_2}d\eta=\int_0^{\Phi_F}K''(\phi)\left(\int_{Q_1(\eta)\geq \phi}f_1(\eta)\sign{f_1-f_2}(\eta)d\eta\right)d\phi.
\end{align*}
Note that for $\phi\in(0,\Phi_F)$, the set $\{\eta:Q_1(\eta)\geq \phi\}$ is a (non-empty) interval since $Q_1$ is non-decreasing (and satisfies~\eqref{lem-IK-BD-Q}). Similarly we use the Fubini theorem to treat the second term in \eqref{def-R}, and arrive at
\begin{align}\label{tmp-R}
    R=\int_0^{\Phi_F}K''(\phi)r(\phi)d\phi,
\end{align} where
\begin{equation}\label{def-r}
    r(\phi)=\int_{Q_1(\eta)\geq \phi}f_1(\eta)\sign{f_1-f_2}(\eta)d\eta-\int_{Q_2(\eta)\geq \phi}f_2(\eta)\sign{f_1-f_2}(\eta)d\eta.
\end{equation}
Combining \eqref{tmp-IK-2} and \eqref{tmp-R}, we actually have
\begin{equation}\label{rewrite-IK}
    I_K=K'(0)I+\int_0^{\Phi_F}K''(\phi)r(\phi)d\phi.
\end{equation}
We claim that 
\begin{equation}\label{claim-rphi}
   0\leq r(\phi)\leq I,\quad \forall \phi\in[0,\Phi_F].
\end{equation} Indeed, if \eqref{claim-rphi} holds, we conclude
\begin{equation}
    k_{\min}I=\left(K'(0) +\int_0^{\Phi_F}(K''(\phi))_- d\phi\right) I  \leq  I_K\leq \left(K'(0) +\int_0^{\Phi_F}(K''(\phi))_+ d\phi \right) I= k_{\max}I
\end{equation}

Finally, we prove \eqref{claim-rphi}. For a fixed $\phi\in[0,\Phi_F]$, let $\eta_1,\eta_2\in[0,1]$ such that
\begin{equation}\label{def-eta1-eta2}
    Q_i(\eta_i)=\phi,\qquad i=1,2,
\end{equation} whose existence is ensured by \eqref{lem-IK-BD-Q} the continuity of $Q_i$ (here we do not need them to be unique). Without loss of generality, we assume $\eta_1\geq \eta_2$. Then we deduce
\begin{align}
    r(\phi)&=\int_{\eta_1}^{1}f_1(\eta)\sign{f_1-f_2}(\eta)d\eta-\int_{\eta_2}^{1}f_2(\eta)\sign{f_1-f_2}(\eta)d\eta\\&=\int_{\eta_1}^{1}|f_1-f_2|d\eta-\int_{\eta_2}^{\eta_1}f_2(\eta)\sign{f_1-f_2}(\eta)d\eta.\label{tmp-pf-42-2024}
\end{align} 
We first estimate 
\begin{align*}
    r(\phi)&\geq  \int_{\eta_1}^{1}(f_1-f_2)d\eta-\int_{\eta_2}^{\eta_1}f_2d\eta,
\end{align*} where we use $f_2\geq0$ since $Q_2$ is non-decreasing. Then we conclude
\begin{align*}
    r(\phi)&\geq \int_{\eta_1}^{1}f_1d\eta-\int_{\eta_2}^{1}f_2d\eta=Q_1(1)-Q_1(\eta_1)-(Q_2(1)-Q_2(\eta_2))=0,
\end{align*} 
where we use \eqref{lem-IK-BD-Q} and \eqref{def-eta1-eta2}.
Similarly we can prove $r(\phi)\leq I$ taking into account \eqref{tmp-pf-42-2024} and noting
\begin{align*}
    I-r(\phi)&=\int_{0}^{\eta_1}|f_1-f_2|d\eta+\int_{\eta_2}^{\eta_1}f_2(\eta)\sign{f_1-f_2}(\eta)d\eta\\& \geq \int_{0}^{\eta_1}(f_2-f_1)d\eta-\int_{\eta_2}^{\eta_1}f_2d\eta= -\int_{0}^{\eta_1}f_1d\eta+\int_{0}^{\eta_2}f_2d\eta=-(Q_1(\eta_1)-Q_1(0))+Q_2(\eta_2)-Q_2(0)=0,
    \end{align*}
where we use $f_2\geq0$, \eqref{lem-IK-BD-Q} and \eqref{def-eta1-eta2} again.
\end{proof}

\begin{remark}\label{rmk:pf-rw-r}
Using $\p_{\eta}Q_1=f_1$ we have $\mathbb{I}_{Q_1(\eta)\geq \phi} f_1 = \p_{\eta}(h_\phi(Q_1(\eta)))$,
where $h_\phi(q)=(q-\phi)_+$. This allows to rewrite $r(\phi)$ in \eqref{def-r} as
\begin{equation}
    r(\phi)=\int_0^{1}\p_{\eta}(h_\phi(Q_1)-h_\phi(Q_2))\sign{{\p_\eta} Q_1-{\p_\eta} Q_2}d\eta,
\end{equation} 
which is in the same form of $I_K$ \eqref{def-IK}, with the function $K(q)$ replaced by $h_{\phi}(q)$. 
Note that the identity
\begin{align}
    \p_\eta(h_\phi(Q_1)-h_\phi(Q_2))\sign{{\p_\eta}(Q_1-Q_2)}&=\p_\eta(h_\phi(Q_1)-h_\phi(Q_2))\sign{\p_\eta(h_\phi(Q_1)-h_\phi(Q_2))}\\&=|\p_\eta(h_\phi(Q_1)-h_\phi(Q_2))|
\end{align}
holds in two cases i) $Q_1\geq \phi$ and $Q_2\geq \phi$ or ii)   $Q_1\leq \phi$ and $Q_2\leq \phi$. 
\end{remark}


\subsection{A $L^2$ estimate}\label{subsec:L2}

In Theorem \ref{thm:BV}, the distance involved
a spatial derivative, which assumes strong regularity for solutions. Moreover, this distance may not be useful for establishing the stability of generic empirical measures with different mass configurations. Though in this work, we focus on classical solutions (which exclude empirical measures), it is still desirable to find estimates in a distance with weaker regularity assumptions.

The derivative $\p_{\eta}$ in Theorem \ref{thm:BV} was used to remove the $1/N$ term in \eqref{eq-Q-tau-classical}, which is a constant in space. Another way to circumvent the challenge due to the $1/N$ term, without resorting to derivatives, is to subtract the integral average. Let
\begin{equation}
    P_0:=\id-\int_0^1d\eta.
\end{equation}
More precisely, for an integrable function $f$ on $[0,1]$, we define another function $P_0f$ via
\begin{equation}
    (P_0f)(\eta):=f(\eta)-\int_0^{1}f(\tilde{\eta})d\tilde{\eta},\quad \eta\in[0,1].
\end{equation} 
Then clearly $P_0f=P_0(f+c)$ for any constant $c$. In this way, we introduce the following distance
\begin{equation}\label{distance-L2-variant}
    \|P_0Q_1(\tau,\cdot)-P_0Q_2(\tau,\cdot)\|_{L^2(0,1)}=\left(\int_0^1(P_0Q_1(\tau,\eta)-P_0Q_2(\tau,\eta))^2d\eta\right)^{1/2},
\end{equation}
where
\begin{equation}
    P_0Q_i(\tau,\eta)=Q_i(\tau,\eta)-\int_0^{1}Q_i(\tau,\tilde{\eta})d\tilde{\eta},\qquad i=1,2.
\end{equation}
For a pesudo-inverse $Q$, $\int_0^{1}Q(\eta)d\eta$ corresponds to the first moment of its corresponding probability measure. And $P_0Q$ is the pesudo-inverse for the translation of the probability measure such that the first moment is zero, see Figure~\ref{fig:sec42}.

\begin{figure}[htbp]
    \centering    \includegraphics[width=1.0\linewidth]{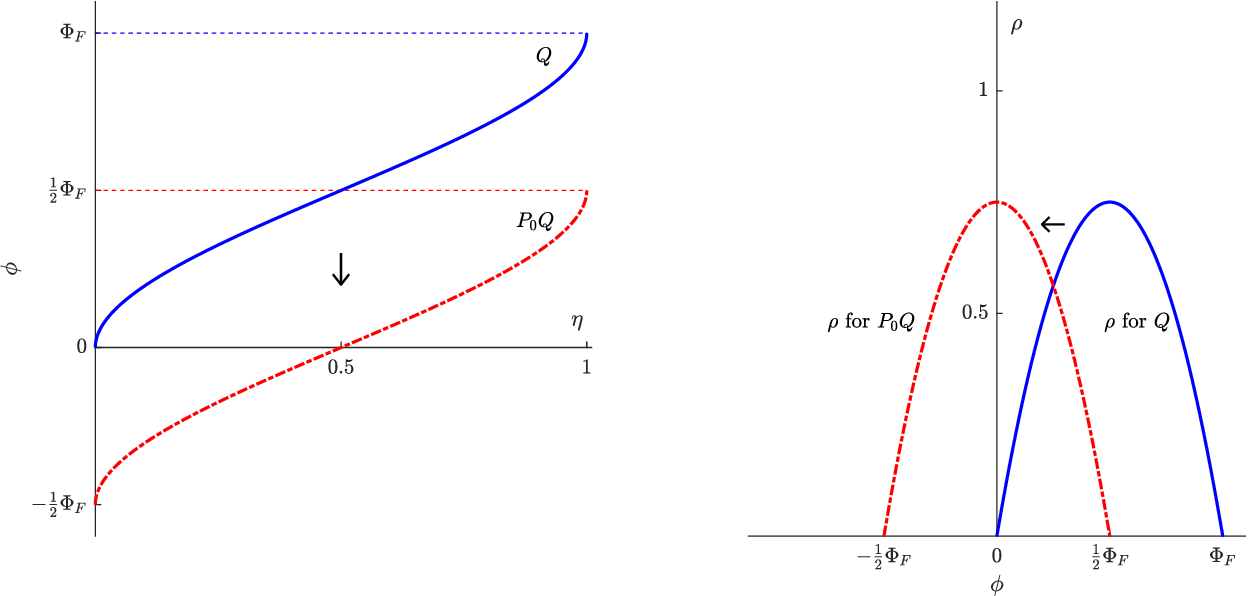}
    \caption{Illustration of $P_0Q$ for a pseudo-inverse $Q$. The red dashed-dotted line represents $P_0Q$, while the blue solid line corresponds to $Q$. Right: Plots of $Q$ and $P_0Q$. Left: Plots of the corresponding probability density $\rho$ for $Q$ and $P_0Q$. Here $\int_0^1Qd\eta=\frac{1}{2}\Phi_F$. And $P_0Q$ is the pesudo-inverse for a probability measure supported on $[-\Phi_F/2,\Phi_F/2]$ whose first moment is zero.}
    \label{fig:sec42}
\end{figure}

An interpretation of the distance \eqref{distance-L2-variant}, linking it to the Wasserstein distance in optimal transport, is given in the next result. 

\begin{proposition}\label{prop:wasserstein-interpret}
    Let $\mu_1,\mu_2$ be two probability measures on $\mathbb{R}$ (with finite second moments), and let $Q_1(\eta),Q_2(\eta)$ be their corresponding pseudo-inverse functions. Then for $P_0=\id-\int_0^1d\eta$
    \begin{align}\label{l2-inter-1}
        \|P_0Q_1-P_0Q_2\|_{L^2(0,1)}&=\min_{c\in\mathbb{R}}\|Q_1-Q_2-c\|_{L^2(0,1)}=\min_{c\in\mathbb{R}}W_2(\mu_1,T_c^{\#}\mu_2).
    \end{align} 
    Here $T_c^{\#}\mu_2$ is the push-forward of $\mu_2$ under the translation map $T_c:=\id+c$. 
\end{proposition}
\begin{proof}
    For the first equality in \eqref{l2-inter-1}, it follows from the fact that $P_0=\id-\int_0^1d\eta$ is the orthogonal projection to the orthogonal complement of constants, $\text{span}\{1\}^{\perp}$, in the Hilbert space $L^2(0,1)$ with the usual inner product. The minimizer is given by $c=\int_0^1(Q_1(\eta)-Q_2(\eta))d\eta$. Using \eqref{isometry}
 and taking into account that by definition $Q_2+c$ coincides with the pseudo-inverse function of $T_c^{\#}\mu_2$, we conclude the desired identity.
 \end{proof}

We now show the following result parallel to Theorem \ref{thm:BV}. 

\begin{theorem}\label{thm:L2}
    Let $(Q_1,N_1)$ and $(Q_2,N_2)$ be two classical solutions to \eqref{eq-Q-tau-classical}-\eqref{ic-q-tau-classical}. Then we have
\begin{equation}\label{dissaption-L2}
    \frac{d}{d\tau}\int_0^{1}(P_0(Q_1)-P_0(Q_2))^2d\eta=2\int_0^{1}P_0(Q_1-Q_2)P_0(K(Q_1)-K(Q_2))d\eta.
\end{equation}    
In particular, if $K$ is a linear function, i.e. $K(Q)=kQ+b$, we have 
\begin{equation}\label{k-affine-L2-classical}
\begin{aligned}
            \frac{d}{d\tau}\int_0^{1}(P_0(Q_1)-P_0(Q_2))^2d\eta=2k\int_0^{1}(P_0(Q_1)-P_0(Q_2))^2d\eta
\end{aligned}
\end{equation} which implies
    \begin{equation}\label{estimate-L2-linear}
         \|P_0Q_1(\tau)-P_0Q_2(\tau)\|_{L^2(0,1)}= e^{k\tau}\|P_0Q_1(0)-P_0Q_2(0)\|_{L^2(0,1)}.
    \end{equation}
    This identity holds whenever the two classical solutions exist up to time $\tau$. 
\end{theorem}

\begin{proof}
    First, we show \eqref{dissaption-L2} implies \eqref{k-affine-L2-classical} when $K(Q)=kQ+b$, in which case we have
    \begin{equation}
        K(Q_1)-K(Q_2)=kQ_1+b-(kQ_2+b)=k(Q_1-Q_2).
    \end{equation} Then as $P_0=\id-\int_0^1d\eta$ is a linear operator, it commutes with the multiplication by a scalar $k\in\mathbb{R}$
    \begin{equation}
        P_0(K(Q_1)-K(Q_2))=P_0(k(Q_1-Q_2))=kP_0(Q_1-Q_2).
    \end{equation} Substitute this into \eqref{dissaption-L2} and we obtain \eqref{k-affine-L2-classical}, which can be rewritten as an ODE
    \begin{equation}
        \frac{d}{d\tau}S(\tau)=2kS(\tau),
    \end{equation} with $S(\tau)=\|P_0Q_1(\tau)-P_0Q_2(\tau)\|_{L^2(0,1)}^2$. Solving this ODE we obtain \eqref{estimate-L2-linear}.
    
    It remains to prove \eqref{dissaption-L2}. As preparations, we study the conmutation properties of the linear operator $P_0=\id-\int_0^1d\eta$ with respect to $\p_{\tau}$ and $\p_{\eta}$. Since $P_0$ only acts on $\eta$ direction, it commutes with $\p_{\tau}$ as follows
    \begin{equation}\label{P-commute-tau}
    \begin{aligned}
        P_0(\p_{\tau}Q)&=\p_{\tau}Q(\tau,\eta)-\int_0^1\p_{\tau}Q(\tau,\tilde{\eta})d\eta=\p_{\tau}Q(\tau,\eta)-\p_{\tau}\int_0^1Q(\tau,\tilde{\eta})d\eta\\&=\p_{\tau}\left(Q(\tau,\eta)-\int_0^1Q(\tau,\tilde{\eta})d\eta\right)=\p_{\tau}(P_0Q).
    \end{aligned}
    \end{equation} 
For $\p_{\eta}$ we can compute the commutator, we first realize that
    \begin{equation}\label{tmp-3.2-1}
            \begin{aligned}
        P_0(\p_{\eta}Q)&=\p_{\eta}Q(\tau,\eta)-\int_0^1\p_{\eta}Q(\tau,\tilde{\eta})d\eta\\&=\p_{\eta}Q(\tau,\eta)-(Q(\tau,1)-Q(\tau,0))=\p_{\eta}Q(\tau,\eta)-\Phi_F,
    \end{aligned}
    \end{equation}
where we use the boundary condition \eqref{bc-Q-tau-classical} and \eqref{eq:q1-phiF} in Proposition \ref{prop:char-N-classical}. And on the other hand, we calculate
    \begin{equation}\label{tmp-3.2-2}
        \begin{aligned}
        \p_{\eta}(P_0Q)&=\p_{\eta}\left(Q(\tau,\eta)-\int_0^1Q(\tau,\tilde{\eta})d\eta\right)=\p_{\eta}Q(\tau,\eta).
    \end{aligned}
    \end{equation} 
Combining \eqref{tmp-3.2-1} and \eqref{tmp-3.2-2} we obtain the commutator between $P_0$ and $\p_{\eta}$ as
    \begin{equation}\label{P-commutator-eta}
        P_0(\p_{\eta}Q)-\p_{\eta}(P_0Q)=-\Phi_F.
    \end{equation}
Next, we apply the linear operator $P_0$ to both sides in equation \eqref{eq-Q-tau-classical} to deduce
    \begin{equation}\label{tmp-3.2-3}
        P_0(\p_{\tau}Q_i)+P_0(\p_{\eta}Q_i)=P_0(K(Q_i)),
    \end{equation} 
for $i=1,2$. Using \eqref{P-commute-tau} and \eqref{P-commutator-eta} in \eqref{tmp-3.2-3}, we get
    \begin{equation}\label{tmp-3.2-4}
        \p_{\tau}(P_0Q_i)+\p_{\eta}(P_0Q_i)+\Phi_F=P_0(K(Q_i)).
    \end{equation}
Hence, subtracting \eqref{tmp-3.2-4} for $i=2$ from that for $i=1$, we derive
    \begin{equation}\label{tmp-3.2-5}
            \p_{\tau}(P_0(Q_1-Q_2))+\p_{\eta}(P_0(Q_1-Q_2))=P_0(K(Q_1)-K(Q_2)).
    \end{equation}
Now, we multiply both side of \eqref{tmp-3.2-5} by ${2}P_0(Q_1-Q_2)$ to obtain
    \begin{equation}
        \p_{\tau}((P_0(Q_1-Q_2))^2)+\p_{\eta}((P_0(Q_1-Q_2))^2)=2P_0(K(Q_1)-K(Q_2))P_0(Q_1-Q_2),
    \end{equation}
and integrating in $\eta$ yields
    \begin{equation}\label{tmp-3.2-6}
    \begin{aligned}
        \frac{d}{d\tau}\int_0^{1}(P_0(Q_1-Q_2))^2d\eta+(P_0(Q_1-Q_2))^2|_{\eta=0}^{\eta=1}=2\int_0^{1}P_0(Q_1-Q_2)P_0(K(Q_1)-K(Q_2))d\eta.
    \end{aligned}
    \end{equation} 
In view of \eqref{tmp-3.2-6}, to prove \eqref{distance-L2-variant}, it all reduces to show
    \begin{equation}
        (P_0(Q_1-Q_2))^2|_{\eta=0}^{\eta=1}=0.
    \end{equation}
Indeed we shall show a slightly stronger result
    \begin{equation}\label{tmp-3.2-7}
                (P_0(Q_1-Q_2))|_{\eta=0}^{\eta=1}=0.
    \end{equation}
Thanks to \eqref{bc-Q-tau-classical} at $\eta=0$ and \eqref{eq:q1-phiF} at $\eta=1$, we have
    \begin{align*}
        P_0Q_i(\tau,\eta=0)&=Q_i(\tau,\eta=0)-\int_0^{1}Q_i(\tau,\tilde{\eta})d\tilde{\eta}=-\int_0^{1}Q_i(\tau,\tilde{\eta})d\tilde{\eta},\intertext{and}
        P_0Q_i(\tau,\eta=1)&=Q_i(\tau,\eta=1)-\int_0^{1}Q_i(\tau,\tilde{\eta})d\tilde{\eta}=\Phi_F-\int_0^{1}Q_i(\tau,\tilde{\eta})d\tilde{\eta},
    \end{align*}
for $i=1,2$. Hence
    \begin{align*}
        P_0(Q_1-Q_2)(\tau,\eta=0)&=-\int_0^{1}(Q_1(\tau,\tilde{\eta})-Q_2(\tau,\tilde{\eta}))d\tilde{\eta}=P_0(Q_1-Q_2)(\tau,\eta=1).
    \end{align*}
Therefore \eqref{tmp-3.2-7} is proved. Combining \eqref{tmp-3.2-7} and \eqref{tmp-3.2-6} we obtain \eqref{dissaption-L2}.
\end{proof}

\begin{remark}
    Unlike \eqref{estimate-BV} in Theorem \ref{thm:BV}, which works for a general class of $K$, \eqref{estimate-L2-linear} only holds for linear $K(Q)=kQ+b$. The gap is that we do not find an analogy of Lemma \ref{lem:integral-ineq-bv} to control the right hand side in \eqref{dissaption-L2}. We will discuss more on this in a general framework in Section \ref{subsec:unifying}.
\end{remark}

The next results justifies that \eqref{distance-L2-variant} can be viewed as a distance.

\begin{corollary}\label{cor:positive-L2}
    Let two $Q_1,Q_2\in C[0,1]$  with  
\begin{equation}\label{tmp-cor-L2-cond}
    Q_1(0)=Q_2(0)=0,\qquad  Q_1(1)=Q_2(1)=\Phi_F.
\end{equation} If
\begin{equation}\label{tmp-cor-L2-cond2}
    \|P_0Q_1-P_0Q_2\|_{L^2(0,1)}=0,
\end{equation} then
\begin{equation}\label{tmp-cor-L2-rst}
    Q_1(\eta)\equiv Q_2(\eta),\qquad \forall \eta\in[0,1].
\end{equation}
\end{corollary}

\begin{proof}
 By \eqref{tmp-cor-L2-cond2} we know $Q_1=Q_2+c$ almost everywhere in $[0,1]$ with the constant $c=\int_0^1(Q_1(\eta)-Q_2(\eta))d\eta$. Then it follows from the continuity of $Q_1$ and $Q_2$ that
    \begin{equation*}
        Q_1(\eta)\equiv Q_2(\eta)+c,\qquad \forall \eta\in[0,1].
    \end{equation*} In particular it holds at $\eta=1$. Then by \eqref{tmp-cor-L2-cond} we know $c=0$.
\end{proof}


\subsection{A unified framework}\label{subsec:unifying}

Now we unify Theorem \ref{thm:BV} and \ref{thm:L2} into a general framework, which may better illustrate how we use the equation structure of \eqref{eq-Q-tau-classical}, which we recall here for convenience,
\begin{equation}
    \p_{\tau}Q+\p_{\eta}Q=\frac{1}{N(\tau)}+{K}(Q).
\end{equation} A first crucial step in both Theorem \ref{thm:BV} and \ref{thm:L2} is to derive an modulated equation, such that the challenging term $1/N(\tau)$ does not show up explicitly. The idea is to utilize the fact that in terms of spatial dependence, $1/N(\tau)$ is just a constant. To remove it, we can apply an operator $P$, whose kernel contains constants, to equation \eqref{eq-Q-tau-classical}. For Theorem \ref{thm:BV} the choice $P=\p_{\eta}$. For Theorem \ref{thm:L2}, $P=\id-\int_0^1d\eta$. 

Here we try to discuss a general linear operator $P$ acting on the spatial variable $\eta$. It maps a function $f$ on $[0,1]$ to another function on $[0,1]$, denoted as $Pf$. For a space-time function $f(\tau,\eta)$, $Pf$ naturally extends to mean the spatial operation being applied to each time slide $\tau$, or more formally
\begin{equation}
    (Pf)(\tau,\eta):=(Pf(\tau,\cdot))(\eta).
\end{equation}
We do not specify the precise space or regularity for $P$ here to focus on more algebraic structures. To get rid of $1/N(\tau)$, we require that its kernel $\ker P$ contains constants. Then we can apply $P$ to equation \eqref{eq-Q-tau-classical} to derive

\begin{equation}\label{first-step-general}
    P(\p_{\tau}Q)+P(\p_{\eta}Q)=P(K(Q)).
\end{equation}

To proceed, we need more assumptions on $P$, roughly meaning it goes well with a solution $Q$. We summarize all assumptions as follows.

\begin{assumption}\label{as:P}
    For the linear spatial operator $P$, we assume
    
    (i) Its kernel contains all constants, i.e. $\text{span}\{1\}\subseteq\ker P$.
    
    (ii) For any classical solution $Q$, the commutator between $P$ and $\p_{\eta}$
    \begin{equation}
        [P,\p_{\eta}]Q:=P(\p_{\eta}Q)-\p_{\eta}(PQ)\equiv C(\eta),
    \end{equation} gives a fixed spatial function.

    (iii) For any classical solution $Q$, the following
    \begin{equation}
        (PQ)(\eta=1)-(PQ)(\eta=0)\equiv C,
    \end{equation} is a fixed constant.
\end{assumption}

It is not difficult to check that both $P=\p_{\eta}$ in Theorem \ref{thm:BV} and $P=\id-\int_0^1d\eta$ in Theorem \ref{thm:L2} satisfy Assumption \ref{as:P}. An analogous computation as in Theorem \ref{thm:L2} implies the following general result.

\begin{theorem}\label{thm:P_first}
    Suppose $P$ is a linear operator satisfying Assumption \ref{as:P} and $f$ is a $C^1$ function.  Let $Q_1,Q_2$ be two classical solutions to \eqref{eq-Q-tau-classical}. Then we have
\begin{equation}\label{dissaption-general}
    \frac{d}{d\tau}\int_0^{1}f(P(Q_1-Q_2))d\eta=\int_0^{1}f'(P(Q_1-Q_2))P(K(Q_1)-K(Q_2))d\eta.
\end{equation}    
\end{theorem}

By choosing $f$ such that $f(0)=0$ and $f(x)>0$ for $x\neq0$, we point out that the quantity 
$$
\int_0^{1}f(P(Q_1-Q_2))d\eta
$$
can be understood as a ``distance'' between $Q_1$ and $Q_2$, see Corollary \ref{cor:positive-L2}. Controlling the right-hand-side of \eqref{dissaption-general} is difficult for general $K$, see Theorem \ref{thm:BV} and Lemma \ref{lem:integral-ineq-bv}. However, in the simple linear $K(Q)=kQ+b$ and $f(x)=|x|^p$ ($p>1$), we can obtain similarly to Theorem \ref{thm:L2} the following:

\begin{corollary}\label{cor:linear-general}
        Suppose $K(Q)=kQ+b$ is a linear function. Let $P$ be an operator satisfying Assumption \ref{as:P}. Let $Q_1,Q_2$ be two classical solutions to \eqref{eq-Q-tau-classical}. Then we have
    \begin{equation}\label{estimate-Lp-linear}
         \|PQ_1(\tau)-PQ_2(\tau)\|_{L^p(0,1)}= e^{k\tau}\|PQ_1(0)-PQ_2(0)\|_{L^p(0,1)},
    \end{equation} for every $p>1$. The identity holds whenever the two classical solutions exist towards time $\tau$. 
\end{corollary}

In view of this framework, previous Theorem \ref{thm:BV} corresponds to the case $P=\p_{\eta}$ and $f=|x|$, and previous Theorem \ref{thm:L2} corresponds to the case $P=\id-\int_0^1d\eta$ and $f=x^2$. 


\section{Global existence and convergence to the steady state: mild interaction,  $K'<0$}\label{sec:global}

In this section, we prove the global existence and long term convergence to the steady state for a solution to \eqref{eq-Q-tau-classical}-\eqref{ic-q-tau-classical}. The result is in the regime when $K'<0$ and the size of $K$ is not very large. We have already seen in Section \ref{sec:dynamics-classical} that the sign of $K'$ can make a crucial difference to the dynamics -- in particular the distance between two solutions shrinks as soon as $k_{\max}<0$. For instance, if $K'<0$ and Assumption \ref{as:K-cc} holds. Here the size of $K$ also comes into play for the existence of a global solution and the steady state. The complementary regime when $K'>0$ or $K$ is large will be studied in the next section, where finite-time blow-ups can be shown.

 We first state and prove a global existence result in Section \ref{subsec:GE-result}. As a key tool for the proof, we derive an integral equation for $1/N$. In Section \ref{subsec:steady-state}, we first prove a dichotomy on the existence of the steady state. Then combining these results with Theorem \ref{thm:BV}, we show the long time convergence to the steady state.
 
\begin{remark}\label{rmk:51}
The main results in this section agree with \cite[Proposition 1-3]{mauroy2012global} which are presented in different forms. Their statements are for the $\rho$-formulation \eqref{eq:rho-t-1}-\eqref{ic-rho-t} in $t$ timescale. Here we develop proofs for $Q$-formulation \eqref{eq-Q-tau-classical}-\eqref{ic-q-tau-classical} in the dilated timescale $\tau$, which reveal several hidden structures of pulse-coupled oscillators. For example, during the global existence proof we work with an auxiliary variable $H$ (Proposition \ref{prop:eq-H}) and develop an integral equation for $1/N$ (Proposition \ref{prop:eq-tildeN}). These relations are not obvious in the previous formulation for $\rho$ in $t$ timescale. Besides, in \cite{mauroy2012global} Assumption \ref{as:K-cc} is needed for the convergence result, which is relaxed here to $k_{\max}<0$.
\end{remark}

\subsection{Global existence}\label{subsec:GE-result}

We define the following auxiliary variable, which is useful for the statement of global existence results,
\begin{equation}\label{def-Hinit}
    H_{\init}(\eta):=\p_{\eta}Q_{\init}(\eta)-K(Q_{\init}(\eta)),\quad \eta\in[0,1].
\end{equation} 
In particular we impose the following assumption on the initial data.
\begin{assumption}\label{as:H} With $H_{\init}$ defined in \eqref{def-Hinit}, we assume 
\begin{equation}
    H_{\init}(\eta)>0,\qquad \forall\eta\in[0,1].
\end{equation}
\end{assumption}
\begin{remark}[Connection to $\rho$ formulation.]\label{rmk:5-2}
    By definition of $H_{\init}$, Assumption \ref{as:H} is equivalent to 
    \begin{equation}\label{no-blow-up-cond-Q}
        \p_{\eta}Q_{\init}(\eta)>K(Q_{\init}(\eta)),\qquad \forall\eta\in[0,1].
    \end{equation}
    In the $\rho$ formulation, \eqref{no-blow-up-cond-Q} becomes
    \begin{equation}\label{no-blow-up-condition-rho}
        \rho_{\init}(\phi)<\frac{1}{K(\phi)},\qquad \forall\phi\in[0,\Phi_F],
    \end{equation} since $\p_{\eta}Q$ corresponds to $1/\rho$ as in \eqref{tmp-pqpe}.
\end{remark}

Now we state the theorem on global existence. In addition to the size assumption on $K$ as in Remark \ref{rmk:5-2}, we also need $K'\leq 0$ on $[0,\Phi_F]$.

\begin{theorem}[See also {\cite[Proposition 2-3]{mauroy2012global}}]\label{thm:global-solu}
Suppose $K$ satisfies that $K'\leq 0$  on $[0,\Phi_F]$ and the initial data $Q_{\init}$  satisfies Assumption \ref{as:wps}  and Assumption \ref{as:H}. 
Then there exists a unique global solution to \eqref{eq-Q-tau-classical}-\eqref{ic-q-tau-classical}. Moreover, we have the following global-in-time bounds on $N(\tau)$
    \begin{equation}\label{gloabl-bd-N}
        0<\frac{1}{\max_{\eta\in[0,1]}H_{\init}(\eta)}\leq N(\tau)\leq\frac{1}{\min_{\eta\in[0,1]}H_{\init}(\eta)}<+\infty,\qquad \tau\geq0,
    \end{equation} where $H_{\init}$ is defined as in \eqref{def-Hinit}.
\end{theorem}

\begin{remark}Assumption \ref{as:H} is necessary, although it may not be sharp, as we may not expect global existence result for general initial data, since there might not be guarantee for the constraint \eqref{constaint-classical-Q-tau} to hold for all times for arbitrary initial data. The finite time blow-up of classical solutions will be anyhow clarified in the next section. 
\end{remark}

In the following we will first work with the relaxed problem \eqref{tilde-eq-classical}-\eqref{tilde-ic-classical} for convenience, as it always has a global classical solution (Theorem \ref{thm:wps-classical-relax}). Its connection to the original problem \eqref{eq-Q-tau-classical}-\eqref{ic-q-tau-classical} has been demonstrated in Section \ref{subsec:relax}, in particular Lemma \ref{lem:relation-relaxed-original}.  With results on the relaxed problems, we will return to the original problem to prove Theorem \ref{thm:global-solu}. Recall in the relaxed problem, $\tilde{N}\in\mathbb{R}$ is used in place of $1/N>0$.
Assumption \ref{as:H} motivates us to look into the quantity
\begin{equation}\label{def-H}
    H(\tau,\eta):=\p_{\eta}Q(\tau,\eta)-K(Q(\tau,\eta)),
\end{equation}
for which we can indeed derive an equation.

\begin{proposition}\label{prop:eq-H} Let $(Q,\tilde{N})$ be a classical solution, $Q\in C^2$, to the relaxed problem  \eqref{tilde-eq-classical}-\eqref{tilde-ic-classical}.
Then $H(\tau,\eta):=\p_{\eta}Q(\tau,\eta)-K(Q(\tau,\eta))$ is $C^1$ and satisfies the following system in the classical sense 
\begin{align}
         &\p_{\tau}H+\p_{\eta}H=K'(Q)(H(\tau,\eta)-H(\tau,1)),\qquad\tau>0,\eta\in(0,1),\label{eq-H}\\\label{bc-H}
     &H(\tau,0)=H(\tau,1),\qquad\tau>0,\\
     &H(\tau=0,\eta)=H_{\init}(\eta),\qquad \eta\in[0,1],\label{ic-H}
\end{align} where $H_{\init}$ is defined as in \eqref{def-Hinit}.
Moreover we have
\begin{equation}\label{tildeN-H}
    \tilde{N}(\tau)=H(\tau,1)=H(\tau,0),\quad \tau\geq0.
\end{equation}
\end{proposition}

\begin{proof}
    Taking the derivative w.r.t. $\tau$ in \eqref{tilde-eq-classical}, we obtain the following equation for $Z:=\p_{\eta}Q$
    \begin{equation}\label{eq-Z-tmp}
        (\p_{\tau}+\p_{\eta})Z=K'(Q)Z,\quad \tau>0,\eta\in(0,1),
    \end{equation}
as we have already seen in \eqref{eq-petaQ-Z} and \eqref{eq-petaQ-Z-recall}.  Multiplying \eqref{tilde-eq-classical} with $K'(Q)$, we obtain an equation for $K(Q)$
    \begin{equation}\label{eq-K(Q)-tmp}
        (\p_{\tau}+\p_{\eta})(K(Q))=K'(Q)K(Q)+K'(Q)\tilde{N},\quad \tau>0,\,\eta\in(0,1).
    \end{equation}
    Subtracting \eqref{eq-K(Q)-tmp} from \eqref{eq-Z-tmp}, we derive an equation for $H(\tau,\eta):=\p_{\eta}Q(\tau,\eta)-K(Q(\tau,\eta))$
    \begin{equation}\label{tmp-eqH}
        (\p_{\tau}+\p_{\eta})H=K'(Q)H-K'(Q)\tilde{N}(\tau),\quad \tau>0,\,\eta\in(0,1).
    \end{equation}
    Finally, we recall the following relation from \eqref{def-tildeN-tau-classical} and Proposition \ref{prop:tildeN-imbc}
    \begin{equation}\label{tmp-tildeN-H}
        \tilde{N}(\tau)=\p_{\tau}Q(\tau,1)-K(Q(\tau,1))=\p_{\tau}Q(\tau,0)-K(Q(\tau,0)),
    \end{equation} which implies the boundary condition \eqref{bc-H}, and the relation \eqref{tildeN-H}. Substitute \eqref{tildeN-H} into \eqref{tmp-eqH}, we obtain \eqref{eq-H}. And the initial condition \eqref{ic-H} holds by definition of $H$ and $H_{\init}$ in \eqref{def-Hinit}.
    
\end{proof}

The regularity assumption, $Q\in C^2$, in the previous theorem is not essential, only for convenience to justify the calculations in the classical sense.
\begin{remark}
We here explore the meaning of $H$ \eqref{def-H} in the $\rho$ formulation. We define the flux for the continuity equation \eqref{eq:rho-t-1} as
\begin{equation}\label{def-flux-J}
    [1+K(\phi)N(t)]\rho(t,\phi)=:J(t,\phi).
\end{equation} Then \eqref{def-N-t} and \eqref{bc-rho-t} {give} $N(t)=J(t,0)=J(t,\Phi_F)$. We can rewrite \eqref{def-flux-J} as
\begin{equation}\label{H-rho}
    \frac{1}{\rho(t,\phi)}-K(\phi)=\frac{1}{J(t,\phi)}{-}K(\phi)\frac{J(t,\phi)-N(t)}{J(t,\phi)},
\end{equation} whose left hand side corresponds to $H$, as $1/\rho$ corresponds to $\p_{\eta}Q$ \eqref{tmp-pqpe}. The second term in the right hand side of \eqref{H-rho} vanishes at $\phi=0$ or $\Phi_F$, as $N(t)=J(t,0)=J(t,\Phi_F)$. Therefore we might interpret the right hand side as an approximation to the inverse of the flux, which becomes exactly the inverse of the flux when $\phi=0$ or $\Phi_F$.
\end{remark}

\begin{remark}
Note that when $K'$ is a constant (or equivalently when $K(Q)=kQ+b$ is linear),  \eqref{eq-H}-\eqref{ic-H} is a self-contained equation which is also linear in $H$. However, in general we cannot express $Q$ in terms of $H$ from \eqref{eq-H}-\eqref{ic-H}. Nevertheless, we can obtain good estimates on $\tilde{N}$ especially when $K'$ has a fixed sign since $H$ is transported with constant speed in $\eta$. 
\end{remark}

Solving $H$ along the characteristics for $\p_{\tau}+\p_{\eta}$, we can derive an integral equation for $\tilde{N}(\tau)$.

\begin{proposition}\label{prop:eq-tildeN}
    Let $(Q,\tilde{N})$ be a classical solution to the relaxed problem  \eqref{tilde-eq-classical}-\eqref{tilde-ic-classical}. Then we can derive the following integral equations for $\tilde{N}$: 
    
    For $0\leq\tau\leq1$, we have
        \begin{equation}\label{eq-tildeN-0<tau<=1}
        \tilde{N}(\tau)=\left(1-\int_{0}^{\tau}c(\tau,s)ds\right)H_{\init}(1-\tau)+\int_{0}^{\tau}c(\tau,s)\tilde{N}(s)ds,
    \end{equation} where the coefficient $c(\tau,s)$ is defined by 
    \begin{equation}\label{def-kernal-c}
        c(\tau,s)=-K'(Q(s,1-(\tau-s)))\exp\left(\int_{s}^{\tau}K'(Q(r,1-(\tau-r)))dr\right),\qquad  \text{for $s\in[\max(0,\tau-1),\tau],$}
    \end{equation} with
     \begin{equation}\label{integral-c-formula-init}
        1-\int_{0}^{\tau}c(\tau,s)ds=\exp\left(\int_{0}^{\tau}K'(Q(s,1-(\tau-s)))ds\right)>0.
    \end{equation}
    And for $\tau>1$, we have
    \begin{equation}\label{eq-tildeN-tau>1}
        \tilde{N}(\tau)=\left(1-\int_{\tau-1}^{\tau}c(\tau,s)ds\right)\tilde{N}(\tau-1)+\int_{\tau-1}^{\tau}c(\tau,s)\tilde{N}(s)ds,
    \end{equation} where the coefficient $c(\tau,s)$ is also given by \eqref{def-kernal-c} with
    \begin{equation}\label{integral-c-formula}
        1-\int_{\tau-1}^{\tau}c(\tau,s)ds=\exp\left(\int_{\tau-1}^{\tau}K'(Q(s,1-(\tau-s)))ds\right)>0.
    \end{equation}
\end{proposition}
\begin{proof}
    We shall use the equation for $H$ \eqref{eq-H} derived in Proposition \ref{prop:eq-H}. While Proposition \ref{prop:eq-H} requires $Q\in C^2$, here we can pass to general $Q$ by a density argument, as in the proof of Lemma \ref{lem:dIdtau}. %
    
    Using the equation \eqref{eq-H}, the formulas in Proposition \ref{prop:eq-tildeN} are just obtained by solving $H$ through the characteristics.
    More precisely, for fixed $\tau>1$, set $h(s):=H(s,1-(\tau-s))$ for $s\in[\tau-1,\tau]$. Then by \eqref{bc-H} and \eqref{tildeN-H} we have
    \begin{equation}\label{tmp-h}
        h(\tau-1)=H(\tau-1,0)=\tilde{N}(\tau-1),\quad h(\tau)=H(\tau,1)=\tilde{N}(\tau).
    \end{equation} 
    We calculate using \eqref{eq-H} and \eqref{tildeN-H}
    \begin{align}
        \frac{d}{ds}h(s)=(\p_{\tau}+\p_{\eta})H(s,1-(\tau-s)) =K'(Q(s,1-(\tau-s)))(h(s)-\tilde{N}(s)),\quad s\in[\tau-1,\tau],\label{tmp-linear-ODE}
    \end{align} 
which is a linear ODE given $Q$ and $\tilde{N}$. Solving the ODE \eqref{tmp-linear-ODE}, we can represent $h(\tau)$ via $h(\tau-1)$, which by \eqref{tmp-h} gives a formula for $\tilde{N}(\tau)$
    \begin{align}
        h(\tau)=\exp\left(\int_{0}^{\tau}K'(Q(s,1-(\tau-s)))ds\right)h(\tau-1)-\int_{\tau-1}^{\tau}\!\!\! K'(Q(s,1-(\tau-s)))\exp\left(\int_{s}^{\tau}K'(Q(r,1-(\tau-r))\right)\tilde{N}(s)ds,
    \end{align} 
equivalently written as \eqref{eq-tildeN-tau>1} in view of the definition of $c(\tau,s)$ in \eqref{def-kernal-c} and \eqref{tmp-h}. Indeed, it suffices to check \eqref{integral-c-formula}, which follows from
\begin{align}
            \int_{\tau-1}^{\tau}c(\tau,s)ds&=\int_{\tau-1}^{\tau}-K'(Q(s,1-(\tau-s)))\exp\left(\int_{s}^{\tau}K'(Q(r,1-(\tau-r)))dr\right)ds\\&=\int_{\tau-1}^{\tau}\frac{d}{ds}\left(\exp\left(\int_{s}^{\tau}K'(Q(r,1-(\tau-r)))dr\right)\right)ds\\&=1-\exp\left(\int_{\tau-1}^{\tau}K'(Q(s,1-(\tau-s)))ds\right).
\end{align}
For the initial case $0\leq\tau\leq1$, the derivation of \eqref{eq-tildeN-0<tau<=1} and \eqref{integral-c-formula-init} is similar.
\end{proof}

    Proposition \ref{prop:eq-tildeN} says that at each time $\tilde{N}(\tau)$ is a \textit{linear} combination of previous $\tilde{N}$ (or initial value $H_{\init}$), treating $c(\tau,s)$ as a given function. And the sum of the linear combination coefficients is $1$. We also note that the coefficient for the first term is always positive, as in \eqref{integral-c-formula-init} for the first term in \eqref{eq-tildeN-0<tau<=1}  and in \eqref{integral-c-formula} for the first term in \eqref{eq-tildeN-tau>1}. 

    When $K'<0$, we also have $c\geq0$ in \eqref{def-kernal-c}, therefore all coefficients of the linear combination are non-negative. As a consequence, the linear combination becomes  a \textit{convex} combination, which allows us to obtain the following bounds.  

\begin{proposition}\label{prop:bd-relaxed-tildeN}
    Let $(Q,\tilde{N})$ be a classical solution to the relaxed problem  \eqref{tilde-eq-classical}-\eqref{tilde-ic-classical}. Suppose additionally $K'\leq0$ on $\mathbb{R}$. Then we have the following global bounds on $\tilde{N}$
    \begin{equation}\label{bd-tildeN-H}
        \min_{\eta\in[0,1]}H_{\init}(\eta)\leq \tilde{N}(\tau) \leq \max_{\eta\in[0,1]}H_{\init}(\eta),\qquad \tau\geq0.
    \end{equation}
\end{proposition}
The proof of Proposition \ref{prop:bd-relaxed-tildeN} relies on the following two lemmas. The first one says that the conclusion of Proposition \ref{prop:bd-relaxed-tildeN} holds for $\tau\in[0,1]$.
\begin{lemma}\label{lem:init-tildeN}
 Let $(Q,\tilde{N})$ be a classical solution to the relaxed problem  \eqref{tilde-eq-classical}-\eqref{tilde-ic-classical}. Suppose additionally $K'\leq 0$ on $\mathbb{R}$. Then for $\tau\in[0,1]$ we have 
 \begin{equation}\label{eq-lem51}
       \min_{\eta\in[0,1]}H_{\init}(\eta)\leq \tilde{N}(\tau) \leq \max_{\eta\in[0,1]}H_{\init}(\eta).
 \end{equation}
 \end{lemma}
 \begin{proof}
     Suppose $\tau_{\max}\in[0,1]$ is a maximum point of $\tilde{N}$ on $[0,1]$, i.e. $\tilde{N}(\tau_{\max})=\max_{\tau\in[0,1]}\tilde{N}(\tau)$, whose existence is guaranteed by the continuity of $\tilde{N}$. Then by \eqref{eq-tildeN-0<tau<=1} in Proposition \ref{prop:eq-tildeN}, we derive
    \begin{align}
        \tilde{N}(\tau_{\max})&=\left(1-\int_{0}^{\tau_{\max}}c(\tau_{\max},s)ds\right)H_{\init}(1-\tau_{\max})+\int_{0}^{\tau_{\max}}c(\tau_{\max},s)\tilde{N}(s)ds\\&\leq\left(1-\int_{0}^{\tau_{\max}}c(\tau_{\max},s)ds\right)H_{\init}(1-\tau_{\max})+ \tilde{N}(\tau_{\max}) \int_{0}^{\tau_{\max}}c(\tau_{\max},s)ds,\label{proof-lemma51-tmp}
    \end{align} where we have used that $c(\tau,s)\geq0$ when $K'\leq 0$ by \eqref{def-kernal-c}, and that $\tau_{\max}$ is the maximum point on $[0,1]$. Then \eqref{proof-lemma51-tmp} simplifies to 
    \begin{equation}
        \left(1-\int_{0}^{\tau_{\max}}c(\tau_{\max},s)ds\right)\tilde{N}(\tau_{\max})\leq  \left(1-\int_{0}^{\tau_{\max}}c(\tau_{\max},s)ds\right) H_{\init}(1-\tau_{\max}),
    \end{equation} which implies
    \begin{equation}
        \tilde{N}(\tau_{\max})\leq  H_{\init}(1-\tau_{\max})\leq \max_{\eta\in[0,1]}H_{\init}(\eta),
    \end{equation} thanks to the positivity of $\left(1-\int_{0}^{\tau_{\max}}c(\tau_{\max},s)ds\right)$ in \eqref{integral-c-formula-init}. This proves the upper bound in \eqref{eq-lem51}, and the proof of the lower bound is similar. 
 \end{proof}
 The second lemma allows us to control $\tilde{N}$ on $[\tau,\tau+1]$ from the bounds on the previous interval $[\tau-1,\tau]$.
\begin{lemma}\label{lem:tau-tau-1}
Let $(Q,\tilde{N})$ be a classical solution to the relaxed problem  \eqref{tilde-eq-classical}-\eqref{tilde-ic-classical}. Suppose additionally $K'\leq 0$ on $\mathbb{R}$.
For $\tau>0$ define
\begin{equation}\label{def-mM}
    m_{\tau}:= \min_{s\in[\tau,\tau+1]}\tilde{N}(\tau),\qquad M_{\tau}:=  \max_{s\in[\tau,\tau+1]}\tilde{N}(\tau) .
\end{equation} Then the following holds for $\tau\geq1$
\begin{equation}
    m_{\tau-1}\leq m_{\tau}\leq M_{\tau}\leq M_{\tau-1}.
\end{equation}
\end{lemma}
\begin{proof}
    Note that $m_{\tau}\leq M_{\tau}$ by definition. It boils down to prove $m_{\tau-1}\leq m_{\tau}$ and $M_{\tau}\leq M_{\tau-1}$. We only give a proof of $m_{\tau-1}\leq m_{\tau}$ here as the other is similar.

    The proof is similar to that of Lemma \ref{lem:init-tildeN}.  Suppose $\tau_{\min}\in[0,1]$ is a minimum point of $\tilde{N}$ on $[\tau,\tau+1]$, i.e. $\tilde{N}(\tau_{\min})=\min_{s\in[\tau,\tau+1]}\tilde{N}(s)=m_{\tau-1}$, whose existence is ensured by the continuity of $\tilde{N}$. Then by \eqref{eq-tildeN-tau>1} in Proposition \ref{prop:eq-tildeN}, we have
    \begin{align}\label{pf-lemma52-tmp}
        m_{\tau}=\tilde{N}(\tau_{\min})&=\left(1-\int_{\tau_{\min}-1}^{\tau_{\min}}c(\tau_{\min},s)ds\right) \tilde{N}(\tau_{\min}-1)+\int_{\tau_{\min}-1}^{\tau_{\min}}c(\tau_{\min},s)\tilde{N}(s)ds.
    \end{align} Note that $\tau_{\min}\in[\tau,\tau+1]$ implies that $\tau_{\min}-1\in[\tau-1,\tau]$, which allows us to deduce $\tilde{N}(\tau_{\min}-1)\geq m_{\tau-1}$ and
    \begin{align}
   \int_{\tau_{\min}-1}^{\tau_{\min}}c(\tau_{\min},s)\tilde{N}(s)ds&=\int_{\tau_{\min}-1}^{\tau}c(\tau_{\min},s)\tilde{N}(s)ds+\int_{\tau}^{\tau_{\min}}c(\tau_{\min},s)\tilde{N}(s)ds\\&\geq m_{\tau-1}\int_{\tau_{\min}-1}^{\tau}c(\tau_{\min},s)ds+m_{\tau}\int_{\tau}^{\tau_{\min}}c(\tau_{\min},s)ds,
    \end{align} where in the last line we use $c(\tau,s)\geq0$ thanks to $K'\leq 0$ by \eqref{def-kernal-c}. Substitute the above equations into \eqref{pf-lemma52-tmp}, and we derive
    \begin{align}
         m_{\tau}\geq \left(1-\int_{\tau_{\min}-1}^{\tau_{\min}}c(\tau_{\min},s)ds\right)m_{\tau-1}+\left(\int_{\tau_{\min}-1}^{\tau}c(\tau_{\min},s)ds\right)m_{\tau-1}+\left(\int_{\tau}^{\tau_{\min}}c(\tau_{\min},s)ds\right)m_{\tau},
    \end{align} which implies
    \begin{equation}
         \left(1-\int_{\tau}^{\tau_{\min}}c(\tau_{\min},s)ds\right)m_{\tau}\geq \left(1-\int_{\tau}^{\tau_{\min}}c(\tau_{\min},s)ds\right)m_{\tau-1}, 
    \end{equation} which proves $m_{\tau}\geq m_{\tau-1}$. Note that we have used the positivity of the following two terms thanks to $c\geq0$ and \eqref{integral-c-formula}
    \begin{equation}
    1-\int_{\tau}^{\tau_{\min}}c(\tau_{\min},s)ds\geq 1-\int_{\tau_{\min}-1}^{\tau_{\min}}c(\tau_{\min},s)ds>0.
    \end{equation}
\end{proof}
\begin{remark}
In Lemma \ref{lem:tau-tau-1}, we control $\tilde{N}$ over the interval $[\tau,\tau+1]$ via its value on $[\tau-1,\tau]$.
While these two consecutive time intervals might correspond to different durations in the original timescale $t$, they are of the same length in $\tau$.
This is a manifestation of the fact that we have normalized the timescale according to the mass flux crossing $\Phi_F$. 
\end{remark}

With Lemmas \ref{lem:init-tildeN}-\ref{lem:tau-tau-1} we can prove Proposition \ref{prop:bd-relaxed-tildeN}.
\begin{proof}[Proof of Proposition \ref{prop:bd-relaxed-tildeN}]
    With $m_{\tau},M_{\tau}$ defined in \eqref{def-mM}, it is equivalent to prove for every $n\in\mathbb{N}_+$,
    \begin{equation}\label{tmp-pf-prop54}
       \min_{\eta\in[0,1]}H_{\init}(\eta)\leq m_{n}\leq M_{n}\leq \max_{\eta\in[0,1]}H_{\init}(\eta).
    \end{equation} Indeed, \eqref{tmp-pf-prop54} is equivalent to  that \eqref{bd-tildeN-H} holds on for $\tau\in[n,n+1]$.

    Lemma \ref{lem:init-tildeN} gives \eqref{tmp-pf-prop54} for $n=1$. Then we can iteratively apply Lemma \ref{lem:tau-tau-1} to derive for each $n\in\mathbb{N}_+$
    \begin{equation}\label{tmp-pf-5-2}
               \min_{\eta\in[0,1]}H_{\init}(\eta)\leq m_{n}\leq m_{n+1}\leq M_{n+1} \leq M_n\leq\max_{\eta\in[0,1]}H_{\init}(\eta).
    \end{equation} Therefore by induction \eqref{tmp-pf-prop54} is proved for every $n\in\mathbb{N}_+$.
\end{proof}

\begin{remark}
    Note that the bounded monotone sequences in \eqref{tmp-pf-5-2} allows us to define limits
    \begin{equation}
                    \min_{\eta\in[0,1]}H_{\init}(\eta)\leq m_{\infty}:=\lim_{n\rightarrow\infty} m_{n}\leq \lim_{n\rightarrow\infty} M_{n}=: M_{\infty}\leq\max_{\eta\in[0,1]}H_{\init}(\eta).
    \end{equation} Moreover the oscillation of $\tilde{N}$ on $[n,n+1]$ is decreasing with respect to $n$. These all indicate some  stable behavior when $K'\leq0$, which is consistent with the convergence to the steady state to be proved later in Theorem \ref{thm:convergence}.
\end{remark}

Now we can prove Theorem \ref{thm:global-solu}.
\begin{proof}[Proof of Theorem \ref{thm:global-solu}]
    Consider the solution to the relaxed problem $(Q,\tilde{N})$ (whose existence is ensured by Theorem \ref{thm:wps-classical-relax} and Assumption \ref{as:wps}), by Proposition \ref{prop:bd-relaxed-tildeN} and Assumption \ref{as:H} we have
    \begin{equation}\label{tmp-pf-thm51}
               0< \min_{\eta\in[0,1]}H_{\init}(\eta)\leq \tilde{N}(\tau) \leq \max_{\eta\in[0,1]}H_{\init}(\eta),\qquad \tau\geq0.
    \end{equation} In particular $\tilde{N}(\tau)>0$ for all $\tau\geq0$. Therefore by Lemma \ref{lem:relation-relaxed-original} we know the solution to the original problem \eqref{eq-Q-tau-classical}-\eqref{ic-q-tau-classical} is global with $N(\tau)=1/\tilde{N}(\tau)$. As a consequence, we derive the bounds \eqref{gloabl-bd-N} from \eqref{tmp-pf-thm51}.

    We note a subtle detail: in Proposition \ref{prop:bd-relaxed-tildeN}, we require $K' \leq 0$ on $\mathbb{R}$, whereas here we assume only $K' \leq 0$ on $[0, \Phi_F]$. However, this discrepancy does not present an issue. Indeed, using Assumption \ref{as:wps} with the bounds in \eqref{tmp-pf-thm51}, we can deduce that even in the context of the relaxed problem, $Q$ only takes values in $[0, \Phi_F]$ (similar to the proof of Proposition \ref{prop:keep-mono}-\ref{prop:char-N-classical}, indeed at each time $Q(\tau,\cdot)$ is an increasing function linking $0$ and $\Phi_F$). Consequently, the behavior of $K$ outside $[0, \Phi_F]$ is irrelevant here.
    \end{proof}

\begin{remark}\label{rmk:K'>0-integral-mM}
We can develop a counterpart for Proposition \ref{prop:bd-relaxed-tildeN} in the case $K'\geq0$. The idea is still to use the integral equation for $\tilde{N}$ in Proposition \ref{prop:eq-tildeN}. It can be shown that when $K'\geq 0$ for $n\in\mathbb{N}_+$ we have
\begin{align}\label{K-positive-mM}
            m_{n+1}\leq m_n\leq \min_{\eta\in[0,1]}H_{\init}(\eta)\leq \max_{\eta\in[0,1]}H_{\init}(\eta)\leq M_n\leq 
            M_{n+1}.
\end{align} Here $m_{\tau}$ and $M_{\tau}$ are defined as in \eqref{def-mM}. The proof of \eqref{K-positive-mM} is similar to Proposition \ref{prop:bd-relaxed-tildeN}. However, in contrast to \eqref{tmp-pf-5-2} for $K'\leq0$, \eqref{K-positive-mM} implies that the difference between the maximum and minimum of $\tilde{N}$ on $[n,n+1]$ is increasing in $n$. This indicates an unstable behavior when $K'>0$, which is consistent with the finite-time blow-up case to be studied in Section \ref{sec:blowup}.
\end{remark}

\subsection{Convergence to the steady state}\label{subsec:steady-state}

We first give a sharp condition about the existence of a steady state. We have the following dichotomy: either there exists a unique steady state to \eqref{eq-Q-tau-classical}-\eqref{ic-q-tau-classical} when $K$ is not very large (precisely \eqref{ss-condition}), or otherwise there is no steady state.

\begin{proposition}[See also {\cite[Proposition 1]{mauroy2012global}}]\label{prop:ss}
There is a dichotomy on the existence of the steady state: If
\begin{equation}\label{ss-condition}
    \int_0^{\Phi_F}\frac{1}{K(\phi)}d\phi> 1,
\end{equation} then there exists a unique steady state to \eqref{eq-Q-tau-classical}-\eqref{constaint-classical-Q-tau}. Otherwise if
\begin{equation}\label{non-ss-condition}
    \int_0^{\Phi_F}\frac{1}{K(\phi)}d\phi\leq 1,
\end{equation}
then there is no steady state.
\end{proposition}
\begin{proof}The steady state $(Q^*,N^*)$ satisfies
\begin{equation}
    \frac{d}{d\eta}Q^*=K(Q^*)+\frac{1}{N^*},\qquad Q(0)=0,\, Q(1)=\Phi_F,
\end{equation} with $N^*>0$. 
We can recast the formulation as to find a $N=N^*>0$ such that the solution to the initial value problem
\begin{equation}\label{def-qN}
    \frac{d}{d\eta}q_N(\eta)=K(q_N(\eta))+\frac{1}{N},\quad q_N(0)=0,
\end{equation} satisfies that the ``first hitting time'' to $\Phi_F$ is $1$, or more precisely
\begin{equation}\label{cond-N}
    \eta_N=1,\qquad \eta_N:=\inf \{\eta>0: q_N(\eta)=\Phi_F\}.
\end{equation} Note that
\begin{equation}
    \Phi_F=q_N(\eta_N)-q_N(0)=\int_0^{\eta_N}(K(q_N)+\frac{1}{N})d\eta\geq \frac{\eta_N}{N},
\end{equation} which gives an upper bound $\eta_N\leq N\Phi_F$ and implies that $\eta_{0^+}:=\lim_{N\rightarrow0^+}\eta_N=0$.

We can check that $\eta_N$ is strictly decreasing w.r.t. $N$ and depends on $N$ continuously, by standard ODE arguments. Then, with $\eta_{0^+}=0$, we see a dichotomy that if $\eta_{\infty}:=\lim_{N\rightarrow0^+}\eta_N>1$, then there is a unique $N>0$ satisfying \eqref{cond-N}, corresponding to a unique steady state; otherwise $\eta_{\infty}\leq1$ and there is no solution, which means that there is no steady state. 

Therefore it remains to calculate $\eta_{\infty}$. It can be verified that $\eta_{\infty}$ is the ``first hitting time'' to $\Phi_F$ for the following $q_\infty(\eta)$ 
\begin{equation}\label{def-qInf}
    \frac{d}{d\eta}q_\infty(\eta)=K(q_\infty(\eta)),\quad q_\infty(0)=0,
\end{equation} which is formally just the case $N=\infty$ in \eqref{def-qN}. Solving \eqref{def-qN} using the change of variable $u(\eta)=\int_0^{q_{\infty}(\eta)}\frac{1}{K(\phi)}d\phi$ gives
\begin{equation}
    \eta_{\infty}= u(\eta_{\infty})=\int_0^{q_\infty(\eta_{\infty})}\frac{1}{K(\phi)}d\phi=\int_0^{\Phi_F}\frac{1}{K(\phi)}d\phi,
\end{equation} which induces the desired result.
\end{proof}

\begin{remark}\label{rmk:implicit-assume-k} 
We point out a connection between condition \eqref{ss-condition} for the steady state and Assumption \ref{as:H} for the global existence.
     Integrating in \eqref{no-blow-up-condition-rho}, which is the equivalence of Assumption \ref{as:H} in $\rho$-formulation, we get
\begin{equation}
1=\int_0^{\Phi_F}\rho_{\init}(\phi)d\phi<\int_0^{\Phi_F}\frac{1}{K(\phi)}d\phi.
\end{equation} This is exactly the condition \eqref{ss-condition} in Proposition \ref{prop:ss}. As a result, the conditions in Theorem \ref{thm:global-solu} imply that there is a unique steady state to \eqref{eq-Q-tau-classical}-\eqref{ic-q-tau-classical}. Conversely, if \eqref{non-ss-condition} holds, then there is no probability distribution $\rho_{\rm init}$ satisfying condition \eqref{no-blow-up-condition-rho}.
\end{remark}

Now we have enough preparations to prove the long time convergence to the steady state.

\begin{theorem}[See also {\cite[Proposition 2 and 3]{mauroy2012global}}]\label{thm:convergence}
    Suppose $k_{\max}<0$, where $k_{\max}$ is given by \eqref{def-kmin_kmax}, and \eqref{ss-condition}.  Suppose the initial data $Q_{\init}$  satisfies Assumption \ref{as:wps}  and Assumption \ref{as:H}.     
    Then there exists a unique global classical solution $(Q,N)$  to  \eqref{eq-Q-tau-classical}-\eqref{ic-q-tau-classical}, which converges to the unique steady state $Q^*$ exponentially in the following sense
    \begin{equation}
         \|\p_{\eta}Q(\tau,\cdot)-\p_{\eta}Q^*(\cdot)\|_{L^1(0,1)}\leq e^{k_{\max}\tau}\|\p_{\eta}Q(0,\cdot)-\p_{\eta}Q^*(\cdot)\|_{L^1(0,1)}.
    \end{equation} 
\end{theorem}
\begin{proof} Note that $k_{\max}<0$ implies that $K'<0$ on $[0,\Phi_F]$ thanks to \eqref{ineq-kminkmax}. By Theorem \ref{thm:global-solu} there is a unique global classical solution $(Q,N)$ with initial data $Q_{\init}$. By Proposition \ref{prop:ss}, there is a unique steady state $Q^*$ due to \eqref{ss-condition}.

    Therefore, we can apply Theorem \ref{thm:BV}  to $Q_1(\tau,\eta):=Q(\tau,\eta)$ and $Q_2(\tau,\eta)=Q^*(\eta)$ to obtain the global convergence.
\end{proof}

\begin{remark}
If additionally Assumption \ref{as:K-cc} holds, then $k_{\max}<0$ is equivalent to $K'<0$ on $[0,\Phi_F]$, see Remark \ref{rmk:Kcc}.
\end{remark}

\begin{remark}
    An analogous result can be obtained using Theorem \ref{thm:L2} instead of Theorem \ref{thm:BV}, with suitable modifications to the setting: $K(Q)=kQ+b$ ($k<0$) and the convergence is in the modified $L^2$ distance \eqref{distance-L2-variant}.
\end{remark}


\section{Blow-up: $K'>0$ or strong interaction}\label{sec:blowup}

In Section \ref{sec:global}, we have shown the global existence and convergence to the steady state, in the regime when $K'<0$ and the size of $K$ is not very large.  In contrast, this section presents finite time blow-up results in the regime when $K'>0$ or the size of $K$ is large.

\begin{remark}\label{rmk:bl-t} Let $\tau^*$ be the maximal existence time in $\tau$ timescale. Then the maximal existence time in the original $t$ timescale is given by $T^*=\int_0^{\tau^*}\frac{1}{N(\tau)}d\tau$ due to \eqref{def-tau-sing}. We note that $\tau^*<+\infty$ directly implies $T^*<+\infty$. More specifically, notice that $\sup_{\tau\in[0,\tau^*)}\frac{1}{N(\tau)}<+\infty$ whenever $\tau^*$ is finite, which follows from \eqref{limit-N-blowup} in Theorem \ref{thm:wps-blow-up-classical}. Hence, we have
\begin{equation}
    T^*=\int_0^{\tau^*}\frac{1}{N(\tau)}d\tau\leq \tau^* \sup_{\tau\in[0,\tau^*)}\frac{1}{N(\tau)}<+\infty,
\end{equation} whenever $\tau^*<+\infty$. Therefore, finite time blow-up in timescale $\tau$ implies finite time blow-up in timescale $t$.
\end{remark}

\subsection{Blow-up 1: $K'>0$}\label{subsec:blow-up1}

We recall from Remark \ref{rmk:Kcc} that $\min_{\phi\in[0,\Phi_F]}K'(\phi)\geq k_{\min}$, thus assuming $k_{\min}>0$ implies that $K'>0$. Moreover, they are equivalent if additionally Assumption \ref{as:K-cc} holds.

\begin{theorem}\label{thm:blow-up-increasingK}
        Suppose $k_{\min}>0$, where $k_{\min}$ is defined in \eqref{def-kmin_kmax}. Then every solution to \eqref{eq-Q-tau-classical}-\eqref{ic-q-tau-classical} blows up in finite time, except if the solution were the unique steady state, when it exists. 
\end{theorem}

\begin{proof}
    Suppose there are \textit{two} global solutions, denoted as $Q_1(\tau,\cdot)$ and $Q_2(\tau,\cdot)$, with \textit{different} initial data, then by Theorem \ref{thm:BV} the following distance grows exponentially in time
    \begin{equation}\label{tmp-bl-1}
        \|\p_{\eta}Q_1(\tau)-\p_{\eta}Q_2(\tau)\|_{L^1(0,1)}\geq e^{k_{\min}\tau}\|\p_{\eta}Q_1(0)-\p_{\eta}Q_2(0)\|_{L^1(0,1)}\rightarrow+\infty,\qquad \text{as $\tau\rightarrow+\infty$}.
    \end{equation} 

    However, as in \eqref{bdd-BV}, we know the distance shall be uniformly bounded as
    \begin{equation}\label{tmp-bl-2}
        \|\p_{\eta}Q_1(\tau)-\p_{\eta}Q_2(\tau)\|_{L^1(0,1)}\leq 2\Phi_F,
    \end{equation} which leads to a contradiction.

    We have shown that there is \textit{at most} one initial data that leads to a global solution, otherwise we contradict the previous statement. Denote it as $Q_1(\tau,\cdot)$. Note that for all $\Delta\tau>0$, $Q_1(\tau+\Delta\tau,\cdot)$ is also a global solution with initial data $Q_1(\Delta\tau,\cdot)$. Hence by previous arguments $Q_1(\Delta\tau,\cdot)$ must coincide with $Q_1(0,\cdot)$, for all $\Delta\tau>0$. This implies that $Q_1$ is the unique steady state.
\end{proof}

\begin{remark}\label{rmk:62} 
Similar blow-up results have appeared in \cite[Proposition 4]{mauroy2012global} under Assumption \ref{as:K-cc}, see Remark \ref{rmk:Kcc}, where the existence of the steady state is needed as an additional requirement. Here in Theorem \ref{thm:blow-up-increasingK} we do not need either an assumption on the steady state or Assumption \ref{as:K-cc}. We just need $k_{\min}>0$ which is stronger than $K'>0$ so there is still some requirement on the second derivatives. Such an improvement is due to that in $\tau$ timescale we can compare two arbitrary solutions (c.f. Theorem \ref{thm:BV}), not necessarily one solution to the steady state (c.f. Corollary \ref{cor:BV}). 
\end{remark}

Using the steady state as a reference, we can also obtain an estimate on the blow-up time.
\begin{corollary}\label{cor:bltime-BV}
Suppose $k_{\min}>0$, where $k_{\min}$ is defined in \eqref{def-kmin_kmax}. Additionally assume \eqref{ss-condition} which ensures that there is a unique steady state $Q^*$. Then for every non-steady initial data $Q_{\init}(\eta)$, the first blow-up time $\tau^*$ to the classical solution of \eqref{eq-Q-tau-classical}-\eqref{ic-q-tau-classical} is bounded by
\begin{equation}
    \tau^*\leq \frac{1}{k_{\min}}\ln\left(\frac{2\Phi_F}{\|\p_{\eta}Q_{\init}(\cdot)-\p_{\eta}Q^*(\cdot)\|_{L^1(0,1)}}\right)<+\infty.
\end{equation}
\end{corollary}
\begin{proof}
    Denote the solution with initial data $Q_{\init}$ as $Q(\tau,\cdot)$. A combination of \eqref{tmp-bl-1}-\eqref{tmp-bl-2} with $Q_1=Q$ and $Q_2=Q^*$ leads to
    \begin{equation}
        e^{k_{\min}\tau}\|\p_{\eta}Q_{\init}(\cdot)-\p_{\eta}Q^*(\cdot)\|_{L^1(0,1)}\leq 2\Phi_F,\qquad \forall\tau< \tau^*,
    \end{equation} which yields the result.
\end{proof}
    
\begin{remark}
Using Theorem \ref{thm:L2}, we can obtain a similar estimate on the blow-up time which involves the modified $L^2$ distance \eqref{distance-L2-variant} in the setting $K(Q)=kQ+b$ ($k>0$).
 \end{remark}

With the steady state criteria in Proposition \ref{prop:ss}, we can refine Theorem \ref{thm:blow-up-increasingK} as
\begin{corollary}\label{cor:refine-blow-up}
            Suppose $k_{\min}>0$, where $k_{\min}$ is defined in \eqref{def-kmin_kmax}.

            (i) Assume \eqref{ss-condition}. Then every solution to \eqref{eq-Q-tau-classical}-\eqref{ic-q-tau-classical} except the unique steady state blows up in finite time. 

            (ii) Otherwise \eqref{non-ss-condition} holds. Then every solution blows up in finite time.
\end{corollary}

 In summary, no matter the size of $K$, as long as $k_{\min}>0$, then every non-steady solution blows up. When $K$ is small \eqref{ss-condition}, the unique steady state is unstable. And when $K$ is large \eqref{non-ss-condition}, there is no steady state and therefore every solution blows up.

Indeed,  next we will show in Section \ref{subsec:blow-up2} that if \eqref{non-ss-condition} holds, then every solution blows up no matter the sign of $K'$. This in particular implies that the second part of Corollary \ref{cor:refine-blow-up} holds, even without the assumption $k_{\min}>0$.

\subsection{Blow-up 2: strong interaction}\label{subsec:blow-up2}

\subsubsection{Proof via the characteristics}
\begin{theorem}\label{thm:bl-char}
    Assume \eqref{non-ss-condition}, then  every classical solution to \eqref{eq-Q-tau-classical}-\eqref{ic-q-tau-classical} blows up in finite time. Furthermore, we have the following upper bound on its maximal existence time $\tau^*$ 
\begin{equation}\label{est_bltime_char}
        \tau^*\leq     \int_0^{\Phi_F}\frac{1}{K(\phi)}d\phi\leq 1<\infty.
    \end{equation} 
\end{theorem}
\begin{proof}
We recall the following auxiliary equation for $q_{\infty}$ \eqref{def-qInf}
\begin{equation}
    \frac{d}{d\eta}q_\infty(\eta)=K(q_\infty(\eta)),\quad q_\infty(0)=0.
\end{equation} In the proof of Proposition \ref{prop:ss}, we have shown that
\begin{equation}
    q_\infty(\eta_{\infty})=\Phi_F,\qquad   0<\eta_{\infty}=\int_0^{\Phi_F}\frac{1}{K(\phi)}d\phi\leq 1,
\end{equation} where the last inequality is \eqref{non-ss-condition}.

Now we argue by contradiction to show the estimate on the existence time \eqref{est_bltime_char}. Suppose the solution exists at least up to $\tau=\eta_{\infty}$. We consider $q(\eta):=Q(\eta,\eta)$ along the characteristic of $\p_{\tau}+\p_{\eta}$ in \eqref{eq-Q-tau-classical}. Then we have
\begin{align}
        \frac{d}{d\eta}q(\eta)&=K(Q(\eta,\eta))+\frac{1}{N(\eta)}\\&=K(q(\eta))+\frac{1}{N(\eta)}>K(q(\eta)),\qquad \eta\in[0,\eta_{\infty}],
\end{align} with $q(0)=0=q_{\infty}(0)$. Therefore, we can apply the comparison principle for one-dimensional ODE to compare $q$ and $q_{\infty}$ and derive at $\eta_{\infty}$
\begin{align}
Q(\eta_{\infty},\eta_{\infty})=q(\eta_{\infty})>q_{\infty}(\eta_{\infty})=\Phi_F.
\end{align} This is a contradiction since for a classical solution we shall have $ Q(\eta_{\infty},\eta_{\infty})\leq Q(\eta_{\infty},1)=\Phi_F$.
    
\end{proof}

Recall that by Proposition \ref{prop:ss}, \eqref{non-ss-condition} is the sharp criteria for non-existence of steady states. As steady states are naturally global solutions, Corollary \ref{thm:bl-char} implies a sharp dichotomy.
\begin{corollary}
    If there is no steady state, then every classical solution to \eqref{eq-Q-tau-classical}-\eqref{ic-q-tau-classical} blows up in finite time.
\end{corollary}

\subsubsection{Proofs via moments}

Now we approach the question of blow-up via a different method, namely calculating the moments
\begin{equation}
    \int_0^{1}m(Q(\tau,\eta))d\eta,
\end{equation} for some function $m(\phi)$. Note that $\int_0^{1}m(Q)d\eta$ corresponds to $\int_0^{\Phi_F}m(\phi)\rho(t,\phi)d\phi$ in the $\rho$ formulation in $t$ timescale \eqref{eq:rho-t-1}-\eqref{ic-rho-t}. Our calculations are based on the following lemma. Using similar quantities to discuss blow-up in computational neuroscience models have been used in \cite{caceres2011analysis,roux2021towards}.

\begin{lemma}\label{lem:moment}
    Let $(Q,N)$ be a classical solution to \eqref{eq-Q-tau-classical}-\eqref{ic-q-tau-classical}. Then for $m(\phi)\in C^1[0,\Phi_F]$ with $m(0)=0$, we have
    \begin{equation}\label{dm-dtau}
    \frac{d}{d\tau} \int_0^{1}m(Q)d\eta= \left(\int_0^{1}m'(Q)K(Q)d\eta-m(\Phi_F)\right)+\frac{1}{N(\tau)}\int_0^{1}m'(Q)d\eta.
\end{equation}
\end{lemma}
\begin{proof} Multiplying \eqref{eq-Q-tau-classical} by $m'(Q)$, we derive
    \begin{equation}\label{m(Q)-eq}
    \p_{\tau}m(Q)+\p_{\eta}m(Q)=m'(Q)K(Q)+\frac{1}{N(\tau)}m'(Q),
\end{equation} Integrating \eqref{m(Q)-eq} in $\eta$, we obtain
\begin{equation}
        \frac{d}{d\tau} \int_0^{1}m(Q)d\eta+m(\Phi_F)-m(0)= \int_0^{1}m'(Q)K(Q)d\eta+\frac{1}{N(\tau)}\int_0^{1}m'(Q)d\eta,
\end{equation} which simplifies to \eqref{dm-dtau} as $m(0)=0$.
\end{proof}
\begin{remark}\label{rmk:moment-rho}
In the $\rho$ formulation in $t$ timescale \eqref{eq:rho-t-1}-\eqref{ic-rho-t}, the formula \eqref{dm-dtau} corresponds to
\begin{equation}
          \frac{d}{dt}\int_0^{\Phi_F}m(\phi)\rho(t,\phi)d\phi=N(t)\left(\int_0^{\Phi_F}m'(\phi)K(\phi)\rho(t,\phi)d\phi-m(\Phi_F)\right)+\int_0^{\Phi_F}m'(\phi)\rho(t,\phi)d\phi.
\end{equation}
\end{remark}

By choosing appropriate $m(\phi)$ in Lemma \ref{lem:moment}, we can prove the following estimates.
\begin{theorem}\label{thm:bl-moment}Assume \eqref{non-ss-condition}. Let $(Q,N)$ be a classical solution to \eqref{eq-Q-tau-classical}-\eqref{ic-q-tau-classical}, and denote $\tau^*$ as its maximal existence time.      Then we have 
     \begin{equation}\label{moment-1}
         \left(\min_{\phi\in[0,\Phi_F]}K(\phi)-\Phi_F\right)\tau^*+\int_0^{\tau^*}\frac{1}{N(\tau)}d\tau \leq \Phi_F-\int_0^{1}Q_{\init}(\eta)d\eta,
     \end{equation}
     and 
     \begin{equation}\label{moment-2}
     \left(1-\int_0^{\Phi_F}\frac{1}{K(\phi)}d\phi\right)\tau^*+\frac{1}{\max_{\phi\in[0,\Phi_F]}K(\phi)}\int_0^{\tau^*}\frac{1}{N(\tau)}d\tau \leq \int_0^{\Phi_F}\frac{1}{K(\phi)}d\phi-\int_0^{1}\left(\int_0^{Q_{\init}(\eta)}\frac{1}{K(\phi)}d\phi\right)d\eta.
     \end{equation}
\end{theorem}
\begin{proof}Under \eqref{non-ss-condition},  by Theorem \ref{thm:bl-char} we already know $\tau^* \leq 1<+\infty$. Here our goal is to provide additional estimates \eqref{moment-1}-\eqref{moment-2}. For \eqref{moment-1} we choose $m(\phi)=\phi$ in Lemma \ref{lem:moment} to derive
\begin{align*}
    \frac{d}{d\tau}\int_0^{1}Qd\eta&= \left(\int_0^{1}K(Q)d\eta-\Phi_F\right)+\frac{1}{N(\tau)}\\&\geq \left(\min_{\phi\in[0,\Phi_F]}K(\phi)-\Phi_F\right)+\frac{1}{N(\tau)}.
\end{align*} Integrating in time on $(0,\tau)$ with $\tau<\tau^*$, we have
\begin{align*}
    \left(\min_{\phi\in[0,\Phi_F]}K(\phi)-\Phi_F\right)\tau+\int_0^{\tau}\frac{1}{N(\tilde{\tau})}d\tilde{\tau}&\leq \int_0^{1}Q(\tau,\eta)d\eta-\int_0^{1}Q_{\init}(\eta)d\eta\\&\leq \Phi_F-\int_0^{1}Q_{\init}(\eta)d\eta.
\end{align*} Taking the limit $\tau\rightarrow(\tau^*)^-$ gives \eqref{moment-1}.

For \eqref{moment-2}, we choose more sophisticated $m(\phi)=\int_0^{\phi}\frac{1}{K(\tilde{\phi})}d\tilde{\phi}$ in Lemma \ref{lem:moment}, which satisfies $m'(\phi)=\frac{1}{K(\phi)}$. Then we derive
\begin{align*}
        \frac{d}{d\tau}\int_0^{1}m(Q)d\eta&= \left(1-m(\Phi_F)\right)+\frac{1}{N(\tau)}\int_0^{1}\frac{1}{K(Q)}d\eta\\&\geq  \left(1-m(\Phi_F)\right)+ \frac{1}{\max_{\phi\in[0,\Phi_F]}K(\phi)}\frac{1}{N(\tau)}.
\end{align*}Integrating in time on $(0,\tau)$ with $\tau<\tau^*$, we get
    \begin{align*}
\left(1-m(\Phi_F)\right)\tau+ \frac{1}{\max_{\phi\in[0,\Phi_F]}K(\phi)}\int_0^{\tau}\frac{1}{N(\tilde{\tau})}d\tilde{\tau}&\leq \int_0^{1}m(Q(\tau,\eta))d\eta-\int_0^{1}m(Q_{\init}(\eta))d\eta\\&\leq m(\Phi_F)-\int_0^{1}m(Q_{\init}(\eta))d\eta.
\end{align*}Taking the limit $\tau\rightarrow(\tau^*)^-$ and recalling the choice here $m(\phi)=\int_0^{\phi}\frac{1}{K(\tilde{\phi})}d\tilde{\phi}$, we obtain \eqref{moment-2}.
\end{proof}

The term $\int_0^{\tau^*}\frac{1}{N(\tau)}d\tau$ in the theorem above is the maximal existence time in $t$ timescale, as denoted by $T^*$ in Remark \ref{rmk:bl-t}. We conclude with the following remarks.

\begin{remark}
If $\min_{\phi\in[0,\Phi_F]}K(\phi)>\Phi_F$, then \eqref{moment-1} implies an upper bound on $\tau^*$. Note however this pointwise condition is stronger than the integral condition $\int_0^{\Phi_F}\frac{1}{K(\phi)}d\phi>1$. The latter can be understood as the harmonic average of $K$ on $[0,\Phi_F]$ is greater than $\Phi_F$, i.e.
    \begin{equation}
        \frac{\Phi_F}{\int_0^{\Phi_F}\frac{1}{K(\phi)}d\phi}> \Phi_F.
    \end{equation}
\end{remark}

\begin{remark}
The right hand side of \eqref{moment-1} corresponds to $\int_0^{\Phi_F}(\Phi_F-\phi)\rho_{\init}(\phi)d\phi$ in $\rho$ formulation. Hence \eqref{moment-1} may be interpreted as that a higher concentration of mass near $\Phi_F$ results in an earlier blow-up. {With the high connectivity assumption $\min_{\phi\in[0,\Phi_F]}K(\phi)>\Phi_F$, we have the following simple bound in original timescale
\[ T^*=\int_0^{\tau^*}\frac{1}{N(\tau)}d\tau \leqslant \int_0^{\Phi_F}(\Phi_F-\phi)\rho_{\init}(\phi)d\phi. \]}  
\end{remark}

\section*{Acknowledgements}

JAC and PR was supported by the Advanced Grant Nonlocal-CPD (Nonlocal PDEs for Complex Particle Dynamics: Phase Transitions, Patterns and Synchronization) of the European Research Council Executive Agency (ERC) under the European Union's Horizon 2020 research and innovation programme (grant agreement No. 883363). JAC was also partially supported by EPSRC grants EP/T022132/1 and EP/V051121/1. XD wants to thank Steven Strogatz for inspiring discussions on pulse-coupled oscillators. ZZ is supported by the National Key R\&D Program of China, Project Number 2021YFA1001200, and the NSFC, grant Number 12031013, 12171013. We also thank Beno\^ \i t Perthame for Remark \ref{rmk:pf-rw-r}.

\bibliographystyle{abbrv}

\end{document}